\documentclass[%
reprint,
superscriptaddress,
showkeys,
nofootinbib,
amsmath,amssymb,
aps,
]{revtex4-2}

\usepackage{graphicx}
\usepackage{dcolumn}
\usepackage{amsfonts,amsmath,amssymb,amsthm,bbm,mathrsfs,mathtools}
\usepackage{defs,fonts}
\usepackage{caption,booktabs,float}
\usepackage{subcaption}
\usepackage[colorlinks=true, allcolors=blue]{hyperref}
\usepackage{makecell}
\usepackage{appendix}
\usepackage{orcidlink}

\newcommand{\vect}[1]{\boldsymbol{#1}}

\newtheorem{theorem}{Theorem}
\newtheorem{corollary}{Corollary}
\newtheorem{prop}[theorem]{Proposition}
\newtheorem{lemma}{Lemma}
\newtheorem{rem}{Remark}

\newcommand{\response}[1]{#1}

\newcommand{\Lo}{\mbox{\tiny L}}
\newcommand{\Hi}{\mbox{\tiny H}}
\newcommand{\MaxEnt}{\mbox{\tiny ME}}
\newcommand{\bcC}{\boldsymbol{\cC}}
\newcommand{\Equilibrium}{\mbox{\tiny Eq}}

\begin{document}

\preprint{APS/123-QED}

\title{Realizability-Preserving Discontinuous Galerkin Method for \\
Spectral Two-Moment Radiation Transport in Special Relativity}

\author{Joseph Hunter\,\orcidlink{0009-0004-7188-1172}}
\email{hunter.926@buckeyemail.osu.edu}
\affiliation{Department of Mathematics, Ohio State University, Columbus, OH, 43210, USA}

\author{Eirik Endeve\,\orcidlink{0000-0003-1251-9507}}
\email{endevee@ornl.gov}
\affiliation{Computer Science and Mathematics Division, Oak Ridge National Laboratory, Oak Ridge, TN 37831, USA}
\thanks{Notice:  This manuscript has been authored, in part, by UT-Battelle, LLC, under contract DE-AC05-00OR22725 with the US Department of Energy (DOE). The US government retains and the publisher, by accepting the article for publication, acknowledges that the US government retains a nonexclusive, paid-up, irrevocable, worldwide license to publish or reproduce the published form of this manuscript, or allow others to do so, for US government purposes. DOE will provide public access to these results of federally sponsored research in accordance with the DOE Public Access Plan (http://energy.gov/downloads/doe-public-access-plan).}
\affiliation{Department of Physics \& Astronomy, University of Tennessee, Knoxville, TN 37996-1200, USA}

\author{M.~Paul Laiu\,\orcidlink{0000-0002-2215-6968}}
\email{laiump@ornl.gov}
\affiliation{Computer Science and Mathematics Division, Oak Ridge National Laboratory, Oak Ridge, TN 37831, USA}

\author{Yulong Xing\,\orcidlink{0000-0002-3504-6194}}
\email{xing.205@osu.edu}
\affiliation{Department of Mathematics, Ohio State University, Columbus, OH, 43210, USA}

\date{\today}

\begin{abstract}
    We present a realizability-preserving numerical method for solving a spectral two-moment model to simulate the transport of massless, neutral particles interacting with a \response{steady} background material moving with relativistic velocities.
    The model is obtained as the special relativistic limit of a four-momentum-conservative general relativistic two-moment model.
    Using a maximum-entropy closure, we solve for the Eulerian-frame energy and momentum.
    The proposed numerical method is designed to preserve moment realizability, which corresponds to moments defined by a nonnegative phase-space density.
    The realizability-preserving method is achieved with the following key components:
    (i) a discontinuous Galerkin (DG) phase-space discretization with specially constructed numerical fluxes in the spatial and energy dimensions;
    (ii) a strong stability-preserving implicit-explicit (IMEX) time-integration method;
    (iii) a realizability-preserving conserved to primitive moment solver;
    (iv) a realizability-preserving implicit collision solver;
    and (v) a realizability-enforcing limiter.
    Component (iii) is necessitated by the closure procedure, which closes higher order moments nonlinearly in terms of primitive moments.
    The nonlinear conserved to primitive and the implicit collision solves are formulated as fixed-point problems, which are solved with custom iterative solvers designed to preserve the realizability of each iterate.  
    With a series of numerical tests, we demonstrate the accuracy and robustness of this DG-IMEX method.  
\end{abstract}

\maketitle

\section{Introduction}

In this paper, we design and analyze a numerical method for solving a spectral two-moment model to simulate the transport of massless, neutral particles (e.g., photons or classical neutrinos) interacting with a \response{steady (i.e., time independent)} background material moving with relativistic velocities.
The proposed method uses the discontinuous Galerkin (DG) method for phase-space discretization and implicit-explicit (IMEX) time-stepping.  
In particular, the fully discrete scheme is designed to maintain certain physical bounds for the evolved angular moments, generated from a nonnegative distribution function.  
These bounds are respected numerically, through careful consideration of the phase-space and temporal discretizations, and the formulation of iterative nonlinear solvers, which are part of the solution process.

Neutral particle transport is an important aspect of modelling many relativistic astrophysical systems, including core-collapse supernovae (CCSNe) \cite{mezzacappa_etal_2020}, neutron star-neutron star or neutron star-black hole mergers \cite{metzger_2019,foucart_2023}, and the dynamics of accretion flows onto black holes \cite{abramowiczFragile_2013}.  
The study of transport processes in these systems is, in part, complicated by the need for kinetic models, seeking the distribution function $f(\vect{p},\vect{x},t)$ --- a phase-space density, which at time $t$ gives the number of particles in an infinitesimal phase-space volume element centered about the phase-space coordinates $\{\vect{p},\vect{x}\}$.
Here, $\vect{p}$ and $\vect{x}$ are momentum-space and position-space coordinates, respectively.  
The evolution of $f$ is governed by a kinetic equation, specifically the Boltzmann equation, which expresses a balance between phase-space advection and collisions.

Solving for $f$ numerically, to study the aforementioned astrophysical systems in full dimensionality (six phase-space dimensions plus time), with high phase-space resolution, is computationally expensive, if not infeasible with present-day computational resources, although full-dimensional solvers, e.g., using Monte Carlo methods \cite{miller_etal_2019,Foucart_2021}, or methods based on spherical harmonics \cite{radice_etal_2013}, discrete ordinates \cite{sumiyoshiYamada_2012,nagakura_etal_2014}, finite-elements \cite{maitrayaRadice_2023}, and finite-volumes \cite{white_etal_2023} to discretize the angular dimensions of momentum space, have been proposed.  
To reduce cost, a common approach is to solve for a few angular moments to capture the lowest-order directional features of the distribution.  
To this end, spherical-polar momentum-space coordinates $\{\varepsilon,\vartheta,\varphi\}$ are introduced, and $f$ is integrated against angular basis functions (depending on momentum-space angles $\omega=\{\vartheta,\varphi\})$ to obtain angular moments, which depend only on $\vect{x}$, $t$, and $\varepsilon$, where $\varepsilon$ is the particle energy.  
In this paper, we consider a nonlinear two-moment model obtained as the special relativistic limit of the 3+1 general relativistic framework in \cite{shibata_etal_2011,cardall_etal_2013a}, where only the zeroth and first moments, representing particle energy and momentum densities, respectively, are evolved.  

Due to the appearance of higher-order moments in the truncated equation hierarchy, obtained after taking moments of the Boltzmann equation, the model is not closed, and a closure procedure is required to express higher-order moments in terms of the lower-order evolved moments.
To close the model, we consider the maximum-entropy closure proposed by Minerbo \cite{minerbo_1978} for Maxwell--Boltzmann statistics --- the low-occupation number approximation to maximum-entropy closures of particle systems obeying Bose--Einstein or Fermi--Dirac statistics \cite{cernohorskyBludman_1994}.  
Specifically, for our numerical experiments, we use an algebraic approximation to the Minerbo closure (e.g., \cite{just_etal_2015}).  

The design of a numerical method to model the transport of particles interacting with a moving fluid is complicated by the need to choose appropriate coordinates for the momentum-space discretization.  
Two obvious choices are Eulerian-frame and comoving-frame momentum coordinates (see, e.g., \cite{munierWeaver_1986a,munierWeaver_1986b}, for detailed discussions).  
In this paper, as in, e.g., \cite{shibata_etal_2011,cardall_etal_2013b}, we use comoving-frame momentum-space coordinates; that is, momentum coordinates with respect to an inertial frame instantaneously comoving with the fluid \cite{mihalasMihalas_1999}.  
The angular moments are thus functions of Eulerian-frame spacetime coordinates, $\{\vect{x},t\}$ and the particle energy measured by a comoving observer, $\varepsilon$.  
On the one hand, this choice simplifies the treatment of particle-fluid interactions (because material properties are more isotropic) and the moment closure procedure (because the distribution function is isotropic when particles are in equilibrium with matter).  
On the other hand, the left-hand side of the moment equations, modeling phase-space advection, is made more complicated by the appearance of velocity-dependent terms (so-called observer corrections).  

While the \emph{primitive} variables of the two-moment model are components of a Lagrangian decomposition of the particle stress-energy tensor, we employ the four-momentum-conservative formulation of \cite{shibata_etal_2011,cardall_etal_2013a}, where the evolution equations evolve components of an Eulerian decomposition of the stress-energy tensor, and, when integrated over the particle energy dimension, express conservation of Eulerian-frame energy and three-momentum.  
The Eulerian-frame (conserved) quantities can be expressed in terms of their comoving-frame (primitive) counterparts \cite{cardall_etal_2013a,mezzacappa_etal_2020} to, e.g., evaluate closure relations and fluxes.  
The mixture of Eulerian- and comoving-frame quantities necessitates a moment conversion process (similar to what is needed in relativistic magnetohydrodynamics \cite{kastaun_etal_2021}).  
To this end, extending prior work on a similar model in the $\cO{(v/c)}$ limit \cite{laiu_etal_2025}, we propose an iterative scheme to solve for comoving-frame quantities from the evolved Eulerian-frame quantities.  

Several numerical methods for general relativistic two-moment (neutrino or photon) transport have been proposed to date.  
These evolve either grey moments, where the energy dependence has been integrated out, or spectral moments, as we consider here.  
For example, grey moment solvers have been developed for application to neutron star merger simulations \cite{foucart_2016,radice_etal_2022}; for CCSN simulations \cite{Kuroda_2012}; and for simulation of black hole accretion disks \cite{Fragile_2012,Fragile_2014}.
For application to CCSN, a general relativistic spectral two-moment solver was proposed by \cite{Kuroda_2016}, and in the context of conformally flat spacetimes by \cite{M_ller_2010,cheong_etal_2023}, and for spherically symmetric spacetimes in \cite{O_Connor_2010}.  
We also mention the general relativistic spectral two-moment solver proposed by \cite{Anninos_2020}.  
These schemes --- largely based on finite-volume or finite-difference methods, in combination with explicit and implicit time stepping --- have been coupled to hydrodynamics solvers and used extensively to model astrophysical phenomena.  

\response{
Here, we apply the DG method (e.g., \cite{cockburnShu_2001}) to discretize the moment equations because of their ability to capture the asymptotic diffusion limit \cite{larsenMorel_1989,adams_2001,lowrieMorel_2002,guermondKanschat_2010}, which is characterized by frequent scattering off the background \cite{larsen_etal_1987}.  
In contrast, finite-volume and finite-difference methods may have difficulties capturing this limit, unless, e.g., the particle mean free path is resolved by the spatial mesh, the numerical flux is modified in the diffusive regime \cite{audit_etal_2002}, or additional degrees of freedom are evolved \cite{ren_etal_2022}.%
\footnote{
\response{When the \emph{equilibrium diffusion limit} \cite{ferguson_etal_2017} is considered (in which scattering is a subdominant opacity, distinct from the scattering dominated diffusion limit referred to in this work), finite-volume spatial discretization with IMEX time-stepping may be asymptotic-preserving \cite{he_etal_2024}, with the caveat of potential numerical instability in this limit due to odd-even decoupling \cite{radice_etal_2022,he_etal_2024}. 
}}
}  
Resolving the mean free path is computationally inefficient for the applications we target.  
Furthermore, with the modified flux, it is difficult to design a provably realizability-preserving method for the two-moment model, for which we take advantage of the flexibility offered by the DG method to more freely specify the numerical fluxes.  
Similar to the finite-volume method, the DG method is locally conservative and is able to capture discontinuities.  
There have been extensive studies of structure-preserving DG methods in recent years. These methods preserve exactly, at the discrete level, the continuum properties of the underlying physical models, such as positivity, maximum principle, and asymptotic limits. Some recent studies on structure-preserving DG methods can be found in \cite{zhangShu_2010a,hu_etal_2018,ZhangXingEndeve,chu_etal_2019} and the survey paper \cite{shu2018}.

To evolve the semi-discrete DG scheme, we use IMEX time-stepping methods \cite{ascher_etal_1997,pareschiRusso_2005}, where we treat the phase-space advection terms explicitly and the collision term implicitly.  
We opt for IMEX methods because treating collisions explicitly results in a highly restrictive time-step condition for stability, while the explicit phase-space advection update is governed by a Courant--Friedrichs--Lewy (CFL) time-step restriction depending on the speed of light, which is not too different from both the sound speed and flow velocity in relativistic systems.  
The added cost of the implicit treatment of collisions is that a nonlinear system of equations (akin to a backward Euler step) must be solved multiple times per time step, depending on the number of implicit stages of the IMEX scheme.  
We point out that collisions are local to each spatial point, which makes the implicit part embarrassingly parallel.  

Under Maxwell--Boltzmann statistics, the particle distribution function $f$ must be nonnegative, which leads to constraints that the moments must satisfy.  
Specifically, the energy density must be positive, and the norm of the momentum density vector must be bounded by the energy density.  
Moments satisfying these constraints are said to be realizable.  
Solving the two-moment model numerically can result in nonrealizable moments, which is unphysical, makes the closure procedure ill-posed, and can result in code crashes.  
Therefore it is desirable to design numerical methods which maintain moment realizability.  
Some realizability-preserving two-moment schemes have been developed in \cite{olbrant_etal_2012,chu_etal_2019,laiu_etal_2025}.

Our proposed method for the special relativistic two-moment model is designed to preserve moment realizability.  
We use strong stability-preserving (SSP) IMEX schemes, which allow us to write the explicit update as a convex combination of forward Euler steps.  
For the DG discretization, we propose numerical phase-space fluxes which, together with the SSP IMEX scheme, maintain moment realizability of the cell-averaged moments in the explicit update under a CFL-type time-step restriction.  
After each stage of the IMEX scheme, the realizability-enforcing limiter proposed in \cite{chu_etal_2019} is applied to ensure realizability pointwise in each element. 
As the moment closure procedure requires a conversion process from conserved to primitive moments, we formulate the conversion problem as a fixed-point problem analogous to the modified Richardson iteration, where each iteration preserves realizability.  
Finally, each implicit update of the IMEX scheme can be formulated as a backward Euler step.
The resulting nonlinear system is an extension of the nonlinear system for the moment conversion process, and, with minor modifications, the iterative scheme for the implicit update is formulated to preserve realizability.
With each of these components, we prove that our proposed DG-IMEX scheme for the special relativistic two-moment model is realizability-preserving.  

The method proposed here can be considered as an extension (to special relativity) of the realizability-preserving methods in \cite{chu_etal_2019} and \cite{laiu_etal_2025}, which are, respectively, non-relativistic with focus on aspects of Fermi--Dirac statistics, or include special relativistic corrections to $\cO(v/c)$.  
An important distinction from \cite{laiu_etal_2025} is that we here consider the special relativistic limit of the four-momentum-conservative general relativistic model from \cite{cardall_etal_2013a}, while \cite{laiu_etal_2025} considered the $\cO(v/c)$ limit of the \emph{number-conservative} general relativistic two-moment model (see, e.g., Section~4.7.3 in \cite{mezzacappa_etal_2020}).  
While the relativistic four-momentum-conservative and number-conservative two-moment models are analytically equivalent at the continuum level, this is not the case for their respective $\cO(v/c)$ limits, which differ to $\cO(v^{2}/c^{2})$ \cite{laiu_etal_2025}.  
Moreover, the two formulations may possess different characteristics of relevance to their respective discrete representations.  
One benefit of the four-momentum-conservative model is that, in the absence of collisions, both the zeroth and first moment equations are in conservative form, while this is true only for the zeroth moment equation of the number-conservative model.  
In \cite{laiu_etal_2025}, the appearance of non-collisional source terms in the first moment equation introduced obstacles to proving the realizability-preserving property of the fully discrete multidimensional scheme (see \cite[Section~5.1.3]{laiu_etal_2025}).  
The extension of the scheme proposed in \cite{laiu_etal_2025} for the $\cO(v/c)$ number-conservative two-moment model to special relativity is also faced with these obstacles.  
Since the realizability-preserving property is closely linked to numerical stability and conservation properties of two-moment solvers, one of the main objectives of this study is to consider this important issue in the context of the four-momentum-conservative two-moment model in special relativity.  
The successful outcome in this context presents an important step towards robust methods for general relativistic two-moment models.  
We also note that the numerical study of CCSNe in \cite{roberts_etal_2016} was performed \emph{without} velocity-dependent terms due to observed numerical instabilities when velocity dependence was included.  
Although the origin of these instabilities is unknown to us, their link to velocity-dependent terms reveals additional challenges associated with their inclusion.  

The paper is organized as follows.  
The special relativistic, four-momentum-conservative two-moment model is presented in Section~\ref{sec:model}.  
The DG-IMEX scheme is presented in Section~\ref{sec:method}.  
Section~\ref{sec:iterativeSolvers} presents the iterative solvers used for the moment conversion process and the implicit collision solver.  
Section~\ref{sec:RealizabilityProof} is devoted to the analysis of the DG-IMEX scheme and proving its realizability-preserving property.  
Results from numerical experiments demonstrating the performance of the proposed method are presented in Section~\ref{sec:numericalResults}. 
Summary and conclusions are given in Section~\ref{sec:summaryConclusions}.  
We include some additional technical results in Appendices:  Appendix~\ref{appendix:q_bound} provides an estimate used in the dissipation term of the numerical flux in the energy dimension, while in  Appendix~\ref{appendix:convergence} we prove conditional convergence of the conserved to primitive solver.  

For the remainder of the paper, we adopt units where the speed of light is unity ($c=1$).  
We also use Einstein's summation convention, where repeated Greek indices imply summation from $0$ to $3$, and repeated Latin indices imply summation from $1$ to $3$.  

\section{Mathematical Model}
\label{sec:model}

We consider a special relativistic two-moment model in Cartesian spatial coordinates, where the evolution of the spectral radiation four-momentum is governed by \cite{shibata_etal_2011,cardall_etal_2013a}
\begin{equation}
  \p_{\nu}\cT^{\mu\nu}
  -\frac{1}{\varepsilon^{2}}\p_{\varepsilon}\big(\,\varepsilon^{2}\,\cQ^{\mu\nu\rho}\,\p_{\nu}u_{\rho}\,\big)
  =\frac{1}{4\pi}\int_{\mathbb{S}^{2}}\cC[f]\,p^{\mu}\,\frac{d\omega}{\varepsilon}.
  \label{eq:fourMomentumConservation}
\end{equation}
The angular moments corresponding to the energy-momentum and heat flux tensors are
\begin{subequations}
    \label{eq:angularMoments}
    \begin{align}
      \cT^{\mu\nu} 
      &= \frac{1}{4\pi}\int_{\bbS^{2}}f\,p^{\mu}p^{\nu}\,\frac{d\omega}{\varepsilon},\\
      \cQ^{\mu\nu\rho}
      &= \frac{1}{4\pi}\int_{\bbS^{2}}f\,p^{\mu}p^{\nu}p^{\rho}\,\frac{d\omega}{\varepsilon},
    \end{align}
\end{subequations}
respectively.  
In Eq.~\eqref{eq:angularMoments}, $f\colon (\omega,\varepsilon,\vect{x},t)\in\bbS^{2}\times\bbR^{+}\times\bbR^{3}\times\bbR^{+}\to\bbR^{+}$ is the particle distribution function, which gives the number of particles in an infinitesimal phase-space volume element propagating in the direction $\omega\in\bbS^{2}$, with energy $\varepsilon\in\bbR^{+}$, at position $\vect{x}\in\bbR^{3}$ and time $t\in\bbR^{+}$, $p^{\mu}$ is the particle four-momentum, $d\omega=\sin\vartheta d\vartheta d\varphi$, and the integrals extend over the sphere
\begin{equation}
  \mathbb{S}^{2} = \big\{\,\omega\in(\vartheta,\varphi)~|~\vartheta\in[0,\pi],\,\varphi\in[0,2\pi)\,\big\},
  \label{eq:twoSphere}
\end{equation}
where $\vartheta$ and $\varphi$ are momentum-space angular coordinates.  
Through the collision term $\mathcal{C}[f]$ on the right-hand side of Eq.~\eqref{eq:fourMomentumConservation}, the particles interact with a fluid whose four-velocity is $u^{\mu}$, and $\{\varepsilon,\vartheta,\varphi\}$ are spherical-polar momentum space coordinates with respect to an orthonormal reference frame comoving with the fluid.  
Thus, the angular moments in Eq.~\eqref{eq:angularMoments} are functions of Eulerian-frame spacetime coordinates, $\{\vect{x},t\}$, and comoving-frame particle energy, $\varepsilon$. 
Eq.~\eqref{eq:fourMomentumConservation} corresponds to the special relativistic (i.e., flat spacetime) limit of Eq.~(3.18) in \cite{shibata_etal_2011} and Eq.~(40) in \cite{cardall_etal_2013a}, assuming Cartesian spatial coordinates, where covariant spacetime derivatives are replaced with partial derivatives (i.e., $\nabla_{\nu}\to\partial_{\nu}$).  

For simplicity, in this paper we write the collision term on the right-hand side of Eq.~\eqref{eq:fourMomentumConservation} as
\begin{equation}
    \frac{1}{\varepsilon}\,\cC[f]
    =\eta-\chi\,f + \sigma\,\big(\,\frac{1}{4\pi}\int_{\bbS^{2}}fd\omega-f\,\big),
    \label{eq:collisionTerm}
\end{equation}
where $\eta$ and $\chi$ are the emissivity and absorption opacity, respectively, while $\sigma$ is the scattering rate for isotropic and elastic scattering.  
Here, $\eta$, $\chi$, and $\sigma$ may depend on $\varepsilon$, but are assumed to be independent of $\omega$.  
We also define the total opacity $\kappa=\chi+\sigma$ and (for $\chi>0$) the equilibrium density $\cJ_{\Equilibrium}=\varepsilon\,(\eta/\chi)$.

The fluid four-velocity, the four-velocity of a Lagrangian/comoving observer, has the \emph{Eulerian decomposition}
\begin{equation}
  u^{\mu} = W\,\big(\,n^{\mu}+v^{\mu}\big)=W\big(1,v^{i}\,\big),
  \label{eq:fluidFourVelocity}
\end{equation}
where 
\begin{equation}
    n^{\mu}=(1,0,0,0)
    \label{eq:eulerianFourVelocity}
\end{equation}
is the four-velocity of an Eulerian observer, normalized so that $n_{\mu}n^{\mu}=-1$, and $v^{\mu}=(0,v^{i})$ are the Eulerian components of the fluid three-velocity, orthogonal to $n^{\mu}$, so that $n_{\mu}v^{\mu}=\eta_{\mu\nu}n^{\nu}v^{\mu}=0$.  
Here, 
\begin{equation}
    \eta_{\mu\nu}=\mbox{diag}[-1,1,1,1]
    \label{eq:minkowskiMetric}
\end{equation}
is the Minkowski metric.  
With the normalization $u_{\mu}u^{\mu}=-1$, it follows that $W^{2}=1/(1-v^{2})$ is the squared Lorentz factor, where $v^{2}=v_{\mu}v^{\mu}=v_{i}v^{i}$.  
(In this paper we always assume $v<1$.)
Similarly, the particle four-momentum has the \emph{Lagrangian decomposition}
\begin{equation}
  p^{\mu} = \varepsilon\big(u^{\mu}+\ell^{\mu}\big),
  \label{eq:particleFourMomentum}
\end{equation}
where $\ell^{\mu}$ is a unit spacelike four-vector, $\ell_{\mu}\ell^{\mu}=1$, orthogonal to $u^{\mu}$ so that $u_{\mu}\ell^{\mu}=0$.  
Then, the component of $p^{\mu}$ along $u^{\mu}$, $\varepsilon=-u_{\mu}p^{\mu}$, is the particle energy measured by a Lagrangian observer.  
Eqs.~\eqref{eq:fluidFourVelocity} and \eqref{eq:particleFourMomentum} are decompositions relative to four-velocities $n^{\mu}$ (Eulerian observer) and $u^{\mu}$ (Lagrangian observer), respectively.  
We let
\begin{align}
  \gamma^{\mu}_{\hspace{4pt}\nu} 
  &= \delta^{\mu}_{\hspace{4pt}\nu} + n^{\mu}n_{\nu}, \label{eq:eulerianProjector} \\
  h^{\mu}_{\hspace{4pt}\nu} 
  &= \delta^{\mu}_{\hspace{4pt}\nu} + u^{\mu}u_{\nu}, \label{eq:lagrangianProjector}
\end{align}
denote the projectors orthogonal to $n^{\mu}$ and $u^{\mu}$, respectively.  
In particular, 
\begin{alignat}{2}
  \gamma^{\mu}_{\hspace{4pt}\nu}n^{\nu}&=0,
  \quad
  \gamma^{\mu}_{\hspace{4pt}\nu}u^{\nu}&&=Wv^{\mu},
  \\
  h^{\mu}_{\hspace{4pt}\nu}u^{\nu}&=0,
  \quad
  h^{\mu}_{\hspace{4pt}\nu}p^{\nu}&&=\varepsilon\ell^{\mu}.  
\end{alignat}

We also introduce the Eulerian decomposition of the particle four-momentum 
\begin{equation}
    p^{\mu} = E\,\big(\,n^{\mu}+L^{\mu}\,\big),
    \label{eq:particleFourMomentum_Eulerian}
\end{equation}
where $n_{\mu}L^{\mu}=0$.  
Using the Eulerian four-velocity, the projector in Eq.~\eqref{eq:eulerianProjector}, and the Lagrangian decomposition in Eq.~\eqref{eq:particleFourMomentum}, the Eulerian components can be expressed in terms of components of the Lagrangian decomposition as
\begin{align}
  E=-n_{\mu}p^{\mu}
  &= \varepsilon\,\big(\,W+v_{\mu}\,\ell^{\mu}\,\big), \label{eq:EulerianParticleEnergy} \\
  L^{\mu}=\gamma^{\mu}_{\hspace{4pt}\nu}p^{\nu}/E
  &= \frac{\big(\,\ell^{\mu}+Wv^{\mu}-n^{\mu}\,v^{\nu}\,\ell_{\nu}\,\big)}{\big(\,W+v^{\nu}\,\ell_{\nu}\,\big)}.
  \label{eq:EulerianUnitDirection}
\end{align}
Here, $E$ is the particle energy measured by an Eulerian observer.  
A straightforward calculation shows that $L_{\mu}L^{\mu}=1$.  
Moreover, since $n_{\mu}L^{\mu}=0$, it follows that $L^{0}=0$.  
This can also be verified by direct evaluation using Eq.~\eqref{eq:EulerianUnitDirection} (noting that $\ell^{0}=v^{\nu}\ell_{\nu}$).
Since $\varepsilon>0$ and $v<1$, it is straightforward to verify that $E>0$.  
Here, $E$ is to be considered a function of the comoving momentum space coordinates $\varepsilon$ and $\omega$, while $L^{\mu}$ is a function of the momentum space angular coordinates $\omega$ (similar to $\ell^{\mu}$).  

\subsection{Evolution Equations}
\label{sec:model.EvolutionEquations}

Following \cite{cardall_etal_2013a}, we define Eulerian and Lagrangian decompositions of the energy-momentum tensor
\begin{align}
  \mathcal{T}^{\mu\nu}
  &= \mathcal{E}\,n^{\mu}\,n^{\nu} + \mathcal{F}^{\mu}\,n^{\nu} + n^{\mu}\,\mathcal{F}^{\nu} + \mathcal{S}^{\mu\nu}, \label{eq:energyMomentumEulerian} \\
  &= \cJ\,u^{\mu}\,u^{\nu} + \mathcal{H}^{\mu}\,u^{\nu} + u^{\mu}\,\mathcal{H}^{\nu} + \mathcal{K}^{\mu\nu}, \label{eq:energyMomentumLagrangian}
\end{align}
where, after inserting Eq.~\eqref{eq:particleFourMomentum} into Eq.~\eqref{eq:angularMoments}, the components of the Lagrangian decomposition are found to be given by
\begin{subequations}
    \label{eq:lagrangianMoments}
    \begin{align}
        \cJ&=\frac{\varepsilon}{4\pi}\int_{\bbS^{2}}fd\omega,\\
        \cH^{\mu}&=\frac{\varepsilon}{4\pi}\int_{\bbS^{2}}f\ell^{\mu}d\omega,\\
        \cK^{\mu\nu}
        &=\frac{\varepsilon}{4\pi}\int_{\bbS^{2}}f\ell^{\mu}\ell^{\nu}d\omega,
    \end{align}
\end{subequations}
and $u_{\mu}\mathcal{H}^{\mu}=0$ and $u_{\mu}\mathcal{K}^{\mu\nu}=u_{\nu}\mathcal{K}^{\mu\nu}=0$.  
The components of the Eulerian decomposition in Eq.~\eqref{eq:energyMomentumEulerian} satisfy $n_{\mu}\mathcal{F}^{\mu}=0$ and $n_{\mu}\mathcal{S}^{\mu\nu}=n_{\nu}\mathcal{S}^{\mu\nu}=0$, and can be expressed in terms of the Lagrangian moments in Eq.~\eqref{eq:lagrangianMoments} as \cite{cardall_etal_2013a}
\begin{alignat}{2}
    \cE
    &=n_{\mu}n_{\nu}\cT^{\mu\nu}
    &&=W^{2}\cJ + 2Wv^{\mu}\cH_{\mu}+v^{\mu}v^{\nu}\cK_{\mu\nu}, \label{eq:eulerianEnergy} \\
    \cF^{i}
    &=-n_{\mu}\gamma^{i}_{\hspace{4pt}\nu}\cT^{\mu\nu}
    &&=W\cH^{i} + Wv^{i}\big(\,W\cJ + v^{j}\cH_{j}\,\big) + v^{j}\cK^{i}_{\hspace{4pt}j}, \label{eq:eulerianMomentum} \\
    \cS^{ij}
    &=\gamma^{i}_{\hspace{4pt}\mu}\gamma^{j}_{\hspace{4pt}\nu}\cT^{\mu\nu}
    &&=\cK^{ij} + W\big(\,\cH^{i}v^{j}
    +v^{i}\cH^{j}\,\big) + W^{2}v^{i}v^{j}\cJ.\label{eq:eulerianStress}
\end{alignat}
Because $n_{\mu}=\eta_{\mu\nu}n^{\nu}=(-1,0,0,0)$, the orthogonality conditions on the Eulerian components imply that $\cF^{0}=0$, $\cS^{0\nu}=0$, and $\cS^{\mu0}=0$.  
Moreover, in Minkowski spacetime, where the metric is given by Eq.~\eqref{eq:minkowskiMetric}, $\cF_{0}=\eta_{0\nu}\cF^{\nu}=\eta_{00}\cF^{0}=0$.  
For the same reason, $\cS_{0\nu}=0$ and $\cS_{\mu0}=0$.  

We obtain evolution equations for the Eulerian energy and momentum by, respectively, projecting Eq.~\eqref{eq:fourMomentumConservation} along $n^{\mu}$ and tangential to the slice with normal $n^{\mu}$ (using $\gamma_{\mu\nu}$).  
The result is (cf. \cite{shibata_etal_2011,cardall_etal_2013a})
\begin{subequations}
    \label{eq:momentEquations}
    \begin{align}
        \p_{t}\cE +& \p_{i}\cF^{i} 
        - \frac{1}{\varepsilon^{2}}\p_{\varepsilon}\big(\,\varepsilon^{2}\,(-n_{\mu})\,\cQ^{\mu\nu\rho}\,\p_{\nu}u_{\rho}\,\big)\nonumber\\
        =&W\chi(\cJ_{\Equilibrium}-\cJ)-\kappa\, v^{i}\cH_{i}, \label{eq:momentEquation_Energy} \\
        \p_{t}\cF_{j} +& \p_{i}\cS^{i}_{\hspace{4pt}j} 
        - \frac{1}{\varepsilon^{2}}\p_{\varepsilon}\big(\,\varepsilon^{2}\,\gamma_{j\mu}\,\cQ^{\mu\nu\rho}\,\p_{\nu}u_{\rho}\,\big) \nonumber\\
        =& -\kappa\,\cH_{j} + Wv_{j}\,\chi (\cJ_{\Equilibrium}-\cJ). \label{eq:momentEquation_Momentum}
    \end{align}
\end{subequations}
Note that indices on vectors and tensors are lowered and raised with the metric, $\eta_{\mu\nu}$, and its inverse, $\eta^{\mu\nu}$, respectively; e.g., $\cF_{j}=\eta_{j\mu}\cF^{\mu}=\delta_{ij}\cF^{i}$.

Eqs.~\eqref{eq:momentEquation_Energy}-\eqref{eq:momentEquation_Momentum} comprise the system for which we develop the DG-IMEX method, beginning in Section~\ref{sec:method}.  
In the absence of collisional sources on the right-hand side, they reduce to phase-space conservation laws for the spectral energy and momentum momentum densities.  
By further integrating over particle energy, with weight $\varepsilon^{2}$, and assuming that the distribution vanishes sufficiently fast for large $\varepsilon$, we obtain conservation laws for the Eulerian-frame energy and momentum,
\begin{equation}
    \p_{t}\mathsf{E}+\p_{i}\mathsf{F}^{i} = 0
    \quad\text{and}\quad
    \p_{t}\mathsf{F}_{j}+\p_{i}\mathsf{S}^{i}_{\hspace{4pt}j} = 0,
\end{equation}
respectively, where the Eulerian-frame grey moments are defined as
\begin{equation}
    \{\,\mathsf{E},\mathsf{F}^{i},\mathsf{S}^{ij}\,\}=4\pi\int_{\bbR^{+}}\{\,\cE,\cF^{i},\cS^{ij}\,\}\,\varepsilon^{2}\,d\varepsilon.
    \label{eq:eulerianGreyMoments}
\end{equation}

\response
{
We point out that the system in Eq.~\eqref{eq:momentEquations} is consistent with the following evolution equation for the Eulerian number density \cite{cardall_etal_2013a}
\begin{equation}
    \p_{t}\cN + \p_{i}\cF_{\cN}^{i} 
    - \frac{1}{\varepsilon^{2}}\p_{\varepsilon}\big(\,\varepsilon^{2}\,\cT^{\nu\rho}\,\p_{\nu}u_{\rho}\,\big)
    =\chi\,\big(\cJ_{\Equilibrium}-\cJ\big)/\varepsilon,
    \label{eq:eulerianNumberEquation}
\end{equation}
where the Eulerian number density is given by
\begin{equation}
    \cN=W(\cE-v^{i}\cF_{i})/\varepsilon=(W\cJ+v^{i}\cH_{i})/\varepsilon
    \label{eq:eulerianNumberDensity}
\end{equation}
and $\cF_{\cN}^{i}=(\cH^{i}+W\cJ v^{i})/\varepsilon$.  
In the absence of collisions, Eq.~\eqref{eq:eulerianNumberEquation} is a phase-space conservation law for the spectral number density.  
Then, by further integrating Eq.~\eqref{eq:eulerianNumberEquation} over particle energy, we obtain a conservation law for the Eulerian-frame number
\begin{equation}
    \p_{t}\mathsf{N}+\p_{i}\mathsf{F}_{\mathsf{N}}^{i}=0,
\end{equation}
where the Eulerian-frame gray number density and number flux are defined as
\begin{equation}
    \{\mathsf{N},\mathsf{F}_{\mathsf{N}}^{i}\}
    =4\pi\int_{\bbR^{+}}\{\cN,\cF_{\cN}^{i}\}\,\varepsilon^{2}d\varepsilon.  
\end{equation}
As elaborated in detail in \cite{cardall_etal_2013a} (see also Section~6.5.4 in \cite{mezzacappa_etal_2020} for the special relativistic case considered here), the analytical relationship between the system in Eq.~\eqref{eq:momentEquations} and Eq.~\eqref{eq:eulerianNumberEquation}, as can be inferred from Eq.~\eqref{eq:eulerianNumberDensity}, is nontrivial.  
It involves exact cancellation of terms remaining after bringing $W/\varepsilon$ and $Wv^{i}/\varepsilon$ inside the phase-space divergences in Eq.~\eqref{eq:momentEquations}.  
When solving the system in Eq.~\eqref{eq:momentEquations} numerically, maintaining consistency with Eq.~\eqref{eq:eulerianNumberEquation} at the discrete level for \emph{simultaneous} four-momentum and number conservation is challenging, and, to our knowledge, remains an open problem.  
Capturing this consistency is not a focus of this paper, but we monitor the number density in some of the test cases considered in Section~\ref{sec:numericalResults}.  
}

\response{
\begin{rem}
    As a simplification to make our analysis more tractable, we assume that the material background is steady; i.e., the fluid four-velocity ($u^{\mu}$), the emissivity ($\eta$), the absorption and scattering opacities ($\chi$ and $\sigma$), and the equilibrium density ($\cJ_{\Equilibrium}$) are assumed to be independent of time.  
    Without these simplifying assumptions, the spectral two-moment model should be coupled to equations for relativistic hydrodynamics (e.g., \cite{rezzollaZanotti_2013}), which is beyond the scope of this paper.  
    We discuss the relevance of our work to this more general case in Section~\ref{sec:summaryConclusions}.  
\end{rem}
}

\subsection{Moment Closure}
\label{sec:closure}

The two-moment model given by Eqs.~\eqref{eq:momentEquation_Energy}-\eqref{eq:momentEquation_Momentum} contains higher order moments, and is not closed.  
To close the system, the higher-order moments are determined through a closure procedure.  
Here, the pressure tensor $\mathcal{K}^{\mu\nu}$ is obtained from the lower-order moments as (e.g., \cite{levermore_1984,Anile_etal_1992})
\begin{equation}
  \mathcal{K}^{\mu\nu}
  = \frac{1}{2}\,\Big[\,(1-\mathsf{k})\,h^{\mu\nu}+(3\mathsf{k}-1)\,\widehat{\mathsf{h}}^{\mu}\widehat{\mathsf{h}}^{\nu}\,\Big]\,\cJ,
  \label{eq:pressureTensor}
\end{equation}
where $\mathsf{k}$ is the Eddington factor (in general a function of the flux factor $\mathsf{h}=\mathcal{H}/\cJ$ and $\cJ$, where $\mathcal{H}=\sqrt{\mathcal{H}_{\mu}\mathcal{H}^{\mu}}$) and $\widehat{\mathsf{h}}^{\mu}=\mathcal{H}^{\mu}/\mathcal{H}$.  
The pressure tensor in Eq.~\eqref{eq:pressureTensor} satisfies the trace condition
\begin{equation}
    \cK^{\mu}_{\hspace{4pt}\mu}
    =\eta_{\mu\nu}\,\cK^{\mu\nu}=\cJ.
\end{equation}
Moreover, since $\widehat{\mathsf{h}}_{\mu}\,\widehat{\mathsf{h}}_{\nu}\,\cK^{\mu\nu}=\mathsf{k}\,\cJ$, the definition of $\cK^{\mu\nu}$ in Eq.~\eqref{eq:lagrangianMoments} implies that the Eddington factor is given by
\begin{equation}
    \mathsf{k} 
    = \frac{\int_{\bbS^{2}}f\, (\widehat{\mathsf{h}}_{\mu}\ell^{\mu})^{2}\, d\omega}{\int_{\bbS^{2}}f d\omega}
    =
    \frac{\frac{1}{2}\int_{-1}^1\mathfrak{f}(\mu)\mu^2\,d\mu}{\frac{1}{2}\int_{-1}^1\mathfrak{f}(\mu)\,d\mu},
\end{equation}
where we have defined
\begin{equation}
    \mathfrak{f}(\mu) = \frac{1}{2\pi}\int_{0}^{2\pi}f(\mu,\varphi)\,d\varphi.
\end{equation}
With this definition, the momentum-space coordinates in the co-moving frame have been aligned with $\widehat{\mathsf{h}}_{\mu}$ so that $\widehat{\mathsf{h}}_{\mu}\ell^{\mu} = \cos\vartheta = \mu$.  
Below, a functional form is imposed on $\mathfrak{f}(\mu)$, which allows the Eddington factor to be computed.  

The energy derivatives in Eqs.~\eqref{eq:momentEquation_Energy}-\eqref{eq:momentEquation_Momentum} contain projections of the rank-three tensor $\mathcal{Q}^{\mu\nu\rho}$, which has the Lagrangian decomposition
\begin{align}
  \mathcal{Q}^{\mu\nu\rho}/\varepsilon
  &= \cJ\,u^{\mu}\,u^{\nu}\,u^{\rho} + \mathcal{H}^{\mu}\,u^{\nu}\,u^{\rho} + \mathcal{H}^{\nu}\,u^{\mu}\,u^{\rho} + \mathcal{H}^{\rho}\,u^{\mu}\,u^{\nu} \nonumber \\
  &\hspace{12pt} + \mathcal{K}^{\mu\nu}\,u^{\rho} + \mathcal{K}^{\mu\rho}\,u^{\nu} + \mathcal{K}^{\nu\rho}\,u^{\mu} + \mathcal{L}^{\mu\nu\rho}, \label{eq:heatFluxLagrangian}
\end{align}
where
\begin{equation}
    \mathcal{L}^{\mu\nu\rho}=\frac{\varepsilon}{4\pi}\int_{\bbS^{2}}f\ell^{\mu}\ell^{\nu}\ell^{\rho}d\omega.  
    \label{eq:lagrangianRankThreeMoment}
\end{equation}
In analogy with Eq.~\eqref{eq:pressureTensor}, the rank-three (heat flux) tensor is written as (e.g., \cite{Pennisi_1992,just_etal_2015})
\begin{align}
  \mathcal{L}^{\mu\nu\rho}
  =\frac{1}{2}\,
  \Big[\,&
    (\mathsf{h}-\mathsf{q})\,\big(\,\widehat{\mathsf{h}}^{\mu}\,h^{\nu\rho}+\widehat{\mathsf{h}}^{\nu}\,h^{\mu\rho}+\widehat{\mathsf{h}}^{\rho}\,h^{\mu\nu}\,\big) \nonumber\\
    &+(5\mathsf{q}-3\mathsf{h})\,\widehat{\mathsf{h}}^{\mu}\,\widehat{\mathsf{h}}^{\nu}\,\widehat{\mathsf{h}}^{\rho}
  \,\Big]\,\cJ,
  \label{eq:heatFluxTensor}
\end{align}
where $\mathsf{q}$ is the heat flux factor (also a function of $\mathsf{h}$ and $\cJ$).  
The heat flux tensor in Eq.~\eqref{eq:heatFluxTensor} satisfies the trace conditions
\begin{equation}
    \cL^{\mu\nu}_{\hspace{8pt}\nu}
    =\eta_{\nu\rho}\,\cL^{\mu\nu\rho}
    =\cH^{\mu}.  
\end{equation}
Then, since $\widehat{\mathsf{h}}_{\mu}\,\widehat{\mathsf{h}}_{\nu}\,\widehat{\mathsf{h}}_{\rho}\,\cL^{\mu\nu\rho}=\mathsf{q}\,\cJ$, the definition in Eq.~\eqref{eq:lagrangianRankThreeMoment} gives the heat flux factor
\begin{equation}
    \mathsf{q}
    = \frac{\int_{\bbS^{2}}f\,(\widehat{\mathsf{h}}_{\mu}\ell^{\mu})^{3}\,d\omega}{\int_{\bbS^{2}}f d\omega}
    =
    \frac{\frac{1}{2}\int_{-1}^1\mathfrak{f}(\mu)\mu^3\,d\mu}{\frac{1}{2}\int_{-1}^1\mathfrak{f}(\mu)\,d\mu}.
\end{equation}

The moment model in Eq.~\eqref{eq:momentEquations} is closed when $\mathsf{k}$ and $\mathsf{q}$ are specified in terms of $\mathsf{h}$ and $\cJ$.  
We determine $\mathsf{k}$ and $\mathsf{q}$ using the maximum entropy closure (see, e.g., \cite{minerbo_1978,cernohorskyBludman_1994,lareckiBanach_2011,richers_2020}).  
In this approach, $\mathsf{k}$ and $\mathsf{q}$ are determined by finding a distribution function $\mathfrak{f}_{\MaxEnt}$ that maximizes the entropy and recovers the lower order moments $\cJ$ and $\cH=\widehat{\mathsf{h}}_{\mu}\cH^{\mu}$.  
In the simple case of Maxwell--Boltzmann statistics, the maximum entropy distribution has the general form \cite{minerbo_1978}
\begin{equation}
    \mathfrak{f}_{\MaxEnt}(\mu) = \exp(\alpha_0+\alpha_1 \mu).
\end{equation}
Given $\cJ$ and $\cH$, direct integration of $\mathfrak{f}_{\MaxEnt}$  yields
\begin{align}
    \cJ 
    = \frac{1}{2}\int_{-1}^1\mathfrak{f}_{\MaxEnt}(\mu)\,d\mu 
    =& e^{\alpha_0}\sinh(\alpha_1)/\alpha_1, \\
    \cH 
    = \frac{1}{2}\int_{-1}^1\mathfrak{f}_{\MaxEnt}(\mu)\mu\,d\mu 
    =&
    e^{\alpha_0}(\alpha_1\cosh(\alpha_1)-\sinh(\alpha_1))/\alpha_1^2,
\end{align}
which can be solved for $\alpha_0$ and $\alpha_1$.  
In particular, the known flux factor, $\mathsf{h}$, can be expressed in terms of the Langevin function on $\alpha_1$
\begin{equation}
    \mathsf{h} = \coth(\alpha_1)-1/\alpha_1 =\vcentcolon L(\alpha_1).
\end{equation}
The coefficient $\alpha_1(\mathsf{h})=L^{-1}(\mathsf{h})$ is then obtained by inverting the Langevin function.  
Once $\alpha_{1}$ is obtained, $\alpha_{0}$ and $\mathfrak{f}_{\MaxEnt}$ are easily obtained.  
Then, by directly integrating $\mathfrak{f}_{\MaxEnt}$, we have
\begin{equation}
    \mathsf{k}(\mathsf{h}) = 1-\frac{2\mathsf{h}}{\alpha_1(\mathsf{h})},
    \quad
    \mathsf{q}(\mathsf{h}) = \coth(\alpha_1(\mathsf{h})) - \frac{3\mathsf{k}(\mathsf{h})}{\alpha_1(\mathsf{h})}.
\end{equation}
This is the closure given by Minerbo.

Inverting the Langevin function requires numerical root finding, which can be costly.  
In practice we use polynomial approximations to $\mathsf{k}$ and $\mathsf{q}$ in terms of $\mathsf{h}$.  
This leads to the computationally more efficient algebraic expressions \cite{cernohorskyBludman_1994,just_etal_2015}
\begin{align}
    \mathsf{k}(\mathsf{h}) 
    &=
    \frac{1}{3}
    +\frac{2}{15}(3\mathsf{h}^2-\mathsf{h}^3+3\mathsf{h}^4), 
    \label{eq:eddingtonFactorAlgebraic} \\
    \mathsf{q}(\mathsf{h})
    &=\frac{\mathsf{h}}{75}(45+10\mathsf{h}-12\mathsf{h}^2-12\mathsf{h}^3+38\mathsf{h}^4-12\mathsf{h}^5+18\mathsf{h}^6).
    \label{eq:heatFluxFactorAlgebraic}
\end{align}
The algebraic expression for $\mathsf{k}$ is accurate to within one percent, while the expression for $\mathsf{q}$ is accurate to within three percent (e.g., \cite{laiu_etal_2025}).  

\subsection{Tetrad Formalism}
\label{sec:tetradFormalism}

In the following, it will sometimes be useful to appeal to the tetrad formalism (e.g., \cite{lindquist_1966,cardallMezzacappa_2003}), which relates components of coordinate basis four-vectors and tensors (e.g., $p^{\mu}$) to corresponding components in a local orthonormal frame comoving with the fluid (indices adorned with a `hat'; e.g., $p^{\hat{\mu}}$)
\begin{equation}
  p^{\mu} = \cL^{\mu}_{\hspace{4pt}\hat{\mu}}\,p^{\hat{\mu}},
  \quad\text{where}\quad
  p^{\hat{\mu}} = \varepsilon\,\big(1,\ell^{\hat{\imath}}\big),
  \label{eq:fourMomentumTetradBasis}
\end{equation}
and where
\begin{equation}
  \ell^{\hat{\mu}} = \big(\,0,\,\ell^{\hat{\imath}}\,\big) = \big(\,0,\,\cos\vartheta,\,\sin\vartheta\cos\varphi,\,\sin\vartheta\sin\varphi\,\big)
  \label{eq:orthonormalComovingUnit}
\end{equation}
is a unit spatial four-vector with spatial components parallel to the particle three-momentum in the orthonormal comoving frame.  

In the special relativistic case with Cartesian spatial coordinates considered here, the transformation $\cL^{\mu}_{\hspace{4pt}\hat{\mu}}$ from the orthonormal comoving basis to the coordinate basis is simply the Lorentz transformation
\begin{align}
  \cL^{\mu}_{\hspace{4pt}\hat{\mu}}
  =&\left(\begin{array}{cc}
  \cL^{0}_{\hspace{4pt}\hat{0}} & \cL^{0}_{\hspace{4pt}\hat{\imath}} \\
  \cL^{i}_{\hspace{4pt}\hat{0}}  & \cL^{i}_{\hspace{4pt}\hat{\imath}}
  \end{array}\right) \nonumber\\
  =&\left(\begin{array}{cc}
  W & W V_{\hat{\imath}} \\
  W V^{i} & \delta^{i}_{\hspace{4pt}\hat{\imath}}+(W-1)(V^{i}/V)(V_{\hat{\imath}}/V)
  \end{array}\right),
  \label{eq:lorentzTransformation}
\end{align}
We follow \cite[Appendix~B]{cardall_etal_2013a}, where $V=\sqrt{v^{\mu}v_{\mu}}=\sqrt{v^{i}v_{i}}$ and
\begin{equation}
  V^{1} = v^{1} = V_{\hat{1}}, \quad
  V^{2} = v^{2} = V_{\hat{2}}, \quad
  V^{3} = v^{3} = V_{\hat{3}}
\end{equation}
are three-velocity parameters in the Lorentz boost (not to be viewed here as components of four-vectors). This is expressed to be consistent with our general index convention, where unadorned indices are used to denote components of coordinate basis four-vectors while indices with a hat denote components of four-vectors expressed in an orthonormal basis comoving with the fluid.  
The inverse transformation (obtained by replacing $V^{i}\to-V^{\hat{\imath}}$, $V_{\hat{\imath}}\to-V_{i}$, and $\delta^{i}_{\hspace{4pt}\hat{\imath}}\to\delta^{\hat{\imath}}_{\hspace{4pt}i}$ in Eq.~\eqref{eq:lorentzTransformation}) is denoted $\cL^{\hat{\mu}}_{\hspace{4pt}\mu}$, and satisfies $\cL^{\mu}_{\hspace{4pt}\hat{\mu}}\cL^{\hat{\mu}}_{\hspace{4pt}\nu}=\delta^{\mu}_{\hspace{4pt}\nu}$.  

Using Eq.~\eqref{eq:fourMomentumTetradBasis} to write
\begin{equation}
    p^{\mu}=\varepsilon\,(\cL^{\mu}_{\hspace{4pt}\hat{0}}+\cL^{\mu}_{\hspace{4pt}\hat{\imath}}\ell^{\hat{\imath}}),
\end{equation}
a comparison with Eq.~\eqref{eq:particleFourMomentum} shows that
\begin{equation}
    u^{\mu}=\cL^{\mu}_{\hspace{4pt}\hat{0}}=\cL^{\mu}_{\hspace{4pt}\hat{\mu}}u^{\hat{\mu}}
    \quad\text{and}\quad
    \ell^{\mu}=\cL^{\mu}_{\hspace{4pt}\hat{\imath}}\ell^{\hat{\imath}}=\cL^{\mu}_{\hspace{4pt}\hat{\mu}}\ell^{\hat{\mu}}.
    \label{eq:tetradUandL}
\end{equation}
(Note that $u^{\hat{\mu}}=(1,0,0,0)$.)  
Using Eq.~\eqref{eq:tetradUandL}, we can, as an example, write $\cH^{\mu}$ in Eq.~\eqref{eq:lagrangianMoments} as
\begin{equation}
    \cH^{\mu}
    =\cL^{\mu}_{\hspace{4pt}\hat{\mu}}\,\cH^{\hat{\mu}},
    \quad\text{where}\quad
    \cH^{\hat{\mu}}
    =\frac{\varepsilon}{4\pi}\int_{\bbS^{2}}f\,\ell^{\hat{\mu}}\,d\omega.
    \label{eq:lagrangianMomentsHat}
\end{equation}

\subsection{Moment Realizability}

In this paper, we aim to design a realizability-preserving numerical method for the two-moment model in Eq.~\eqref{eq:momentEquations}.  
In the relativistic setting, multiple moment pairs are encountered; e.g., the Lagrangian moments $(\cJ,\cH_{\mu})^{\intercal}$ and the Eulerian moments $(\cE,\cF_{\mu})^{\intercal}$.  
To introduce the concept of moment realizability, we first define the set of non-negative distribution functions 
\begin{equation}
  \mathfrak{R} = \Big\{\,f~\big|~f\ge0 \quad\text{and}\quad \frac{1}{4\pi}\int_{\mathbb{S}^{2}}f\,d\omega > 0\,\Big\}.  
  \label{eq:realizableDistributions}
\end{equation}
(I.e., we avoid the trivial case where $f$ is zero everywhere.)
The moments $\boldsymbol{\mathscr{M}}=(\mathscr{J},\mathscr{H}_{\mu})^{\intercal}$, defined by
\begin{equation}
    \mathscr{J}=\frac{1}{4\pi}\int_{\bbS^{2}}g\,d\omega
    \quad\text{and}\quad
    \mathscr{H}_{\mu}=\frac{1}{4\pi}\int_{\bbS^{2}}g\,\mathscr{L}_{\mu}\,d\omega,
    \label{eq:momentsGeneric}
\end{equation}
where $\mathscr{L}_{\mu}(\omega)$ is a spacelike unit four-vector ($\mathscr{L}_{\mu}\mathscr{L}^{\mu}=1$), are said to be realizable if they arise from a distribution function $g\in\mathfrak{R}$, and we define the set of realizable moments $\cR$ as 
\begin{equation}
    \cR = \big\{\, \boldsymbol{\mathscr{M}} = (\mathscr{J}, \mathscr{H}_{\mu})^{\intercal} \,\, \rvert \,\, \mathscr{J} > 0 \text{ and } \gamma(\boldsymbol{\mathscr{M}}) \geq 0 \,\big\},
    \label{eq:realizableSet}
\end{equation}
where $\gamma(\boldsymbol{\mathscr{M}}) = \mathscr{J} - \sqrt{\mathscr{H}_{\mu}\mathscr{H}^{\mu}}$.
It is straightforward to verify that $\boldsymbol{\mathscr{M}}\in\cR$ if and only if there exists an underlying distribution $g$ satisfying Eq.~\eqref{eq:momentsGeneric} that is in $\mathfrak{R}$. 

From the definition of the realizable set in Eq.~\eqref{eq:realizableSet}, the set of all realizable moments forms a convex cone.  
As a consequence, we have the following lemma
\begin{lemma}
\label{lem:convexcone}
    Let $\mathscr{M}^{(a)}=(\mathscr{J}^{(a)},\mathscr{H}^{(a)}_{\mu})^{\intercal}$ and $\mathscr{M}^{(b)}=(\mathscr{J}^{(b)},\mathscr{H}^{(b)}_{\mu})^{\intercal}$ be realizable moments.
    For $\theta_1,\theta_2>0$, define $\mathscr{M}^{(c)}\vcentcolon=\theta_1\mathscr{M}^{(a)}+\theta_2\mathscr{M}^{(b)}$.
    Then $\mathscr{M}^{(c)}\in\cR$. 
\end{lemma}\label{lem:case2_addition_of_moments}
\begin{proof}
    Since $\mathscr{M}^{(a)}\in\cR$, there exists an $f^{(a)}\in\mathfrak{R}$ such that
    \[
    \mathscr{J}^{(a)}=\frac{1}{4\pi}\int_{\bbS^2}f^{(a)}\,d\omega \quad \text{and} \quad \mathscr{H}^{(a)}_\mu=\frac{1}{4\pi}\int_{\bbS^2}f^{(a)}\mathscr{L}^\mu\,d\omega.
    \]
    There exists an analogous $f^{(b)}\in\mathfrak{R}$ for $\mathscr{M}^{(b)}$.
    Let $f^{(c)}=\theta_1f^{(a)}+\theta_2f^{(b)}$.
    Then, since $\theta_1,\theta_2>0$, $f^{(c)}\in\mathfrak{R}$, and thus $\mathscr{M}^{(c)}$ is realizable as 
    \[
    \mathscr{J}^{(c)}=\frac{1}{4\pi}\int_{\bbS^2}f^{(c)}\,d\omega \quad \text{and} \quad \mathscr{H}^{(c)}_\mu=\frac{1}{4\pi}\int_{\bbS^2}f^{(c)}\mathscr{L}^\mu\,d\omega.
    \]
\end{proof}
\begin{corollary}
    Let $\mathscr{M}\in\cR$.  Then, for $\theta>0$, $\theta\mathscr{M}\in\cR$.
\end{corollary}

\begin{rem}
    The definition of the realizable set in Eq.~\eqref{eq:realizableSet} is made with respect to the \underline{generic} moment pair in Eq.~\eqref{eq:momentsGeneric}.  
    In Propositions~\ref{prop:momentBoundsOrthonormalComoving}-\ref{prop:momentBoundsEulerian} we prove the realizability of \underline{specific} moment pairs that appear in the analysis of our proposed numerical method.
\end{rem}

\begin{prop}
    Let the moments $\cJ$ and $\cH^{\hat{\mu}}$ be defined as in Eqs.~\eqref{eq:lagrangianMoments} and \eqref{eq:lagrangianMomentsHat}, respectively, 
    where $\varepsilon>0$, $f\in\mathfrak{R}$, and $\ell^{\hat{\mu}}$ is defined in Eq.~\eqref{eq:orthonormalComovingUnit}.  
    Then $(\cJ,\cH_{\hat{\mu}})^{\intercal}\in\cR$.
    \label{prop:momentBoundsOrthonormalComoving}
\end{prop}
\begin{proof}
    Since $\varepsilon f\in\mathfrak{R}$ it follows that $\cJ>0$.  
    Using the Cauchy--Schwarz inequality
    \begin{align*}
        \cH_{\hat{\mu}}\cH^{\hat{\mu}}
        = \cH_{\hat{\imath}}\cH^{\hat{\imath}}
        &= \Big(\frac{\varepsilon}{4\pi}\Big)^{2}\int_{\bbS^{2}}f\ell_{\hat{\imath}}d\omega\int_{\bbS^{2}}f\ell^{\hat{\imath}}d\omega\\
        &\le \Big(\frac{\varepsilon}{4\pi}\Big)^{2}\int_{\bbS^{2}}fd\omega\int_{\bbS^{2}}f\ell_{\hat{\imath}}\ell^{\hat{\imath}}d\omega
        =\cJ^{2}.
    \end{align*}
    Taking the square root on both sides completes the proof.
\end{proof}

\begin{prop}
    Let the angular moments $\cJ$ and $\cH^{\mu}$ be defined as in Eq.~\eqref{eq:lagrangianMoments} with $\varepsilon f\in\mathfrak{R}$.  
    Then, $(\cJ,\cH_{\mu})^{\intercal}\in\cR$.
    \label{prop:momentBoundsLagrangian}
\end{prop}
\begin{proof}
    Since $\varepsilon f\in\mathfrak{R}$, positivity of $\cJ$ follows from Proposition~\ref{prop:momentBoundsOrthonormalComoving}.  
    Lorentz invariance of the inner product and using Proposition~\ref{prop:momentBoundsOrthonormalComoving} gives
    \begin{equation*}
        \cH^{2} 
        = \cH_{\mu}\cH^{\mu}
        = \cL^{\hat{\mu}}_{\hspace{4pt}\mu}\,\cL^{\mu}_{\hspace{4pt}\hat{\nu}}\,\cH_{\hat{\mu}}\,\cH^{\hat{\nu}}
        = \delta^{\hat{\mu}}_{\hspace{4pt}\hat{\nu}}\,\cH_{\hat{\mu}}\,\cH^{\hat{\nu}}
        = \cH_{\hat{\mu}}\,\cH^{\hat{\mu}}
        \le \cJ^{2}.
    \end{equation*}
    Taking the square root on both sides completes the proof.
\end{proof}

\begin{prop}
    Let the angular moments $\widehat{\cE}$ and $\widehat{\cF}^{\mu}$ be defined as
    \begin{subequations}
    \begin{alignat}{2} \label{eq:hatEF}
        \widehat{\cE} 
        &= n_{\mu}\,u_{\nu}\,\cT^{\mu\nu}
        &&= \frac{1}{4\pi}\int_{\bbS^{2}}E\,f\,d\omega
        \\
        \widehat{\cF}^{\mu} 
        &= - n_{\nu}\,h^{\mu}_{\hspace{4pt}\rho}\,\cT^{\nu\rho}
        &&= \frac{1}{4\pi}\int_{\bbS^{2}}E\,f\,\ell^{\mu}\,d\omega,
    \end{alignat}
    \end{subequations}
    respectively, where $\cT^{\mu\nu}$ is given in Eq.~\eqref{eq:angularMoments} with $f\in\mathfrak{R}$, and $E>0$ is defined in Eq.~\eqref{eq:EulerianParticleEnergy}.
    Then, $(\widehat{\cE},\widehat{\cF}_{\mu})^{\intercal}\in\cR$.  
    \label{prop:momentBoundsHat}
\end{prop}
\begin{proof}
    Since $E>0$, it follows that $E\,f\in\mathfrak{R}$, so that $\widehat{\cE}>0$.  
    Using the Cauchy--Schwarz inequality
    \begin{align*}
        \widehat{\cF}^{2}
        &= \Big(\frac{1}{4\pi}\Big)^{2}\int_{\bbS^{2}}E\,f\,\ell_{\hat{\imath}}\,d\omega\int_{\bbS^{2}}E\,f\,\ell^{\hat{\imath}}\,d\omega\\
        &\le \Big(\frac{1}{4\pi}\Big)^{2}\int_{\bbS^{2}}E\,f\,d\omega\int_{\bbS^{2}}E\,f\,\ell_{\hat{\imath}}\ell^{\hat{\imath}}\,d\omega
        =\widehat{\cE}^{2}.
    \end{align*}
    Taking the square root on both sides completes the proof.
\end{proof}

\begin{prop}
    Let the angular moments $\cE$ and $\cF^{\mu}$ be defined as in Eqs.~\eqref{eq:eulerianEnergy} and \eqref{eq:eulerianMomentum},
    \begin{subequations}
    \begin{alignat}{2}
        \cE
        &= n_{\mu}n_{\nu}\cT^{\mu\nu} 
        &&= \frac{1}{4\pi}\int_{\bbS^{2}}(E^{2}/\varepsilon)\,f\,d\omega
        \\
        \cF^{\mu}
        &= - \gamma^{\mu}_{\hspace{4pt}\nu} n_{\rho}\cT^{\nu\rho}
        &&= \frac{1}{4\pi}\int_{\bbS^{2}}(E^{2}/\varepsilon)\,f\,L^{\mu}\,d\omega,
    \end{alignat}
    \end{subequations}
    where $f\in\mathfrak{R}$ and $E,\varepsilon>0$.  
    Then $(\cE,\cF_{\mu})^{\intercal}\in\cR$.  
    \label{prop:momentBoundsEulerian}
\end{prop}
\begin{proof}
    Since $E,\varepsilon>0$ and $f\in\mathfrak{R}$, it follows that $g\vcentcolon=(E^{2}/\varepsilon)f\in\mathfrak{R}$ and $\cE>0$.  
    Using the Cauchy--Schwarz inequality we have
    \begin{align*}
        \cF^{2} 
        = \cF_{i}\cF^{i}
        &= \Big(\frac{1}{4\pi}\Big)^{2}\int_{\bbS^{2}}g\,L_{i}\,d\omega
        \int_{\bbS^{2}}g\,L^{i}\,d\omega\\
        &\le\Big(\frac{1}{4\pi}\Big)^{2}\int_{\bbS^{2}}g\,d\omega
        \int_{\bbS^{2}}g\,L_{i}\,L^{i}\,d\omega
        =\cE^{2}.
    \end{align*}
    Taking the square root on both sides completes the proof.  
\end{proof}

\section{Numerical Method}
\label{sec:method}

In this section, we present the DG-IMEX method for the two-moment model in Eq.~\eqref{eq:momentEquations}.  
The semi-discretization of the two-moment model with the DG method is provided in Section~\ref{sec:semiDiscreteDG}, while Section~\ref{sec:imex} details the integration of the semi-discrete DG scheme with IMEX time-stepping.  

\subsection{Discontinous Galerkin Phase-Space Discretization}
\label{sec:semiDiscreteDG}

To discretize the phase-space of Eq.~\eqref{eq:momentEquations}, we divide the phase-space domain $\Omega=\Omega_{\varepsilon}\times\Omega_{\vect{x}}$ into a disjoint union $\cT$ of open elements $\bK=K_{\varepsilon}\times \bK_{\vect{x}}$, so that $\Omega=\cup_{\bK\in\cT}\bK$.  
Here, $\Omega_{\varepsilon}$ is the energy domain, and $\Omega_{\vect{x}}$ is the $d$-dimensional spatial domain.
We define the DG spatial and energy elements as
\begin{align}
    \bK_{\vect{x}}
    &=\{\,
        \vect{x}: x^i\in K^i = (x^i_{\Lo},x^i_{\Hi})\,\, \text{ for }\,\, i=1,\ldots,d
    \,\}
    \\
    K_{\varepsilon} &= (\varepsilon_{\Lo}, \varepsilon_{\Hi}),
\end{align}
respectively.  
We denote the volume of the DG element as
\begin{equation}
    |\bK|=\int_{\bK}\varepsilon^2\,d\varepsilon\,d\vect{x},
    \quad
    \text{where}
    \quad
    d\vect{x}=\prod_{i=1}^{d} dx^i.
\end{equation}
The length of individual spatial elements is given by $|K^i|=\int_{K^{i}}dx^{i}=x^i_{\Hi}-x^i_{\Lo}=\Delta x^{i}$.  
In the energy dimension we let $|K_{\varepsilon}|=\int_{K^{\varepsilon}}\varepsilon^{2}d\varepsilon=\frac{1}{3}(\varepsilon_{\Hi}^{3}-\varepsilon_{\Lo}^{3})$ and $\Delta\varepsilon=\varepsilon_{\Hi}-\varepsilon_{\Lo}$.  
We introduce the notation $\widetilde{\bK}^{i}= K_{\varepsilon}\times\widetilde{\bK}_{\vect{x}}^{i}$, where $\widetilde{\bK}_{\vect{x}}^{i}=(\times_{j\neq i}K^j)$, and $d\tilde{\vect{x}}^i=\prod_{j\neq i}dx^j$ to define the surface elements orthogonal to the $i$th spatial direction.  
We also define $|\widetilde{\bK}^{i}|=|K_{\varepsilon}|\times(\prod_{j\ne i}|K^{j}|)$.  
Finally we let $\vect{z}=(\varepsilon,\vect{x})$ denote the phase-space coordinate, and define $d\vect{z}=d\varepsilon\,d\vect{x}$ and $d\tilde{\vect{z}}^i=d\varepsilon\,d\tilde{\vect{x}}^i$.

We let the approximation space for the DG method be
\begin{equation}
    \bbV^k_h(\Omega) = \{\phi_h : \phi_h\rvert_{\bK}\in\bbQ^k(\bK),\,\forall\bK\in\cT\},
\end{equation}
where, locally on $\bK$, $\bbQ^k(\bK)$ is the tensor product of one-dimensional polynomials of maximal degree $k$.

To write the two-moment model in Eq.~\eqref{eq:momentEquations} compactly, we define
\begin{subequations}
    \label{eq:momentsCompact}
    \begin{align}
        \bU 
        = &\begin{bmatrix} \cE \\ \cF_j\end{bmatrix}, \\
        \bF^i 
        = &\begin{bmatrix} \cF^i \\ \cS^i_{\hspace{4pt}j} \end{bmatrix}, \\
        \bF^\varepsilon 
        = -\frac{1}{\varepsilon}&\begin{bmatrix} -n_{\mu}\,\cQ^{\mu\nu\rho}\,\p_{\nu}u_{\rho} \\ \gamma_{j\mu}\,\cQ^{\mu\nu\rho}\,\p_{\nu}u_{\rho} \end{bmatrix},\\
        \bcC 
        = &\begin{bmatrix} W\chi\,(\cJ_{\Equilibrium}-\cJ)-\kappa\,v^i\cH_i \\ -\kappa\,\cH_j+Wv_j\,\chi\,(\cJ_{\Equilibrium}-\cJ) \end{bmatrix}.
    \end{align}
\end{subequations}
Then, with the closure specified in Section~\ref{sec:closure}, Eq.~\eqref{eq:momentEquations} can be written as
\begin{equation}\label{eq:momentEquations_vector}
    \p_t\bU + \sum_{i=1}^{d}\p_i\big(\,\bF^i(\bU)\,\big)
    +\frac{1}{\varepsilon^2}\p_{\varepsilon}\big(\,\varepsilon^3\,\bF^\varepsilon(\bU)\,\big)=\bcC(\bU).
\end{equation}
The semi-discrete DG problem for Eq.~\eqref{eq:momentEquations_vector} is to find $\bU_h\in\bbV^k_h(\Omega)$, where $\bU_{h}$ approximates $\bU$, such that
\begin{align}\label{eq:semi-discrete_DG_full_source}
    &\int_{\bK}\p_t\bU_h\,\phi_h\,\varepsilon^2\,d\vect{z} \nonumber\\
    &\hspace{12pt}+\sum_{i=1}^{d}
    \Big\{\,
        \int_{\widetilde{\bK}^i}
        \big[\,\widehat{\bF}^i(\bU_{h})\,\phi_h\rvert_{x^i_{\Hi}}
        -\widehat{\bF}^i(\bU_{h})\phi_h\rvert_{x^i_{\Lo}}\,\big]\varepsilon^2\,d\Tilde{\vect{z}}^i \nonumber\\
        &\hspace{48pt}-\int_{\bK}\bF^i(\bU_{h})\p_i \phi_h\,\varepsilon^2\,d\vect{z}
    \,\Big\}
    \nonumber \\
    &\hspace{12pt}
    +\int_{\bK_{\vect{x}}}
    \big[\,\varepsilon^3\widehat{\bF}^\varepsilon(\bU_{h})\phi_h\rvert_{\varepsilon_{\Hi}} - \varepsilon^3\widehat{\bF}^\varepsilon(\bU_{h})\phi_h\rvert_{\varepsilon_{\Lo}}\,\big]\,d\vect{x} \nonumber\\
    &\hspace{12pt}-\int_{\bK}\varepsilon^3\bF^{\varepsilon}(\bU_{h})\p_\varepsilon \phi_h\,d{\vect{z}}
    =\int_{\bK}\bcC(\bU_h)\phi_h\varepsilon^2\,d\vect{z},
\end{align}
holds for all $\phi_h\in\bbV^k_h(\Omega)$ and all $\bK\in\cT$.  
In Eq.~\eqref{eq:semi-discrete_DG_full_source}, the numerical fluxes, $\widehat{\bF}^{i}$ and $\widehat{\bF}^{\varepsilon}$, approximate $\bF^{i}$ and $\bF^{\varepsilon}$, respectively, on the surface of $\bK$, and are given by Lax--Friedrichs-like fluxes.  
The numerical flux in the $i$th spatial dimension is 
\begin{equation}\label{eq:spatial_flux}
    \widehat{\bF}^i\rvert_x
    =\frac{1}{2}\Big[\,
        (\bF^i\rvert_{x^-}+\bF^i\rvert_{x^+})
        -a^i\,(\bU_h\rvert_{x^+}-\bU_h\rvert_{x^-})
    \,\Big],
\end{equation}
where $x^{\pm}=\lim_{\delta\to0^{+}}x\pm\delta$, and $a^i>0$ is an estimate of the spectral radius of the flux Jacobian matrix $\p\bF^i/\p\bU$. In practice we always assume $a^i = 1$. 
To compute the fluxes $\bF^{i}$, which requires closure evaluations, we first compute the \emph{primitive moments} $\bM=(\cJ,\cH_i)^{\intercal}$ from the \emph{conserved moments} $\bU$.  
This procedure is described in detail in Section~\ref{sec:moment_conversion}.  

The numerical flux in the energy dimension is given by
\begin{equation}\label{eq:energy_flux}
    \widehat{\bF}^\varepsilon\rvert_{\varepsilon}
    =
    \frac{1}{2}\Big[\,(\bF^\varepsilon\rvert_{\varepsilon^-}+\bF^\varepsilon\rvert_{\varepsilon^+})-a^\varepsilon(\widetilde{\bU}\rvert_{\varepsilon^+}-\widetilde{\bU}\rvert_{\varepsilon^-})\,\Big],
\end{equation}
where $\varepsilon^{\pm}=\lim_{\delta\to0^{+}}\varepsilon\pm\delta$,
\begin{equation}
    \widetilde{\bU}
    =\begin{bmatrix} \widetilde{\cE} \\ \widetilde{\cF}_j \end{bmatrix}
    \vcentcolon=
    \begin{bmatrix} 
        n_\mu u_\nu \cT^{\mu\nu} \\ 
        -\gamma_{j\mu}u_\nu\cT^{\mu\nu} 
    \end{bmatrix}
    =
    \begin{bmatrix}
        W\cJ+v^{i}\cH_{i} \\
        \cH_{j}+W\cJ v_{j}
    \end{bmatrix},
    \label{eq:Lagrangian_U_tilde}
\end{equation}
and $a^\varepsilon\geq |q|$ with the quantity $q$ defined as
\begin{equation}
    \label{eq:q}
    q\vcentcolon=\frac{p^\nu p^\rho}{\varepsilon^2}\p_\nu u_\rho=k^\nu k^\rho \p_\nu u_\rho,
\end{equation}
where $k^{\mu}=p^{\mu}/\varepsilon$. 
To compute $\bF^\varepsilon$ and $\widetilde{\bU}$, we again compute $\bM=(\cJ,\cH_i)^{\intercal}$ from $\bU$ (discussed in Section~\ref{sec:moment_conversion}).  
The processes by which we evaluate $\p_\nu u_\rho$ and $a^\varepsilon$ are discussed in the following paragraphs.

To compute $\bF^\varepsilon$, we need to compute the derivatives of the four-velocity, $u_\mu$.
In this paper, as in \cite{laiu_etal_2025}, we assume the four-velocity is independent of time, i.e., $\p_t u_\mu=0$.  
Let $(u_{\mu})_{h}\in\bbV_{h}^{k}(\Omega_{\vect{x}})$ denote the approximation of $u_{\mu}$.  
We compute the derivatives, $(\p_i u_\mu)_h\in\bbV^k_h(\Omega_{\vect{x}})$, by demanding that 
\begin{align}
    \int_{\bK_{\vect{x}}}(\p_i u_\mu)_h\phi_h\,d\vect{x}
    =&
    \int_{\tilde{\bK}_{\vect{x}}^i}
    \big[\,\widehat{u}_{\mu}\,\phi_{h}\rvert_{x^i_{\Hi}}-\widehat{u}_{\mu}\,\phi_{h}\rvert_{x^{i}_{\Lo}}\,\big]\,d\tilde{\vect{x}}^{i} \nonumber\\
    &-
    \int_{\bK_{\vect{x}}}(u_\mu)_{h}\,\p_{i}\phi_{h}\,d\vect{x},
    \label{eq:fourvelocitygradient}
\end{align}
holds for all $\phi_h\in\bbV^k_h(\Omega_{\vect{x}})$, and where velocity on the element boundaries is approximated with the average
\begin{equation}
    \widehat{u}_{\mu}(x^i,\tilde{\vect{x}}^i) 
    = \frac{1}{2}\big[\,(u_{\mu})_{h}(x^{i,-},\tilde{\vect{x}}^{i})+(u_{\mu})_{h}(x^{i,+},\tilde{\vect{x}}^{i})\,\big].
\end{equation}

For the numerical flux in Eq.~\eqref{eq:energy_flux}, we require that $a^\varepsilon\geq|q|$.
Thus, we need an upper bound for $|q|$.  
To this end, we first express $q$ in terms of a contraction with a symmetric tensor $A_{\nu\rho}$
\begin{align}
    q
    &=k^\nu k^\rho \p_\nu u_\rho \nonumber\\
    &=k^\nu k^\rho\Big[\frac{1}{2}(\p_\nu u_\rho+\p_\rho u_\nu)+\frac{1}{2}(\p_\nu u_\rho-\p_\rho u_\nu)\Big] \nonumber\\
    &=k^\nu k^\rho A_{\nu\rho}, 
\end{align}
where $A_{\nu\rho}=\frac{1}{2}(\p_\nu u_\rho + \p_\rho u_\nu)$, since the contraction of the symmetric $k^\nu k^\rho$ with the antisymmetric $\p_\nu u_\rho-\p_\rho u_\nu$ vanishes.  
Using an Eulerian decomposition of $A_{\nu\rho}$, we can derive an upper bound for $q$.  
Next, we introduce the Eulerian decomposition
\begin{equation}
    A_{\mu\nu}=An_\mu n_\nu+B_\mu n_\nu+B_\nu n_\mu+C_{\mu\nu},
    \label{eq:Eulerian_A_munu}
\end{equation}
where $n_\mu B^\mu=0$ and $n_\mu C^{\mu\nu}=0=n_\nu C^{\mu\nu}$.  
Using Eqs.~\eqref{eq:particleFourMomentum_Eulerian} and \eqref{eq:Eulerian_A_munu}, we have
\begin{equation}
    q = \frac{E^2}{\varepsilon^2}\,\big(\, A - 2L^i B_i + L^i L^j C_{ij}\,\big),
\end{equation}
where the components of the Eulerian decomposition in Eq.~\eqref{eq:Eulerian_A_munu} can be expressed in terms of derivatives of the three-velocity as 
\begin{alignat}{2}
    A     &= n_\mu n_\nu A^{\mu\nu} && = \p_t W, \\
    B_{i} &= -\gamma_{i\mu}n_\nu A^{\mu\nu} && = - \frac{1}{2}\,\big(\,\p_t(Wv_i)+\p_i (-W)\,\big),\\
    C_{ij}&= \gamma_{i\mu}\gamma_{j\nu}A^{\mu\nu} && = \frac{1}{2}\,\big(\,\p_i(Wv_j)+\p_j (Wv_i)\,\big).
\end{alignat}
Then we bound $|q|$ as
\begin{equation}
    |q| \leq \frac{E^2}{\varepsilon^2}\,\big(\,|A|+2\sqrt{B_iB^i}+\rho(C_{ij})\,\big),
\end{equation}
where the inequality follows from the triangle inequality, Cauchy--Schwarz, and using $L_iL^i=1$.  
Here, $\rho(C_{ij})$ denotes the spectral radius of $C_{ij}$.  
Since $C_{ij}$ is a symmetric $3\times 3$ matrix, we compute its eigenvalues directly to determine $\rho(C_{ij})$.  
The quantity $E/\varepsilon$ depends on $\ell^\mu$ (see Eq.~\eqref{eq:EulerianParticleEnergy}), which in practice is not known.  
We have that 
\begin{align}
    (E/\varepsilon)^2
    =(W+v^\mu\ell_\mu)^2
    &=(W+Wv^{\hat{\imath}}\ell_{\hat{\imath}})^2 \nonumber\\
    &\leq W^2(1+v)^2
    =\frac{(1+v)}{(1-v)},
    \label{eq:E_over_epsilon_bound}
\end{align}
where $v=\sqrt{v_iv^i}$.  
Hence
\begin{equation}
    |q|\leq \Big(\frac{1+v}{1-v}\Big)\,\big(\,|A|+2\sqrt{B_iB^i}+\rho(C_{ij})\,\big) =\vcentcolon a_{\max}^{\varepsilon}.
    \label{eq:q_upperbound}
\end{equation}
In Eq.~\eqref{eq:energy_flux}, we use this upper bound and set $a^{\varepsilon}\vcentcolon=a_{\max}^{\varepsilon}$.
We provide a discussion on the tightness of the upper bound $a_{\max}^{\varepsilon}$ in Appendix \ref{appendix:q_bound}.  
\response{
\begin{rem}
    While we assume that the fluid four-velocity is independent of time in this paper, this assumption is not a fundamental limitation of the proposed method.  
    Specifically, the bound on $|q|$ in Eq.~\eqref{eq:q_upperbound} is still valid for the case when the four-velocity is time dependent.  
    Thus, extension of the proposed method will require an approximation of $\p_{t}u_{\mu}$ to be included in the numerical flux in Eq.~\eqref{eq:energy_flux}.  
    This inclusion will increase $a_{\max}^{\varepsilon}$ (and $a^{\varepsilon}$), and potentially decrease the time step needed to preserve realizability of the cell average, given later in Proposition~\ref{prop:explicit_step_realizable}.  
\end{rem}
}

We use quadratures --- constructed by tensorization of one-dimensional quadratures --- to evaluate integrals over the multi-dimensional elements in Eq.~\eqref{eq:semi-discrete_DG_full_source}.
We denote the $N$-point Legendre--Gauss (LG) quadrature on $K_\varepsilon$ and $K^i$, respectively, by the points $S^{N}_\varepsilon(\bK)=\{\varepsilon_1,\ldots,\varepsilon_{N}\}$ and $S^{N}_i(\bK)=\{x^i_1,\ldots,x^i_{N}\}$.
Then the set of local DG nodes in the element $\bK$ is denoted as
\begin{equation}
    S_{\otimes}(\bK)=S^{N}_{\varepsilon}(\bK)\otimes\big(\bigotimes\nolimits_{i=1}^d S^{N}_i(\bK)\big).
\end{equation}
Also, the analysis of the realizability-preserving property in Section~\ref{sec:RealizabilityProof} uses specific quadrature rules.  
We denote the $M_\varepsilon$-point and $M_x$-point Legendre--Gauss--Lobatto (LGL) quadrature rules on the intervals $K_{\varepsilon}$ and $K^{i}$, respectively, by the points $\hat{S}^{M_\varepsilon}_\varepsilon(\bK)=\{\varepsilon_{\Lo}^{+}=\hat{\varepsilon}_{1},\hat{\varepsilon}_{2},\ldots,\hat{\varepsilon}_{M_\varepsilon}=\varepsilon_{\Hi}^{-}\}$ and $\hat{S}^{M_x}_{i}(\bK)=\{x_{\Lo}^{i,+}=\hat{x}^{i}_{1},\hat{x}_{2}^{i},\ldots,\hat{x}^{i}_{M_x}=x_{\Hi}^{i,-}\}$, with respective weights $\{\hat{w}^\varepsilon_q\}_{q=1}^{M_\varepsilon}$ and $\{\hat{w}^i_q\}_{q=1}^{M_x}$.
Here $M_\varepsilon\geq \frac{k+5}{2}$ so that the LGL quadrature rules integrate polynomials of degree $k+2$ or less exactly, and $M_x\geq \frac{k+3}{2}$ so that the LGL quadrature rules integrate polynomials of degree $k$ or less exactly.
Exact integration is required for the realizability-preserving analysis.
In an element $\bK$ we define the auxiliary sets
\begin{subequations}
\begin{align}
\label{eq:auxiliary_quadrature}
    \hat{S}_{\varepsilon,\otimes}(\bK)
    &=
    \hat{S}^{M_\varepsilon}_{\varepsilon}(\bK)\otimes\big(\bigotimes\nolimits_{i=1}^d S^{N}_i(\bK)\big),\\
    \hat{S}_{i,\otimes}(\bK)
    &=
    S^{N}_{\varepsilon}(\bK)\otimes\big(\bigotimes\nolimits_{j=1,j\neq i}^d S^{N}_j(\bK)\big)\otimes\hat{S}^{M_x}_i(\bK).
\end{align}
\end{subequations}
The union of the auxiliary sets in an element ${\bK}$ is denoted as
\begin{equation}
    \hat{S}_{\otimes}(\bK)
    =
    \hat{S}_{\varepsilon,\otimes}(\bK)\cup\big( 
\bigcup\nolimits_{i=1}^d \hat{S}_{i,\otimes}(\bK) \big),
\end{equation}
and the union of the auxiliary sets and the local DG nodes is denoted as
\begin{equation}
    \widetilde{S}_{\otimes}(\bK)
    =
    S_{\otimes}(\bK)\cup\hat{S}_{\otimes}(\bK).
\end{equation}

\subsection{Time Integration}
\label{sec:imex}

The semi-discrete DG problem in Eq.~\eqref{eq:semi-discrete_DG_full_source} can be rearranged into a system of ODEs of the form
\begin{equation}
    \dot{\bu} = \bT(\bu)+\bC(\bu),
    \label{eq:odeSystem}
\end{equation}
where $\bu = \{ \frac{1}{|\bK|}\int_{\bK}\bU_h\phi_h\,d\vect{z}: \bK\in\cT \}$ represents all the evolved degrees of freedom, $\bT$ is the transport operator corresponding to the terms in Eq.~\eqref{eq:semi-discrete_DG_full_source} arising from the spatial and energy derivatives, and $\bC$ is the collision operator corresponding to the right-hand side of Eq.~\eqref{eq:semi-discrete_DG_full_source}.
We use IMEX Runge--Kutta (RK) methods to evolve the degrees of freedom forward in time, and integrate the transport operator explicitly, and the collision operator implicitly.  
Diagonally implicit $s$-stage IMEX methods for the system of ODEs in Eq.~\eqref{eq:odeSystem} can be written generally as \cite{pareschiRusso_2005}
\begin{align}
    \bu^{(i)}
    &=
    \bu^{n} + 
    \Delta t \sum_{j=1}^{i-1}\tilde{a}_{ij}\,\bT(\bu^{(j)}) + 
    \Delta t \sum_{j=1}^i a_{ij}\,\bC(\bu^{(j)}),
    \quad i=1,\ldots,s \label{eq:IMEX_ith_stage}\\
    \bu^{n+1}
    &=
    \bu^{n} + 
    \Delta t \sum_{i=1}^{s}\tilde{w}_i\,\bT(\bu^{(i)}) + 
    \Delta t \sum_{i=1}^s w_i\,\bC(\bu^{(i)}). \label{eq:IMEX_assembly}
\end{align}
Here $\tilde{A}=\begin{pmatrix}\tilde{a}_{ij}\end{pmatrix}$ and $A=\begin{pmatrix}a_{ij}\end{pmatrix}$ are the $s\times s$ stage coefficient matrices for the respective explicit and implicit parts of the IMEX scheme.
While the $s$-vectors $\tilde{\bw}=(\tilde{w}_i)$ and $\bw=(w_i)$ are the respective coefficients for the assembly stage of the explicit and implicit parts.  
These matrices and vectors are subject to order conditions, which can be found in  \cite[Section 2.1]{pareschiRusso_2005}.

Importantly, for the purpose of proving the realizability-preserving property of our method, one hopes that the stage equations in Eq.~\eqref{eq:IMEX_ith_stage} can be expressed in the Shu--Osher form
\begin{align}
    \bu^{(0)} &= \bu^{n}\\
    \bu^{(i)} &= \sum_{j=0}^{i-1} c_{ij}\bu^{(ij)} + a_{ii}\,\Delta t\, \bC(\bu^{(i)}),
    \quad
    i=1,\ldots,s,
    \label{eq:stageq_ShuOsher}
\end{align}
where the coefficients satisfy $c_{ij}\geq 0$ and $\sum_{j=0}^{i-1} c_{ij} = 1$, and $\bu^{(ij)}$ is given by the forward Euler step
\begin{equation}
    \bu^{(ij)} = \bu^{(j)}+\hat{c}_{ij}\,\Delta t\,\bT(\bu^{(j)}),
\end{equation}
where the parameters $\hat{c}_{ij}\geq 0$. 
Details on determining $c_{ij}$ and $\hat{c}_{ij}$ for IMEX schemes can be found in \cite{chu_etal_2019,hu_etal_2018}, and for explicit RK methods in \cite{shuOsher_1988}.  
The Shu--Osher form of the stage equations in Eq.~\eqref{eq:stageq_ShuOsher} allows us to express $\bu^{(i)}$ as a convex combination of forward Euler steps with step sizes $\hat{c}_{ij}\Delta t$.  
Then to show the realizability-preserving property of the updates in Eq.~\eqref{eq:stageq_ShuOsher}, one only needs to consider a sequence of forward and backward Euler steps.  
The other advantage of the Shu--Osher form is the derived time-step restriction.  
If $\Delta t\leq\Delta t_{\mathrm{Ex}}$ is the time-step restriction to maintain realizability of the cell average in the Forward Euler step, then, assuming the implicit solve in Eq.~\eqref{eq:stageq_ShuOsher} does not impose a restriction on $\Delta t$, the time-step restriction to maintain realizability for the IMEX method is $\Delta t \leq \hat{c}\,\Delta t_{\mathrm{Ex}}$, where
\begin{equation}
    \hat{c}=\min_{i,j} \frac{1}{\hat{c}_{ij}}.
    \label{eq:timestep_c_hat}
\end{equation}
It is then ideal for the IMEX method to have $\hat{c}$ as close to $1$ as possible. 

In this paper, we consider $s$-stage IMEX methods that are diagonally implicit ($a_{ij}=0$ for $j>i$) and globally stiffly accurate (GSA), i.e., $\tilde{a}_{si}=\tilde{w}_{i}$ and $a_{si}=w_{i}$, so that the assembly step in Eq.~\eqref{eq:IMEX_assembly} can be omitted and $\bu^{n+1}=\bu^{(s)}$.  
Specifically, we use the IMEX PD-ARS \cite{chu_etal_2019} scheme, whose Butcher tableau is
\begin{equation}
    \begin{array}{c|c}
         \tilde{\bc} &\tilde{A}  \\
         \hline
         & \tilde{\bw}
    \end{array}
    =
    \begin{array}{c|ccc}
         0 &0 &0 &0 \\
         1 &1 &0 &0 \\
         1 &1/2 &1/2 &0 \\
         \hline
         &1/2 &1/2 &0
    \end{array}
    \quad
    \begin{array}{c|c}
         \bc& A  \\
         \hline
         & \bw
    \end{array}
    =
    \begin{array}{c|ccc}
         0 &0 &0 &0  \\
         1 &0 &1 &0\\
         1 &0 &1/2 &1/2\\
         \hline
         &0 &1/2 &1/2
    \end{array}
    \label{eq:IMEXPDARS_table}
\end{equation}
where the left table represents the coefficients for the explicit method, and the right table represents the coefficients for the implicit method.
The vectors $\tilde{\bc}$ and $\bc$ are used for the treatment of non-autonomous systems.
The IMEX PD-ARS scheme is formally only first order accurate, but performs well in the diffusion limit, and recovers the optimal second-order explicit RK scheme from \cite{shuOsher_1988} in the streaming limit ($\bC = 0$).
The Shu--Osher coefficients for the IMEX PD-ARS method are
\begin{subequations}
\label{eq:IMEXPDARS_ShuOsher}
\begin{align}
    \begin{bmatrix}
        c_{10} &\\
        c_{20} &c_{21}
    \end{bmatrix}
    &=
    \begin{bmatrix}
        1 &\\
        1/2 &1/2
    \end{bmatrix}\\
    \begin{bmatrix}
        \hat{c}_{10} &\\
        \hat{c}_{20} &\hat{c}_{21}
    \end{bmatrix}
    &=
    \begin{bmatrix}
        1 &\\
        0 &1
    \end{bmatrix}.
\end{align}
\end{subequations}

For problems with no collisions, $\bC = 0$, we set $a_{ii}=0$, and use explicit strong stability-preserving (SSP) RK methods.  
Specifically, we use the second and third order SSP-RK methods from \cite{shuOsher_1988} (SSPRK$2$ and SSPRK$3$, respectively).  
The Shu--Osher coefficients for SSPRK$2$ are identical to Eq.~\eqref{eq:IMEXPDARS_ShuOsher}, while for the SSPRK$3$ method they are
\begin{subequations}
\label{eq:SSPRK3_ShuOsher}
\begin{align}
    \begin{bmatrix}
        c_{10} & &\\
        c_{20} &c_{21} &\\
        c_{30} &c_{31} &c_{32}
    \end{bmatrix}
    &=
    \begin{bmatrix}
        1 & &\\
        3/4 &1/4 &\\
        1/3 &0 &2/3
    \end{bmatrix}\\
    \begin{bmatrix}
        \hat{c}_{10} & &\\
        \hat{c}_{20} &\hat{c}_{21} &\\
        \hat{c}_{30} &\hat{c}_{31} &\hat{c}_{32}
    \end{bmatrix}
    &=
    \begin{bmatrix}
        1 & &\\
        0 &1 &\\
        0 &0 &1
    \end{bmatrix}.
\end{align}
\end{subequations}
Note for each method we have $\hat{c}=1$.
As we will show in Section \ref{sec:CollisionAnalysis} that our implicit collision solver does not impose a restriction on the time step $\Delta t$, then our derived time step will be no different than if we were solving a collisionless problem with forward Euler.  

\section{Iterative Solvers}
\label{sec:iterativeSolvers}

The proposed DG-IMEX scheme requires two nonlinear solvers: (a) for the recovery of primitive moments $\bM$ from the conserved moments $\bU$, and (b) for the implicit collision solver, which is formulated directly on the primitive moments.  
Both solvers involve element-local data and are formulated in terms of point values, where input data is provided from pointwise evaluations of the polynomial representation; i.e., $\bU=\bU_{h}(\vect{z})$ for some $\vect{z}\in\bK$.  
Both solvers also require the components of the three-velocity, $\vect{v}=(v^{1},v^{2},v^{3})^{\intercal}$, as input.   

\subsection{Conserved to Primitive Moment Conversion}
\label{sec:moment_conversion}

The recovery of the Lagrangian moments $\bM=(\cJ,\cH_{i})^{\intercal}$ from the evolved Eulerian moments $\bU=(\cE,\cF_{j})^{\intercal}$ consists of two steps.  
First, the ``hat" moments $\widehat{\bU}=(\widehat{\cE},\widehat{\cF}_j)^{\intercal}$, defined in Eq.~\eqref{eq:hatEF}, are obtained from $\bU$.  
These moments are linearly related by
\begin{equation}\label{eq:Hat_to_Conserved}
    \cE
    = W\widehat{\cE}+v^{i}\widehat{\cF}_{i}
    \quad\text{and}\quad
    \cF_{i}
    = \widehat{\cF}_{i}+Wv_{i}\widehat{\cE},
\end{equation}
which can be easily inverted to give
\begin{equation}\label{eq:Conserved_to_Hat}
    \widehat{\cE}
    = W\big(\cE-v^{k}\cF_{k}\big)
    \quad\text{and}\quad
    \widehat{\cF}_{i}
    =\cF_{i}-W^{2}v_{i}\big(\cE-v^{k}\cF_{k}\big).
\end{equation}
For use later, we write Eq.~\eqref{eq:Hat_to_Conserved} as
\begin{equation}
    \bU = A(\vect{v})\,\widehat{\bU},
    \quad\text{where}\quad
    A(\vect{v})=
    \left(\begin{array}{cc} 
        W & \vect{v}^{\intercal} \\ 
        W\vect{v} & I 
    \end{array}\right),
    \label{eq:Hat_to_Conserved_Matrix}
\end{equation}
where $I$ is the $3\times3$ identity matrix.  

Next, with $\widehat{\bU}$ known, the Lagrangian moments $\bM$ are obtained through an iterative procedure.  
The moments are nonlinearly related by
\begin{equation}\label{eq:inverse_problem}
    \widehat{\cE} = W\cJ+v^{i}\cH_{i}
    \quad\text{and}\quad
    \widehat{\cF}_{j} = W\cH_{j}+v^{i}\cK_{ij}(\bM).  
\end{equation}
To solve this system for the Lagrangian moments, we adopt the idea from \cite{laiu_etal_2025}, which is based on Richardson iteration for linear systems, and formulate the following fixed-point problem
\begin{align}
    \bM 
    = \left(\begin{array}{c} \cJ \\ \cH_{j} \end{array}\right)
    &= \left(\begin{array}{c} \cJ \\ \cH_{j} \end{array}\right)
    -\frac{\lambda}{W}\,\left(\begin{array}{c}
        (W\cJ+v^{i}\cH_{i}) - \widehat{\cE} \\
        (W\cH_{j}+v^{i}\cK_{ij}) - \widehat{\cF}_{j}
    \end{array}\right) \nonumber\\
    &=\bM-\frac{\lambda}{W}\,\bF_{\widehat{\bU}}(\bM)
    \vcentcolon=\bH_{\widehat{\bU}}(\bM),
    \label{eq:con2prim_fixed_point}
\end{align}
where $\lambda/W$ is a constant ``step size'', and $\lambda$ is taken to be
\begin{equation}
    \lambda = \frac{1}{1+v}, \qquad v=\sqrt{v_{i}v^{i}}.
\end{equation}
The Lagrangian moments are then obtained with Picard iteration,
\begin{equation}
    \bM^{[k+1]} = \bH_{\widehat{\bU}}(\bM^{[k]}),
    \quad k=0,1,\ldots,
    \label{eq:Richardson_update}
\end{equation}
with initial guess $\bM^{[0]}=\widehat{\bU}/W$.  
We consider the method to have converged when the residual satisfies
\begin{equation}
\label{eq:convergence_criteria}
    \|\bF_{\widehat{\bU}}(\bM^{[k]})\|
    \leq \texttt{tol}_{\rm C2P},
\end{equation}
where $\|\cdot\|$ is the Euclidean norm and $\texttt{tol}_{\rm C2P}$ is a user-specified tolerance.  
\begin{rem}\label{rem:convergence}
    Convergence of the iterates generated by Eq.~\eqref{eq:Richardson_update} to a unique solution is guaranteed for $v < 0.221075$ (see Appendix \ref{appendix:convergence}), however, in our numerical tests we have not encountered a velocity for which the fixed-point iteration fails to converge when using $\lambda = (1+v)^{-1}$ (see Figure \ref{fig:iteration_count_comparison}).
    The convergence analysis requires that, when $\bM^{[0]}\in\cR$, each iterate $\bM^{[k]}$ remains realizable. This property is proved later in Lemma~\ref{lem:momentconversion_realizable}.
\end{rem}

The conserved to primitive moment conversion problem can also be solved by Newton's method.  
To this end, we write Eq.~\eqref{eq:inverse_problem} as
\begin{equation}
    \bF_{\widehat{\bU}}(\bM)=0,
\end{equation}
The Newton update step is then given by
\begin{equation}
    \bJ(\bM^{[k]})\,\big(\bM^{[k+1]}-\bM^{[k]}\big)
    =-\bF_{\widehat{\bU}}(\bM^{[k]}),
    \quad k=0,1,\ldots,
\end{equation}
where $\bJ=[\partial\bF_{\widehat{\bU}}/\partial\bM]$ is the Jacobian matrix, and $\bM^{[0]}=\widehat{\bU}/W$ is used as the initial guess.  
We also consider Newton's method to have converged when Eq.~\eqref{eq:convergence_criteria} is satisfied.  
Both the modified Richardson iteration method and Newton’s method have been implemented and are studied later in Section \ref{sec:MomentConverstionTest}.

\subsection{Collision Solver}
\label{sec:collision_solver}

With collisional source terms, a nonlinear solve is required when performing the implicit update during time integration.  
The element-local nodal implicit update of the IMEX scheme in Eq.~\eqref{eq:stageq_ShuOsher} can be formulated as a backward Euler solve with step size $\Delta \tau$,
\begin{equation}
    \bU = \bU^{(*)} + \Delta \tau\,\bcC(\bU),
    \label{eq:backwardEulerConserved}
\end{equation}
where $\bU^{(*)}$ represents the first (known) term on right-hand side of Eq.~\eqref{eq:stageq_ShuOsher}, and $\bU$ represents the unknowns to be solved for.  
(For notational convenience we omit the superscript $(i)$ on the unknowns.)
We formulate the implicit solve in terms of the primitive moments.  
To this end, using the matrix $A$ defined in Eq.~\eqref{eq:Hat_to_Conserved_Matrix}, we can write the collision term as
\begin{equation}
    \bcC(\bU) = A(\vect{v})\,\widehat{\bcC}(\bM)
    \quad\text{where}\quad
    \widehat{\bcC}(\bM)
    =\left(\begin{array}{c}
        \chi\,(\cJ_{\Equilibrium}-\cJ) \\
        -\kappa\,\cH_{j}
    \end{array}\right),
    \label{eq:conservedCtoHatC}
\end{equation}
so that multiplying Eq.~\eqref{eq:backwardEulerConserved} by $A^{-1}$ on both sides leads to 
\begin{equation}
    \widehat{\bU}(\bM)=\widehat{\bU}^{(*)}
    +\Delta\tau\,\widehat{\bcC}(\bM),
    \label{eq:backwardEulerHat}
\end{equation}
where $\widehat{\bU}^{(*)}=A^{-1}\bU^{(*)}=(\widehat{\cE}^{(*)},\widehat{\cF}_{j}^{(*)})^{\intercal}$.  
Next, we formulate a fixed-point method to solve Eq.~\eqref{eq:backwardEulerHat}.  
With $\bM$ known, $\bcC(\bU)$ is evaluated and used to update the conserved moments in Eq.~\eqref{eq:backwardEulerConserved}.  

Similar to Eq.~\eqref{eq:con2prim_fixed_point}, and following \cite{laiu_etal_2025}, we use the Richardson iteration idea to formulate the fixed-point problem 
\begin{align}\label{eq:collision_fixed_point}
    \bM &= \bM 
    - \lambda\,\Lambda\,
    \big[\,\widehat{\bU}(\bM)-\big(\,\widehat{\bU}^{(*)}+\Delta\tau\,\widehat{\bcC}(\bM)\,\big)\,\big] \nonumber\\
    &=\bM 
    - \lambda\,\Lambda\,\bG_{\widehat{\bU}}(\bM) =\vcentcolon \bQ_{\widehat{\bU}}(\bM),
\end{align}
where $\Lambda=\mbox{diag}(\mu_{\chi},\mu_{\kappa},\mu_{\kappa},\mu_{\kappa})$, $\mu_{\chi}=(W+\lambda\,\Delta\tau\,\chi)^{-1}$, $\mu_{\kappa}=(W+\lambda\,\Delta\tau\,\kappa)^{-1}$, and $\lambda\in(0,1]$ is a step size parameter.  
In this paper, for reasons detailed in Section~\ref{sec:CollisionAnalysis}, we propose $\lambda=1/(1+v)$.  
\begin{rem}
    The fixed-point problem in Eq.~\eqref{eq:collision_fixed_point} reduces to Eq.~\eqref{eq:con2prim_fixed_point} when $\chi=\kappa=0$; i.e., when $\bcC(\bU)=0$.
\end{rem}
\noindent
The primitive moments are then obtained through Picard iteration, 
\begin{equation}
    \bM^{[k+1]} = \bQ_{\widehat{\bU}}(\bM^{[k]}),
    \quad k=0,1,\ldots,
    \label{eq:collision_richardson}
\end{equation}
with initial guess $\bM^{[0]}=\widehat{\bU}^{(*)}$.  
Iterations continue until the residual satisfies
\begin{equation}
    \|\bG_{\widehat{\bU}}(\bM^{[k]})\|
    \leq \texttt{tol}_{\rm Coll},
    \label{eq:convergence_criteria_collisions}
\end{equation}
where, as in Eq.~\eqref{eq:convergence_criteria}, we use the Euclidean norm, and $\texttt{tol}_{\rm Coll}$ is a user-specified tolerance.  

\section{Realizability-Preserving Property of the DG-IMEX Scheme}
\label{sec:RealizabilityProof}

The realizability-preserving scheme is designed to preserve realizability of cell averages during time integration.  
The realizability of cell averages is then leveraged to recover pointwise realizability within each element with the aid of a realizability-enforcing limiter, which is detailed below in Section~\ref{sec:limiter}. 

The cell average of the moments is defined as
\begin{equation}
    \bU_{\bK}=\frac{1}{|\bK|}\int_{\bK}\bU_h\,\varepsilon^{2}\,d\vect{z}.
\end{equation}
Recall the evolved degrees of freedom are $\bu = \{ \frac{1}{|\bK|}\int_{\bK}\bU_h\phi_h\,\varepsilon^{2}\,d\vect{z}: \bK\in\cT \}$. 
For $\phi_{h}=1$, the evolved degrees of freedom represent the cell averaged moments.  
Using Eq.~\eqref{eq:stageq_ShuOsher}, the stage equation for the cell average of the moments is then
\begin{equation}
\label{eq:cell_avg_stage_equation}
    \bU_{\bK}^{(i)}
    =
    \sum_{j=0}^{i-1}c_{ij}\bU_{\bK}^{(ij)}
    +a_{ii}\,\Delta t\,\bC(\bU_{\bK}^{(i)}), 
    \quad i=1,\ldots,s,
\end{equation}
where
\begin{equation} 
\label{eq:U_K^ij}
    \bU_{\bK}^{(ij)}
    =\bU_{\bK}^{(j)}
    +\hat{c}_{ij}\,\Delta t\,\bT(\bU_{h}^{(j)})_{\bK},
\end{equation}
and (see Eq.~\eqref{eq:semi-discrete_DG_full_source})
\begin{align}
    \bT(\bU_{h}^{(j)})_{\bK}
    =& 
    -\frac{1}{|\bK|}\int_{\bK_{\vect{x}}}
    \big[\,
        \varepsilon^3\,\widehat{\bF}^\varepsilon\rvert_{\varepsilon_{\Hi}} 
        -\varepsilon^3\,\widehat{\bF}^\varepsilon\rvert_{\varepsilon_{\Lo}}
    \,\big]\,d\vect{x} \nonumber \\
    &-\sum_{i=1}^{d}
    \frac{1}{|\bK|}\int_{\widetilde{\bK}^i}
    \big[\,
        \widehat{\bF}^i\rvert_{x^i_{\Hi}}
        -\widehat{\bF}^i\rvert_{x^i_{\Lo}}
    \,\big]\,\varepsilon^2\,d\Tilde{\vect{z}}^i.
\end{align}
Eq.~\eqref{eq:cell_avg_stage_equation} can be separated into an explicit update and an implicit update.
We define the explicit update as
\begin{equation}
\label{eq:ExplicitUpdate}
    \bU^{(i)}_{\mathrm{Ex}}
    =
    \sum_{j=0}^{i-1}c_{ij}\bU_{\bK}^{(ij)},
\end{equation}
so that the implicit update is
\begin{equation}
    \label{eq:ImplicitUpdate}
    \bU_{\bK}^{(i)}
    =
    \bU^{(i)}_{\mathrm{Ex}}+a_{ii}\Delta t\bC(\bU_{\bK}^{(i)}), 
    \quad i=1,\ldots,s.
\end{equation}

\begin{prop}
\label{prop:explicit_step_realizable}
    Consider the stages of the IMEX scheme in Eq.~\eqref{eq:stageq_ShuOsher} applied to the DG discretization of the evolution equations in Eq.~\eqref{eq:semi-discrete_DG_full_source}.
    Assume that 
    \begin{enumerate}
        \item For all $\ell=1,\ldots,d$, the LGL quadrature points, $\hat{S}^{M_x}_\ell(\bK)$, are chosen such that it is exact for computing the cell average of $\bU_h^{(j)}$ over $K^\ell$. 
        \item The LGL quadrature points, $\hat{S}^{M_\varepsilon}_{\varepsilon}(\bK)$, are chosen such that it is exact for computing the cell average of $\bU_h^{(j)}$ over $K_\varepsilon$.
        \item For all $0\leq j\leq i-1<s$, the values $\bU_h^{(j)}(\vect{x})\in\cR$ for all $\vect{x}\in\widetilde{S}_{\otimes}(\bK)$.
        \item The time step $\Delta t$ is chosen such that 
        \begin{equation}\label{eq:realizable_timestep}
            \Delta t 
            \leq 
            \hat{c}\min_{K_\varepsilon,K^\ell}
            \left\{\, 
                W(1-v)\frac{\hat{w}_{M_\varepsilon}^\varepsilon\Delta\varepsilon}{(d+1)\varepsilon_{\Hi} a^\varepsilon},\, 
                \frac{\hat{w}_{M_x}^\ell\Delta x^\ell}{(d+1)a^\ell}
            \,\right\},
        \end{equation}
        with $a^\varepsilon\geq|q|$, $a^\ell\geq 1$, and $\hat{c}$ is defined in Eq.~\eqref{eq:timestep_c_hat}.
    \end{enumerate}
    Then $\bU^{(i)}_{\mathrm{Ex}}\in\cR$.
\end{prop}
As a consequence of Proposition \ref{prop:explicit_step_realizable}, we have that the implicit update given by Eq.~\eqref{eq:ImplicitUpdate} is realizable.
\begin{prop}
\label{prop:realizable_stage_values}
    Consider the stages of the IMEX scheme in Eq.~\eqref{eq:stageq_ShuOsher} applied to the DG discretization of the evolution equations in Eq.~\eqref{eq:semi-discrete_DG_full_source}.  
    Assume that the conditions of Proposition \ref{prop:explicit_step_realizable} hold so that $\bU^{(i)}_{\mathrm{Ex}}\in\cR$.  
    Let $\bU_{\bK}^{(i)}$ in Eq.~\eqref{eq:ImplicitUpdate} be obtained by the modified Richardson iteration described by Eq.~\eqref{eq:collision_richardson} with $\lambda \leq (1+v)^{-1}$.  
    Then $\bU_{\bK}^{(i)}\in\cR$.
\end{prop}
Finally using Propositions \ref{prop:explicit_step_realizable} and \ref{prop:realizable_stage_values} in conjunction with the realizability-enforcing limiter described in Section~\ref{sec:limiter} we arrive at our main result.
\begin{theorem}
\label{thm:realizable_scheme}
    Consider the stages of the IMEX scheme in Eq.~\eqref{eq:stageq_ShuOsher} applied to the DG discretization of the evolution equations in Eq.~\eqref{eq:semi-discrete_DG_full_source}.  Assume that
    \begin{enumerate}
        \item Conditions 1, 2, 4 of Proposition \ref{prop:explicit_step_realizable} hold, and Condition~3 of Proposition~\ref{prop:explicit_step_realizable} holds for $i=1$.  
        \item With $\bU_{\bK}^{(i)}\in\cR$, the realizability-enforcing limiter is invoked to enforce $\bU_h^{(i)}(\vect{x})\in\cR$ for all $\vect{x}\in\widetilde{S}_{\otimes}(\bK)$.
        \item In the $s$-stage IMEX scheme $\bu^{n+1} = \bu^{(s)}$.
    \end{enumerate}
    Then $\bU_{\bK}^{n+1}\in\cR$.
\end{theorem}

\begin{figure}
    \centering
    \includegraphics[width = 8.6cm]{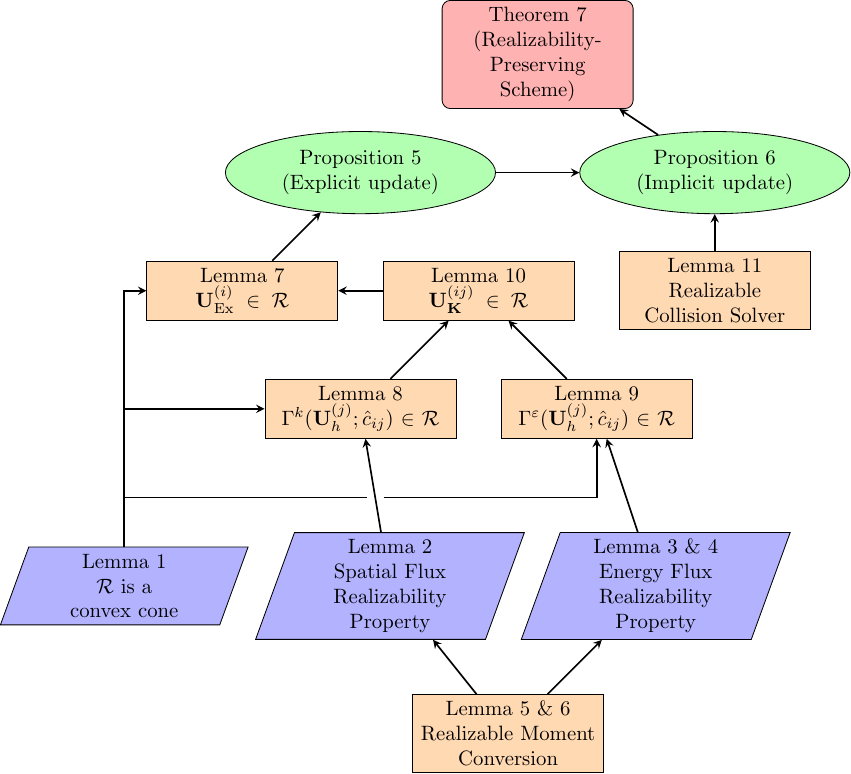}
    \caption{Flowchart depicting how the Lemmas introduced in Section~\ref{sec:RealizabilityProof} build upon each other to prove Theorem~\ref{thm:realizable_scheme}.
    Realizability of the moment conversion process guarantees realizability-preserving properties that can be derived from our numerical fluxes introduced by the DG discretization.
    These properties, along with the set $\cR$ being a convex cone, ensure the explicit update of the conserved moments maintains realizability.
    Then the realizability-preserving collision solver guarantees the implicit update maintains realizability.
    Finally, this allows us to conclude that updating $\bU_{\bK}^{n}$ to $\bU_{\bK}^{n+1}$ preserves the realizability of the conserved moments.
    }
    \label{fig:proof_flowchart}
\end{figure}

The remainder of Section~\ref{sec:RealizabilityProof} is devoted to proving Propositions~\ref{prop:explicit_step_realizable} and \ref{prop:realizable_stage_values} and Theorem~\ref{thm:realizable_scheme}.
We encourage the reader to refer to the flowchart in Figure \ref{fig:proof_flowchart}, which aims to explain how various lemmas are combined to prove Propositions~\ref{prop:explicit_step_realizable} and \ref{prop:realizable_stage_values} and Theorem~\ref{thm:realizable_scheme}.

In the following analysis, as in \cite[Assumption~1]{laiu_etal_2025}, we employ the exact closure assumption; i.e., given the lower-order primitive moments $\{\cJ,\cH^{\mu}\}$, the higher-order primitive moments $\{\cK^{\mu\nu},\cL^{\mu\nu\rho}\}$ are computed such that $\{\cJ,\cH^{\mu},\cK^{\mu\nu},\cL^{\mu\nu\rho}\}$ satisfy Eqs.~\eqref{eq:lagrangianMoments} and \eqref{eq:lagrangianRankThreeMoment} for some nonnegative distribution $f$.  

\subsection{Preparations for Analysis of Numerical Method}

Our choice of numerical fluxes is motivated by the goal of designing a realizability-preserving numerical scheme for the two-moment model.
The following lemma is useful when proving that our choice of spatial numerical fluxes in Eq.~\eqref{eq:spatial_flux} will preserve moment realizability under a CFL-type restriction on the time step.
\begin{lemma}
\label{lem:realizableflux} 
    For a given $i\in\{1,\ldots,d\}$, let $\bU=(\cE,\cF_j)^{\intercal}$ and $\bF^i=(\cF^i,\cS^i_{\hspace{4pt}j})^{\intercal}$, as defined in Eq.~\eqref{eq:momentsCompact}, be moments of a distribution function $f\in\mathfrak{R}$. 
    Let $\alpha$ be such that $0<\alpha^{-1}\leq 1$, then
    $\Phi^{i,\pm}(\bU)\vcentcolon=\frac{1}{2}(\bU\pm\frac{1}{\alpha}\bF^i)\in\cR$.
\end{lemma}
\begin{proof}
    The first component of $\Phi^{i,\pm}(\bU)$ is
    \begin{align*}
        \frac{1}{2}(\cE\pm\frac{1}{\alpha}\cF^i)
        =
        \frac{1}{4\pi}\int_{\mathbb{S}^2}g\,d\omega,
    \end{align*}
    where we have defined $g:  =\frac{1}{2}(1\pm L^i/\alpha)f$, where $L^i$ is given by Eq.~\eqref{eq:EulerianUnitDirection}.
    Similarly, the second component of $\Phi^{i,\pm}(\bU)$ is
    \begin{align*}
        \frac{1}{2}(\cF_j\pm\frac{1}{\alpha}\cS^i_{\hspace{4pt}j})
        &=
        \frac{1}{4\pi}\int_{\mathbb{S}^2}g L_j\,d\omega.
    \end{align*}    
    Since $|L^{i}|\le1$ and $0<\alpha^{-1}\leq1$, $g\in\mathfrak{R}$.
    Using the same reasoning as in the proof of Proposition~\ref{prop:momentBoundsEulerian}, it follows that $\Phi^{i,\pm}(\bU)\in\cR$.
\end{proof}

The following two lemmas will be useful when proving that our choice of numerical flux in the energy dimension in Eq.~\eqref{eq:energy_flux} will result in a realizability-preserving scheme under a CFL condition.
\begin{lemma}
\label{lem:energy_lemma_1} 
    Given $f\in\mathfrak{R}$, let $\bU=(\cE,\cF_j)^{\intercal}$ and $\widetilde{\bU}=(\widetilde{\cE},\widetilde{\cF}_{j})$ be as defined in Eq.~\eqref{eq:momentsCompact} and \eqref{eq:Lagrangian_U_tilde}, respectively. Then, for $\alpha\geq0$ such that $0\leq\alpha a^\varepsilon\leq W+v^\mu\ell_\mu=E/\varepsilon$ (see Eq.~\eqref{eq:EulerianParticleEnergy}),
    $\bU-\alpha a^\varepsilon\widetilde{\bU}$ is realizable.  
\end{lemma}
\begin{proof}
    The first component of $\bU-\alpha a^\varepsilon\,\widetilde{\bU}$ is
    \begin{equation*}
        \frac{1}{4\pi}\int_{\bbS^2}(E^2/\varepsilon)f\,d\omega-\frac{1}{4\pi}\int_{\bbS^2}\alpha a^\varepsilon Ef\,d\omega
        =\frac{1}{4\pi}\int_{\bbS^2}g\,d\omega,
    \end{equation*}
    where we have defined $g\vcentcolon=(E^2/\varepsilon)(1-\alpha a^\varepsilon(\varepsilon/E))\,f$.
    Since $\alpha a^\varepsilon\leq E/\varepsilon$, $g\in\mathfrak{R}$.
    The second component of $\bU-\alpha a^\varepsilon\,\widetilde{\bU}$ is
    \begin{equation*}
        \frac{1}{4\pi}\int_{\bbS^2}(E^2/\varepsilon)fL_j\,d\omega-\frac{1}{4\pi}\int_{\bbS^2}\alpha a^\varepsilon EfL_j\,d\omega
        =\frac{1}{4\pi}\int_{\bbS^2}gL_j\,d\omega.
    \end{equation*}
    It follows from arguments similar to Proposition \ref{prop:momentBoundsEulerian} that $\bU-\alpha a^\varepsilon\widetilde{\bU}\in\cR$.  
\end{proof}
\begin{lemma}
\label{lem:energy_lemma_2}
    Given $f\in\mathfrak{R}$, let $\widetilde{\bU}$ and  $\bF^\varepsilon$ be given by the definitions in Eqs.~\eqref{eq:Lagrangian_U_tilde} and \eqref{eq:momentsCompact}, respectively.  
    Let $a_{\max}^{\varepsilon}\ge|q|$, where $q$ is defined in Eq.~\eqref{eq:q}.
    Then $\widetilde{\bU}\pm\frac{1}{a_{\max}^{\varepsilon}}\bF^\varepsilon\in\cR$. 
\end{lemma}
\begin{proof}
    The first component of $\widetilde{\bU}\pm\frac{1}{a_{\max}^{\varepsilon}}\bF^\varepsilon$ is
    \begin{equation*}
        \frac{1}{4\pi}\int_{\bbS^2}E\,f\,d\omega
        \pm\frac{1}{4\pi}\int_{\bbS^2}\frac{q}{a_{\max}^{\varepsilon}}\,E\,f\,d\omega
        =
        \frac{1}{4\pi}\int_{\bbS^2}g\,d\omega,
    \end{equation*}
    where we have defined $g\vcentcolon=E\,(1\pm\frac{q}{a_{\max}^{\varepsilon}})\,f$.
    Since $\frac{|q|}{a_{\max}^{\varepsilon}}\leq 1$, $g\in\mathfrak{R}$.  The second component of $\widetilde{\bU}\pm\frac{1}{a_{\max}^{\varepsilon}}\bF^\varepsilon$ is
    \begin{equation*}
        \frac{1}{4\pi}\int_{\bbS^2}E\,f\,L_j\,d\omega
        \pm\frac{1}{4\pi}\int_{\bbS^2}\frac{q}{a_{\max}^{\varepsilon}}E\,f\,L_j\,d\omega
        =
        \frac{1}{4\pi}\int_{\bbS^2}g\,L_j\,d\omega.
    \end{equation*}
    It follows that $\widetilde{\bU}\pm\frac{1}{a_{\max}^{\varepsilon}}\bF^\varepsilon\in\cR$.
\end{proof}

\subsection{Conserved to Primitive Moment Conversion}
\label{sec:MomentConversionAnalysis}

In this section we prove that the Picard iteration defined in Eq.~\eqref{eq:Richardson_update} for the moment conversion problem is realizability-preserving.  
We need the conversion to be realizability-preserving so that the numerical fluxes are evaluated using moments of a distribution $f\in\mathfrak{R}$ so that Lemmas~\ref{lem:realizableflux}, \ref{lem:energy_lemma_1}, and \ref{lem:energy_lemma_2} hold.  
The following two lemmas establish realizability-preserving conversion from conserved moments $\bU$ to primitive moments $\bM$.
\begin{lemma}\label{lem:realizable_hat_moments}
    Assume that $\bU=(\cE,\cF_j)^{\intercal}$, defined in Proposition~\ref{prop:momentBoundsEulerian}, are realizable.  
    Let the three-velocity be $v^\mu=(0,v^i)^{\intercal}$ such that $0\leq v_iv^i<1$.  
    Then $\widehat{\bU}=(\widehat{\cE},\widehat{\cF}_j)^{\intercal}$, defined in Proposition~\ref{prop:momentBoundsHat}, are realizable.
\end{lemma}
\begin{proof}
    To prove realizability of $(\widehat{\cE},\widehat{\cF}_j)^{\intercal}$, it is sufficient to show that $\widehat{\cE}>0$ and $\widehat{\cE}\geq\widehat{\cF}$, where $\widehat{\cF}=\sqrt{\widehat{\cF}_\mu\widehat{\cF}^\mu}$.  Since $(\cE,\cF_j)^{\intercal}$ are realizable, $\cE>0$ and $\cE\geq\cF$, where $\cF=\sqrt{\cF_{i}\cF^{i}}$.  
    Applying the Cauchy--Schwarz inequality to the left expression in Eq.~\eqref{eq:Conserved_to_Hat}, it is straightforward to show that $\widehat{\cE}>0$;
    \[
        \widehat{\cE}=W(\cE-v^k\cF_k)
        \geq W(\cE-\sqrt{v_iv^i}\sqrt{\cF_k\cF^k})
        >W(\cE-\cF)\geq 0.
    \]
    To show that $\widehat{\cE}\geq\widehat{\cF}$, we note two useful equalities.  
    First, since $u^\mu\widehat{\cF}_\mu=0$, $\widehat{\cF}_{0}=-v^{i}\widehat{\cF}_{i}$.  
    Then, using the right expression in Eq.~\eqref{eq:Hat_to_Conserved} gives
    \begin{equation*}
        \widehat{\cF}_0=Wv^iv_i\widehat{\cE}-v^i\cF_i. 
    \end{equation*}
    Second, using the definition of the Lorentz factor, we have
    \begin{equation*}
        1+W^2v_iv^i=1+\frac{1}{1-v_iv^i}v_iv^i=
        \frac{1}{1-v_iv^i}=W^2.
    \end{equation*}
    We then find
    \begin{align*}
        \widehat{\cF}^2
        &=
        -(\,Wv^iv_i\widehat{\cE}-v^i\cF_i\,)^2
        +(\cF_i-Wv_i\widehat{\cE})(\cF^i-Wv^i\widehat{\cE}) \\
        &=
        -(v^i\cF_i)^2+\cF^2+2W\widehat{\cE}v^i\cF_i(v_iv^i-1) \\
        &\hspace{12pt}+W^2\widehat{\cE}^2v_iv^i(1-v_iv^i)\\
        &=
        -(v^i\cF_i)^2+\cF^2-2\,\widehat{\cE}\,v^i\cF_i/W+\widehat{\cE}^2v_iv^i \\ 
        &\leq
        -(v^i\cF_i)^2+\cE^2-2\,\widehat{\cE}\,v^i\cF_i/W+\widehat{\cE}^2v_iv^i\\
        &=
        (\,\cE-v^i\cF_i\,)\,(\,\cE+v^i\cF_i\,)\,
        -2\,\widehat{\cE}v^i\cF_i/W+\widehat{\cE}^2v_iv^i\\
        &=
        \widehat{\cE}\,
        (\,\cE-v^i\cF_i\,)/W+\widehat{\cE}^2v_iv^i\\
        &=\widehat{\cE}^2\,(\,1+W^2v_iv^i\,)/W^{2}
        =\widehat{\cE}^2.
    \end{align*}
    Taking a square root on both sides gives the desired result.
\end{proof}

\begin{lemma}\label{lem:momentconversion_realizable}
    Let $\widehat{\bU} = (\widehat{\cE},\widehat{\cF}_i)^{\intercal}$ and $\bM^{[k]} = (\cJ^{[k]},\cH_i^{[k]})^{\intercal}$ be realizable.
    If $\lambda\leq (1+v)^{-1}$, then $\bM^{[k+1]} = (\cJ^{[k+1]},\cH_i^{[k+1]})^{\intercal}$, determined by the modified Richardson iteration scheme given by Eq.~\eqref{eq:Richardson_update}, is realizable.
\end{lemma}
\begin{proof}
    The first component of $\bM^{[k+1]}$ is
    \begin{equation*}
        \cJ^{[k]}
        = \mathscr{J}^{[k]} + \frac{\lambda}{W}\widehat{\cE},
    \end{equation*}
    where we have defined
    \begin{equation*}
        \mathscr{J}^{[k]}
        =
        (1-\lambda)\cJ^{[k]}-\frac{\lambda}{W} v^{j}\cH^{[k]}_{j}
        =
        \frac{1}{4\pi}
        \int_{\bbS^2}g^{[k]}\,d\omega,
    \end{equation*}
    and $g^{[k]}\vcentcolon=\big[\,(1-\lambda)-\frac{\lambda}{W}\,v^{j}\ell_{j}\,\big]\,\varepsilon f^{[k]}$.  
    Similarly, the second component of $\bM^{[k+1]}$ is
    \begin{equation*}
        \cH_i^{[k+1]}
        =\mathscr{H}_i^{[k]}
        +\frac{\lambda}{W}\widehat{\cF}_i,
    \end{equation*}
    where
    \begin{equation*}
        \mathscr{H}_i^{[k]}
        =
        (1-\lambda)\cH^{[k]}_i-\frac{\lambda}{W} v^j\cK^{[k]}_{ij}
        =
        \frac{1}{4\pi}\int_{\bbS^2}g^{[k]}\,\ell_i\,d\omega.
    \end{equation*}
    Note that, since $\ell_{0}=-v^{i}\ell_{i}$, we also have
    \begin{equation*}
        \mathscr{H}_{0}^{[k]}=-v^{i}\mathscr{H}_{i}^{[k]}
        =
        \frac{1}{4\pi}\int_{\bbS^2}
        g^{[k]}\,\ell_0\,d\omega.
    \end{equation*}
    We let $\mathscr{M}^{[k]} = (\mathscr{J}^{[k]},\mathscr{H}_{\mu}^{[k]})^{\intercal}$. 
    Since the first and second components of $\mathscr{M}^{[k]}$ are expressed as moments arising from the same distribution function, $g^{[k]}$, then
    $\mathscr{M}^{[k]}\in\cR$ follows from Proposition \ref{prop:momentBoundsLagrangian}, provided $g^{[k]}\in\mathfrak{R}$.  
    By assumption $\varepsilon f^{[k]}\in\mathfrak{R}$, in order to obtain $g^{[k]}\in\mathfrak{R}$, we require
    \[
        (1-\lambda)-\frac{\lambda}{W} v^i\ell_i\geq 0.
    \]
    Since $v^i\ell_i=\ell^0=Wv_{\hat{\iota}}\ell^{\hat{\iota}}$, 
    \[
        -Wv\leq v^i\ell_i\leq Wv.
    \]
    Then since $(1-\lambda)-\frac{\lambda}{W} v^i\ell_i
    \geq (1-\lambda)-\lambda v$, requiring the latter to be greater than $0$ yields
    \begin{equation*}
        \lambda\leq\frac{1}{1+v}.
    \end{equation*}
    Since $\mathscr{M}^{[k]},\widehat{\bU}\in\cR$, realizability of $\bM^{[k+1]}=\mathscr{M}^{[k]}+\frac{\lambda}{W}\widehat{\bU}$ follows from Lemma~\ref{lem:convexcone}.
\end{proof}

\begin{rem}
    The mapping in Eq.~\eqref{eq:Richardson_update} is formulated for an effective step size $\lambda'=\lambda/W$.  
    Given the result of Lemma~\ref{lem:momentconversion_realizable}, the Picard iteration is realizability-preserving when $\lambda'\le\frac{1}{W(1+v)}$.
\end{rem}

Following the realizability-preserving result in Lemma~\ref{lem:momentconversion_realizable}, the convergence of the modified Richardson iteration scheme in Eq.~\eqref{eq:Richardson_update} is proved in Appendix~\ref{appendix:convergence} under conditions detailed in Remark~\ref{rem:convergence}, which concludes the realizability-preserving property analysis for the conserved to primitive moment conversion step.

\subsection{Explicit Update}

Here we establish conditions for which the explicit update $\bU^{(i)}_{\mathrm{Ex}}$ in Eqs.~\eqref{eq:U_K^ij}-\eqref{eq:ExplicitUpdate} is realizable.
\begin{lemma}
\label{lem:Ui_realizable}
    Assume that $\bU^{(ij)}_{\bK}\in\cR$ for all i and $j\leq i-1$.
    Then $\bU^{(i)}_{\mathrm{Ex}}\in\cR$ for all i.
\end{lemma}
\begin{proof}
    The proof is a consequence of Lemma \ref{lem:convexcone} since we consider methods where $c_{ij}\geq 0$.
\end{proof}

We now establish the conditions for which $\bU_{\bK}^{(ij)}$ is realizable.
To this end, define
\begin{subequations}
    \label{eq:Gamma_3D}
    \begin{align}
        \Gamma^\varepsilon[\bU_h^{(j)};\hat{c}_{ij}]
        =&
        \frac{1}{|K_\varepsilon|}
        \int_{K_\varepsilon}\bU_h^{(j)}\,\varepsilon^2\,d\varepsilon \nonumber \\
        &-
        \frac{(d+1)\,\hat{c}_{ij}\,\Delta t}{|K_\varepsilon|}
        \big[\,
            \varepsilon^3\widehat{\bF}^\varepsilon(\bU_{h}^{(j)})\rvert_{\varepsilon_{\Hi}} \nonumber \\
            &\hspace{6em}-\varepsilon^3\widehat{\bF}^\varepsilon(\bU_{h}^{(j)})\rvert_{\varepsilon_{\Lo}}
        \,\big], \label{eq:Gamma_3D_energy} \\
        \Gamma^\ell[\bU_h^{(j)};\hat{c}_{ij}]
        =&
        \frac{1}{|K^\ell|}\int_{K^\ell}\bU_h^{(j)}\,dx^\ell \nonumber \\
        &-
        \frac{(d+1)\,\hat{c}_{ij}\,\Delta t}{|K^\ell|}
        \big[\,
            \widehat{\bF}^\ell(\bU_{h}^{(j)})\rvert_{x^\ell_{\Hi}} \nonumber \\
            &\hspace{6em}-\widehat{\bF}^\ell(\bU_{h}^{(j)})\rvert_{x^\ell_{\Lo}}
        \,\big], \label{eq:Gamma_3D_space}
    \end{align}
\end{subequations}
for $\ell=1,\ldots,d$.
Note that $\Gamma^{\varepsilon}$ is a function of $\vect{x}$ and $\Gamma^{\ell}$ is a function of $\tilde{\vect{z}}^{\ell}=\{\varepsilon,\tilde{\vect{x}}^{\ell}\}$, but to ease notation we do not include this explicit dependence below.  
It is straightforward to show that
\begin{align}
\label{eq:Uij_equals_sum_Gamma}
    \bU_{\bK}^{(ij)}
    &=
    \frac{1}{(d+1)|\bK_{\vect{x}}|}\int_{\bK_{\vect{x}}}\Gamma^\varepsilon[\bU_h^{(j)};\hat{c}_{ij}]\,d\vect{x} \nonumber \\
    &+
    \sum_{\ell=1}^{d}\frac{1}{(d+1)|\widetilde{\bK}^\ell|}\int_{\widetilde{\bK}^\ell}\Gamma^\ell[\bU_h^{(j)};\hat{c}_{ij}]\,\varepsilon^{2}\,d\Tilde{\vect{z}}^\ell.
\end{align}
Recall that we use $M_\varepsilon$-point and $M_x$-point LGL quadrature rules on the intervals $K_\varepsilon$ and $K^{\ell}$, given respectively by the points $\hat{S}^{M_\varepsilon}_{\varepsilon}(\bK)=\{\varepsilon_{\Lo}^{+}=\hat{\varepsilon}_1,\hat{\varepsilon}_{2},\ldots,\hat{\varepsilon}_{M_\varepsilon}=\varepsilon_{\Hi}^{-}\}$ and $\hat{S}^{M_x}_{\ell}(\bK)=\{x_{\Lo}^{\ell,+}=\hat{x}_{1}^{\ell},\hat{x}_{2}^{\ell},\ldots,\hat{x}_{M_x}^{\ell}=x_{\Hi}^{\ell,-}\}$ with respective weights $\{\hat{w}^\varepsilon_q\}_{q=1}^{M_\varepsilon}$ and $\{\hat{w}^\ell_q\}_{q=1}^{M_x}$.
\begin{lemma}
    \label{lem:Gamma:realizable_space}
    Assume $\bU_h^{(j)}(\hat{x}^\ell_q)\in\cR$ for all $\hat{x}^\ell_q\in \hat{S}^{M_x}_\ell(\bK)$, and $M_{x}\ge (k+3)/2$, where $k$ is the polynomial degree of $\bU_{h}^{(j)}$.  
    Let the spatial numerical flux be given as in Eq.~\eqref{eq:spatial_flux}, and $\Delta t$ be chosen such that
    \begin{equation}
      \gamma^\ell_{ij}
      \vcentcolon=\frac{(d+1)\,\hat{c}_{ij}\,\Delta t\,a^\ell}{\hat{w}^\ell_{M_x}\,\Delta x^\ell}\leq 1.
    \end{equation}
    Then $\Gamma^\ell[\bU_h^{(j)};\hat{c}_{ij}]\in\cR$ for $\ell=1,\ldots,d$.
\end{lemma}
\begin{proof}
    Using the spatial flux in Eq.~\eqref{eq:spatial_flux} and the $M_x$-point LGL quadrature rule defined by $\hat{S}^{M_x}_\ell(\bK)$ to evaluate the integral in Eq.~\eqref{eq:Gamma_3D_space} exactly, it is straightforward to show that
    \begin{align*}
        &\Gamma^\ell[\bU_h^{(j)};\hat{c}_{ij}]
        =\sum_{q=2}^{M_x-1} \hat{w}_q^\ell\bU_h^{(j)}(\hat{x}^\ell_q) \\
        &\qquad 
        +(1-\gamma_{ij}^{\ell})\,\hat{w}_{1}^{\ell}\,\bU_{h}^{(j)}(\hat{x}_{1}^{\ell})
        +(1-\gamma_{ij}^{\ell})\,\hat{w}_{M_x}^{\ell}\,\bU_{h}^{(j)}(\hat{x}_{M}^{\ell}) \\
        &\qquad 
        +\gamma_{ij}^{\ell}\,\hat{w}_{M_x}^{\ell}
        \big[\,
            \Phi^{\ell,-}(\bU_{h}^{(j)})\rvert_{x_{\Hi}^{\ell,+}}
            +\Phi^{\ell,-}(\bU_{h}^{(j)})\rvert_{x_{\Hi}^{\ell,-}}\\
            &\hspace{24pt}
            +\Phi^{\ell,+}(\bU_{h}^{(j)})\rvert_{x_{\Lo}^{\ell,-}}
            +\Phi^{\ell,+}(\bU_{h}^{(j)})\rvert_{x_{\Lo}^{\ell,+}}
        \,\big],
    \end{align*}
    where $\Phi^{\ell,\pm}$ is defined as in Lemma \ref{lem:realizableflux}.  
    Since $\bU_h^{(j)}(\hat{x}^\ell_q)\in\cR$ for all $\hat{x}^\ell_q\in \hat{S}^{M_x}_\ell(\bK)$, application of Lemma~\ref{lem:realizableflux} gives $\Phi^{\ell,\pm}\in\cR$.  
    Since $\Delta t$ is chosen such that $\gamma^{\ell}_{ij}\leq 1$, Lemma~\ref{lem:convexcone} implies $\Gamma^\ell[\bU_h^{(j)};c_{ij}]\in\cR$, since it is expressed as a conical combination of realizable states.
\end{proof}

\begin{lemma}
    \label{lem:Gamma:realizable_energy}
    Assume $\bU_h^{(j)}(\hat{\varepsilon}_q)\in\cR$ for all $\hat{\varepsilon}_q\in \hat{S}^{M_\varepsilon}_\varepsilon(\bK)$. 
    Let the numerical flux in the energy dimension be given as in Eq.~\eqref{eq:energy_flux}, with $a^\varepsilon=a_{\max}^{\varepsilon}\ge|q|$, as defined in Lemma~\ref{lem:energy_lemma_2}.  
    Let $\Delta t$ be chosen such that
    \begin{equation}
        \gamma^\varepsilon_{ij}
        \vcentcolon=\frac{(d+1)\,\hat{c}_{ij}\,\Delta t\,\varepsilon_{\Hi}\,a_{\max}^{\varepsilon}}{\hat{w}^\varepsilon_{M_\varepsilon}\,\Delta\varepsilon}\leq \frac{E}{\varepsilon}.
    \end{equation}
    Then $\Gamma^\varepsilon[\bU_h^{(j)};\hat{c}_{ij}]\in\cR$.
\end{lemma}
\begin{proof}
    Using the $M_\varepsilon$-point LGL quadrature rule defined by $\hat{S}^{M_\varepsilon}_\varepsilon(\bK)$ to evaluate the integral in Eq.~\eqref{eq:Gamma_3D_energy} exactly, we can express $\Gamma^\varepsilon[\bU_h^{(j)};\hat{c}_{ij}]$ as
    \begin{align*}
        \Gamma^\varepsilon[\bU_h^{(j)};\hat{c}_{ij}]
        =&
        \frac{\Delta\varepsilon}{|K_{\varepsilon}|}
        \Big\{
        \sum_{q=2}^{M_\varepsilon-1}
        \hat{w}^{\varepsilon}_{q}\,
        \hat{\varepsilon}^{2}_{q}\,
        \bU^{(j)}_{q} \\
        &+\hat{w}^{\varepsilon}_{1}\,
        \hat{\varepsilon}^{2}_{1}\,
        \big(\,
            \bU^{(j)}_{h}\rvert_{\varepsilon_{\Lo}^{+}}
            +
            \lambda^{\Lo}_{ij}
            \,
            \widehat{\bF}^{\varepsilon}(\bU_{h}^{(j)})\rvert_{\varepsilon_{\Lo}}
        \,\big) \nonumber \\
        &
        +\hat{w}^{\varepsilon}_{M_\varepsilon}\,
        \hat{\varepsilon}^{2}_{M_\varepsilon}\,
        \big(\,
            \bU^{(j)}_{h}\rvert_{\varepsilon_{\Hi}^{-}}
            -
            \lambda^{\Hi}_{ij}
            \,
            \widehat{\bF}^\varepsilon(\bU_{h}^{(j)})\rvert_{\varepsilon_{\Hi}}
        \,\big)
        \Big\}.
    \end{align*}
    Here, $\bU_q^{(j)}=\bU_h^{(j)}\rvert_{\hat{\varepsilon}_{q}}$, and $\lambda^{\Lo,\Hi}_{ij} = \frac{(d+1)\,\hat{c}_{ij}\,\Delta t\,\varepsilon_{\Lo,\Hi}}{\hat{w}^\varepsilon_M\,\Delta\varepsilon}$.
    Note that $\hat{\varepsilon}_{1} = \varepsilon_{\Lo}$ and $\hat{\varepsilon}_{M_\varepsilon} = \varepsilon_{\Hi}$.
    Since $\hat{w}^\varepsilon_q\,\hat{\varepsilon}_{q}^{2}\ge0$ for all $q$, the above expression is a conical combination.  
    To prove realizability it remains to show that
    \[
        \bU^{(j)}_{h}\rvert_{\varepsilon_{\Lo}^{+}}
        +\lambda_{ij}^{\Lo}\,\widehat{\bF}^\varepsilon(\bU_{h}^{(j)})
        \rvert_{\varepsilon_{\Lo}}\in\cR
    \]
    and
    \[
        \bU^{(j)}_{h}\rvert_{\varepsilon_{\Hi}^{-}}
        -\lambda_{ij}^{\Hi}\,\widehat{\bF}^\varepsilon(\bU_{h}^{(j)})
        \rvert_{\varepsilon_{\Hi}}\in\cR,
    \]  
    Note that $\lambda^{\Hi}_{ij}>\lambda^{\Lo}_{ij}$, since $\varepsilon_{\Hi}>\varepsilon_{\Lo}$,
    Expanding the flux term in the first expression using Eq.~\eqref{eq:energy_flux} we have
    \begin{align*}
        &\bU^{(j)}_{h}\rvert_{\varepsilon_{\Lo}^{+}} +\lambda_{ij}^{\Lo}\,\widehat{\bF}^\varepsilon(\bU_{h}^{(j)})|_{\varepsilon_{\Lo}} \\
        &\hspace{12pt}=
        \big(\,
            \bU^{(j)}_{h}\rvert_{\varepsilon_{\Lo}^{+}}
            -\lambda_{ij}^{\Lo}\, a_{\max}^{\varepsilon}\,\widetilde{\bU}\rvert_{\varepsilon_{\Lo}^{+}}
        \,\big)\\
        &\hspace{12pt}
        +\frac{\lambda_{ij}^{\Lo}\,a_{\max}^{\varepsilon}}{2}
        \big(\,
            \widetilde{\bU}\rvert_{\varepsilon_{\Lo}^{-}}
            +\frac{1}{a_{\max}^{\varepsilon}}\,\bF^\varepsilon\rvert_{\varepsilon_{\Lo}^{-}}
        \,\big)\\
        &\hspace{12pt}
        +\frac{\lambda_{ij}^{\Lo}\,a_{\max}^{\varepsilon}}{2}
        \big(\,
            \widetilde{\bU}\rvert_{\varepsilon_{\Lo}^{+}}
            +\frac{1}{a_{\max}^{\varepsilon}}\,\bF^\varepsilon\rvert_{\varepsilon_{\Lo}^{+}}
        \,\big).
    \end{align*}
    Applying Lemmas~\ref{lem:energy_lemma_1} and \ref{lem:energy_lemma_2}, we have that $\bU^{(j)}_{h}\rvert_{\varepsilon_{\Lo}^{+}} +\lambda_{ij}^{\Lo}\,\widehat{\bF}^\varepsilon(\bU_{h}^{(j)})|_{\varepsilon_{\Lo}}\in\cR$.
    Similarly,
    \begin{align*}
        &\bU^{(j)}_{h}\rvert_{\varepsilon_{\Hi}^{-}}
        -\lambda_{ij}^{\Hi}\,\widehat{\bF}^\varepsilon(\bU_{h}^{(j)})
        \rvert_{\varepsilon_{\Hi}} \\
        &\hspace{12pt}=
        \big(\,
            \bU^{(j)}_{h}\rvert_{\varepsilon_{\Hi}^{-}}
            -\lambda_{ij}^{\Hi}\,a_{\max}^\varepsilon\,\widetilde{\bU}\rvert_{\varepsilon_{\Hi}^{-}}
        \,\big)\\
        &\hspace{12pt}
        +\frac{\lambda_{ij}^{\Hi}\,a_{\max}^{\varepsilon}}{2}\,
        \big(\,
            \widetilde{\bU}\rvert_{\varepsilon_{\Hi}^{-}}
            -\frac{1}{a_{\max}^\varepsilon}\,\bF^\varepsilon\rvert_{\varepsilon_{\Hi}^{-}}
        \,\big)\\
        &\hspace{12pt}
        +\frac{\lambda_{ij}^{\Hi}\,a_{\max}^{\varepsilon}}{2}\,
        \big(\,
            \widetilde{\bU}\rvert_{\varepsilon_{\Hi}^{+}}
            -\frac{1}{a_{\max}^{\varepsilon}}\bF^\varepsilon\rvert_{\varepsilon_{\Hi}^{+}}).
    \end{align*}
    Again, applying Lemmas~\ref{lem:energy_lemma_1} and \ref{lem:energy_lemma_2} we have $\bU^{(j)}_{h}\rvert_{\varepsilon_{\Hi}^{-}}
    -\lambda_{ij}^{\Hi}\,\widehat{\bF}^\varepsilon(\bU_{h}^{(j)})
    \rvert_{\varepsilon_{\Hi}}\in\cR$.
    Realizability of $\Gamma^\varepsilon[\bU_h^{(j)};\hat{c}_{ij}]$ then follows from Lemma~\ref{lem:convexcone}.
\end{proof}

The previous two lemmas require assumptions on the time step $\Delta t$.  Satisfying those assumptions gives us our realizability-preserving time-step restriction in Eq.~\eqref{eq:realizable_timestep}.  
Now, to establish realizability of $\bU_{\bK}^{(ij)}$, it is sufficient for Eq.~\eqref{eq:Gamma_3D} to hold for the auxiliary quadrature sets defined in Eq.~\eqref{eq:auxiliary_quadrature}, $\hat{S}_{\varepsilon,\otimes}(\bK)$ and $\hat{S}_{\ell,\otimes}(\bK)$, provided $\Gamma^\varepsilon[\bU_h^{(j)};c_{ij}]$ and $\Gamma^\ell[\bU_h^{(j)};c_{ij}]$ are realizable.
\begin{lemma}
\label{lem:Uij_realizable}
    Assume that the conditions for Lemmas~\ref{lem:Gamma:realizable_space} and \ref{lem:Gamma:realizable_energy} hold for all quadrature points in $\hat{S}_{\varepsilon,\otimes}(\bK)$ and $\hat{S}_{\ell,\otimes}(\bK)$.
    Then $\bU_{\bK}^{(ij)}\in\cR$.
\end{lemma}
\begin{proof}
    The conditions of Lemma~\ref{lem:Gamma:realizable_space} and Lemma~\ref{lem:Gamma:realizable_energy} hold, which allow us to conclude $\Gamma^\varepsilon[\bU_h^{(j)};c_{ij}]\in\cR$ and $\Gamma^{\ell}[\bU_h^{(j)};\hat{c}_{ij}]\in\cR$ for $\ell=1,\ldots,d$.
    Then using the quadrature sets $\hat{S}_{\varepsilon,\otimes}(\bK)$ and $\hat{S}_{\ell,\otimes}(\bK)$ to evaluate the integrals over $\bK_{\vect{x}}$ and $\widetilde{\bK}^\ell$ in Eq.~\eqref{eq:Uij_equals_sum_Gamma}, respectively, yields
    \begin{align*}
        \bU^{(ij)}_{\bK}
        =&
        \frac{1}{(d+1)|\bK_{\vect{x}}|}\sum_{\vect{x}_q\in\hat{S}_{\varepsilon,\otimes}(\bK)}\vect{w}_q\Gamma^\varepsilon[\bU_h^{(j)};\hat{c}_{ij}](\vect{x}_q) \\
        &+
        \sum_{\ell=1}^d \frac{1}{(d+1)|\widetilde{\bK}^\ell|}
        \sum_{\widetilde{\vect{z}}^\ell_q\in\hat{S}_{\ell,\otimes}(\bK)}\varepsilon^2\widetilde{\vect{w}}^\ell_q\Gamma^\ell[\bU_h^{(j)}\hat{c}_{ij}](\widetilde{\vect{z}}^\ell_q)
    \end{align*}
    Here $\vect{x}_q=(x^1_{q_1},x^2_{q_2},\cdots,x^d_{q_d})$ and $\vect{w}_q=w^1_{q_1} w^2_{q_2}\hdots w^d_{q_2}$, where $1\leq q_1,q_2,\cdots,q_d \leq N$.  The quadrature points $\widetilde{\vect{z}}^\ell_q$ and their associated weights, $\widetilde{\vect{w}}^\ell_q$ are similarly defined, and recall by the tilde we mean to exclude the $\ell^{\mathrm{th}}$ dimension.
    Because the weights $\vect{w}_q$ and $\widetilde{\vect{w}}^\ell_q$ are from LG quadratures,
    we have expressed $\bU^{(ij)}_{\bK}$ as convex combinations of $\Gamma^\varepsilon[\bU_h^{(j)};c_{ij}]$ and $\Gamma^\ell[\bU_h^{(j)};c_{ij}]$, respectively.
    Applying Lemma \ref{lem:convexcone} then gives the desired result.
\end{proof}

We can now prove Proposition~\ref{prop:explicit_step_realizable}.
\begin{proof}[Proof of Proposition \ref{prop:explicit_step_realizable}]
    We have that
    \[
    \bU^{(i)}_{\mathrm{Ex}}=\sum_{j=0}^{i-1}c_{ij}\bU^{(ij)}_{\bK}.
    \]
    The four assumptions of the proposition statement satisfy the assumptions needed for Lemmas \ref{lem:Gamma:realizable_space} and \ref{lem:Gamma:realizable_energy}. Since these two lemmas hold, we can apply Lemma \ref{lem:Uij_realizable} to get $\bU^{(ij)}_{\bK}\in\cR$.
    Finally we can apply Lemma \ref{lem:Ui_realizable} to get the desired result.
\end{proof}

\begin{rem}
    Lemma \ref{lem:Gamma:realizable_space} and \ref{lem:Gamma:realizable_energy} require inequalities on $\gamma_{ij}^\ell$ and $\gamma_{ij}^\varepsilon$, respectively, to hold.
    Solving for $\Delta t$ in both inequalities gives the quantities we minimize over to determine the realizable time step in Eq.~\eqref{eq:realizable_timestep}.
    The inequality in Lemma \ref{lem:Gamma:realizable_energy} requires that $\gamma^{\varepsilon}_{ij}\leq E/\varepsilon$.
    Then for realizability we require $\Delta t \leq \frac{E}{\varepsilon}\frac{\hat{w}_{M_\varepsilon}^\varepsilon\,\Delta\varepsilon}{(d+1)\,\varepsilon_{\Hi}\,a^\varepsilon}$.
    In practice, $E/\varepsilon$ is not known precisely.
    Therefore, to arrive at Eq.~\eqref{eq:realizable_timestep}, we lower bound $E/\varepsilon$ by $W(1-v)$.
\end{rem}

\subsection{Realizability-Enforcing Limiter}\label{sec:limiter}

Proposition~\ref{prop:explicit_step_realizable} guarantees the explicit update of the $i$th stage cell-average is realizable.
It does not guarantee that $\bU_h(\vect{x})^{(j)}\in\cR$ for all $\vect{x}\in\widetilde{S}_{\otimes}(\bK)$, as required by Condition~3 of Proposition~\ref{prop:explicit_step_realizable}.
To ensure this requirement, we employ the following element-local realizability-enforcing limiter from \cite{chu_etal_2019} after each explicit update.

To enforce positivity of the zeroth moment, $\cE$, we replace the polynomial $\cE_h(\vect{x})$ with the limited polynomial
\begin{equation}
    \overline{\cE}_h(\vect{x}) 
    = \theta_1\,\cE_h(\vect{x}) + (1-\theta_1)\,\cE_{\bK},
\end{equation}
where the limiter parameter $\theta_1$ is given by
\begin{equation}
    \theta_1 = \min\left\{\left| \frac{m-\cE_{\bK}}{m_{\widetilde{S}_{\otimes}}-\cE_{\bK}} \right|,1\right\},
\end{equation}
with
\[
    m=0
    \quad
    \text{and}
    \quad
    m_{\widetilde{S}_{\otimes}}
    =\min_{\vect{x}\in\widetilde{S}_{\otimes}(\bK)}\cE_{h}(\vect{x}).
\]

The next step enforces $\cE-\cF \geq 0$.
We define $\overline{\bU}_h = (\overline{\cE}_{h},\cF_h)^{\intercal}$.
If $\overline{\bU}_{q}=\overline{\bU}_h(\vect{x}_{q})$ lies outside of $\cR$ for any quadrature point $\vect{x}_q$, then $\vect{x}_q$ is a troubled quadrature point. 
For troubled quadrature points, there exists a line connecting the realizable cell average, $\bU_{\bK}$, and $\overline{\bU}_q$ that intersects the boundary of $\cR$.  
This line is parametrized by
\begin{equation}
    \bs_q(\psi) = \psi\overline{\bU}_q + (1-\psi) \bU_{\bK}, \quad \psi\in[0,1].
\end{equation}
The unique point where $\bs_q(\psi)$ intersects the boundary of $\cR$ is obtained by solving $\gamma(\bs_q(\psi))=0$ for $\psi$ using the bisection algorithm.  
(Recall that for $\bU = (\cE,\cF_j)^{\intercal}$, $\gamma(\bU) = \cE - \cF$.)  
We then replace $\overline{\bU}_h$ by
\begin{equation}
    \overline{\overline{\bU}}_h
    =
    \theta_2\,\overline{\bU}_h + (1-\theta_2)\,\bU_{\bK},
\end{equation}
where the parameter $\theta_2$ is the smallest $\psi$ obtained in element $\bK$ by considering all troubled quadrature points.  
The limiter is conservative in the sense that it preserves the cell average; i.e., $\bU_{\bK} = \overline{\bU}_{\bK} = \overline{\overline{\bU}}_{\bK}$.

\subsection{Collision Update}
\label{sec:CollisionAnalysis}

In this section, we prove the realizability-preserving property of the implicit collision solver.
We first note that our analysis of the explicit update shows that $\bU^{(*)}$ in Eq.~\eqref{eq:backwardEulerConserved} is realizable, when $\bU^{(*)}=\bU^{(i)}_{\mathrm{Ex}}$.
Furthermore, by Lemma \ref{lem:realizable_hat_moments}, $\widehat{\bU}^{(*)}$, as defined in Eq.~\eqref{eq:backwardEulerHat}, is realizable if $\bU^{(*)}$ is realizable.  
\begin{lemma}
\label{lem:realizable_collision_solver}
    Let $\bU^{(*)}=(\cE^{(*)},\cF_\mu^{(*)})^\intercal$, defined as in Eq.~\eqref{eq:backwardEulerConserved}, be realizable, and assume that $\bM^{[k]} = (\cJ^{[k]},\cH_\mu^{[k]})^{\intercal}$ is realizable.
    Let $\bM^{[k+1]} = (\cJ^{[k+1]},\cH_\mu^{[k+1]})^{\intercal}$ be given by the Picard iteration in Eq.~\eqref{eq:collision_richardson} with $\lambda \leq (1+v)^{-1}$.
    Then $\bM^{[k+1]}$ is realizable.
\end{lemma}
\begin{proof}
    The first component of $\bM^{[k+1]}$ is
\begin{equation*}
    \cJ^{[k+1]}
    =
    \mathscr{J}^{[k]} + \lambda\mu_{\chi}(\widehat{\cE}^{(*)}+\Delta\tau \chi \cJ_{\Equilibrium}),
\end{equation*}
where we have defined
\begin{equation*}
    \mathscr{J}^{[k]}
    =
    \mu_\chi\,
    \big[\,
        (1-\lambda)\,W\cJ^{[k]}-\lambda\,v^i\cH^{[k]}_i
    \,\big]
    =
    \frac{1}{4\pi}\int_{\bbS^2}g^{[k]}\,d\omega
\end{equation*}
and $g^{[k]}=\mu_{\chi}\big[\,(1-\lambda)W-\lambda v^i\ell_i\,\big]\,\varepsilon f^{[k]}$.
Similarly, the second component of $\bM^{[k+1]}$ is
\begin{equation*}
    \cH_{\mu}^{[k+1]}
    =
    \mathscr{H}^{[k]}_\mu + \lambda\mu_{\kappa}\widehat{\cF}_\mu^{(*)},
\end{equation*}
where we have defined
\begin{equation*}
    \mathscr{H}^{[k]}_\mu
    =
    \mu_{\kappa}\,
    \big[\,(1-\lambda)\,W\cH^{[k]}_\mu-\lambda\, v^j\cK^{[k]}_{\mu j}\,\big]
    =\frac{\mu_\kappa}{\mu_\chi}\,\widetilde{\mathscr{H}}^{[k]}_\mu,
\end{equation*}
with
\begin{equation*}
    \widetilde{\mathscr{H}}^{[k]}_\mu=\frac{1}{4\pi}\int_{\bbS^2}g^{[k]}\ell_\mu\,d\omega.
\end{equation*}
Note that there is no independent update for $\cH^{[k+1]}_0$ in Eq.~\eqref{eq:collision_richardson}.
Instead, it is completely determined by $\cH^{[k+1]}_i$, since $\cH^{[k+1]}_0=-v^i\cH^{[k+1]}_i=-v^i\mathscr{H}^{[k]}_i-\lambda\mu_{\kappa}v^i\widehat{\cF}^{(*)}_\mu=\mathscr{H}^{[k]}_0+\lambda\mu_{\kappa}\widehat{\cF}^{(*)}_0$.

We desire $g^{[k]}\in\mathfrak{R}$. 
By assumption, $\varepsilon f^{[k]}\in\mathfrak{R}$, therefore we need to show $W(1-\lambda)-\lambda v^j\ell_j\geq 0$.  
Similar to the moment conversion problem considered in Lemma~\ref{lem:momentconversion_realizable}, this condition is satisfied when $\lambda\le(1+v)^{-1}$.  
Then, since the moment pair $\widetilde{\mathscr{M}}^{[k]}=(\mathscr{J}^{[k]},\widetilde{\mathscr{H}}^{[k]}_\mu)^\intercal$ arise from $g^{[k]}\in\mathfrak{R}$, Proposition~\ref{prop:momentBoundsLagrangian} implies that $\widetilde{\mathscr{M}}^{[k]}\in\cR$.
Furthermore, $\mathscr{M}^{[k]}=(\mathscr{J}^{[k]}, \mathscr{H}^{[k]}_\mu)^\intercal\in\cR$, since
\begin{equation*}
    \mathscr{H}^{[k]}=\frac{\mu_\kappa}{\mu_\chi}\widetilde{\mathscr{H}}^{[k]}\leq\widetilde{\mathscr{H}}^{[k]}\leq \mathscr{J}^{[k]}.
\end{equation*}
Here, $\mathscr{H}^{[k]}=\sqrt{\eta^{\mu\nu}\mathscr{H}^{[k]}_\mu\mathscr{H}^{[k]}_\nu}$, $\widetilde{\mathscr{H}}^{[k]}=\sqrt{\eta^{\mu\nu}\widetilde{\mathscr{H}}^{[k]}_\mu\widetilde{\mathscr{H}}^{[k]}_\nu}$, and $0\leq\mu_\kappa\leq{\mu_\chi}$ because $0\leq\chi\leq\kappa$.
Note that since $\bU^{(*)}\in\cR$, we have $\widehat{\bU}^{(*)}=(\widehat{\cE}^{(*)},\widehat{\cF}_\mu^{(*)})^\intercal\in\cR$.
Then, by a similar argument, $\Lambda\widehat{\bU}^{(*)}\in\cR$.
It follows from Lemma~\ref{lem:convexcone} that
\[
    \bM^{[k+1]}
    =\mathscr{M}^{[k]}
    +\lambda\Lambda\widehat{\bU}^{(*)}
    +\lambda\mu_\chi\Delta\tau\chi(\cJ_{\Equilibrium},0)^\intercal\in\cR.
\]
\end{proof}

Similar to the analysis of the moment conversion solver in Section~\ref{sec:MomentConversionAnalysis}, the result of Lemma~\ref{lem:realizable_collision_solver} is needed for the convergence analysis of the implicit collision solver in Eq.~\eqref{eq:collision_richardson}. In Eq.~\eqref{eq:collision_richardson}, the collision term introduces damping factors $\mu_\chi$ and $\mu_\kappa$ that are stronger than the $W^{-1}$ factor in the moment conversion solver given in Eq.~\eqref{eq:Richardson_update}. Therefore, with Lemma~\ref{lem:realizable_collision_solver}, the implicit solver is expected to converge given the convergence of the moment conversion solver analyzed in Appendix~\ref{appendix:convergence}.

To prove Proposition \ref{prop:realizable_stage_values}, we also need to show that realizable primitive moments lead to realizable conserved moments, which is given in Proposition~\ref{prop:momentBoundsEulerian} under the exact closure assumption (see, e.g., \cite[Assumption~1]{laiu_etal_2025}).
In the following lemma, we prove the desired result when the algebraic Eddington factor $\mathsf{k}$ (see Eq.~\eqref{eq:eddingtonFactorAlgebraic}) is used, i.e., when the closure is not exact.

\begin{lemma}\label{lemma12}
    Assume that $v<1$.  
    Let $\bU=(\cE,\cF_{\mu})$, with $\cE$ and $\cF_{\mu}$ given as in Eqs.~\eqref{eq:eulerianEnergy}-\eqref{eq:eulerianMomentum} with $\cK^{\mu\nu}(\bM)$ and $\mathsf{k}$ given by Eqs.~\eqref{eq:pressureTensor} and \eqref{eq:eddingtonFactorAlgebraic}, respectively, and $\bM=(\cJ,\cH_{\mu})^{\intercal}\in\cR$.  
    Then, $\bU\in\cR$.
\end{lemma}
\begin{proof}
    We break the proof into two steps.  
    In the first step, where we use some key results from \cite[Lemma~11]{laiu_etal_2025}, we show that $\widehat{\bU}=(\widehat{\cE},\widehat{F}_{\mu})^{\intercal}\in\cR$.  
    Using the bound
    \begin{equation}
        -vW\cH \le v^{\mu}\cH_{\mu} \le vW\cH,
        \label{eq:vDotHBound}
    \end{equation}
    we have
    \begin{equation}
        \widehat{\cE}
        =W\cJ+v^{\mu}\cH_{\mu}
        \ge W\cJ-vW\cH
        > W(\cJ-\cH)
        \ge0.  
    \end{equation}
    Next, we prove ${\widehat{\cE}}^{2}-\widehat{\cF}^{2}\ge0$, which implies $\widehat{\cE}-\widehat{\cF}\ge0$ and thus $\widehat{\bU}\in\cR$.  
    Direct calculations give
    \begin{align}
        \widehat{\cE}^{2}
        &=(W\cJ)^{2}\,\big(\,1+2s\mathsf{h}+s^{2}\mathsf{h}^{2}\,\big), \\
        \widehat{\cF}^{2}
        &=\cJ^{2}\,
        \big(\,
            W^{2}\mathsf{h}^{2}
            +2Wv^{\nu}\mathsf{k}_{\nu\mu}\widehat{\mathsf{h}}^{\mu}\mathsf{h}
            +v^{\nu}\mathsf{k}_{\nu}^{\hspace{4pt}\mu}\mathsf{k}_{\mu\rho}v^{\rho}
        \,\big), \label{eq:fHatSquared}
    \end{align}
    where $\mathsf{k}_{\mu\nu}=\cK_{\mu\nu}/\cJ$, $\mathsf{h}\in[0,1]$ is the flux factor, and we have defined $s=v^{\mu}\widehat{\mathsf{h}}_{\mu}/W\in[-1,1]$.  
    Using Eq.~\eqref{eq:pressureTensor}, we obtain
    \begin{align}
        v^{\nu}\mathsf{k}_{\nu\mu}\widehat{\mathsf{h}}^{\mu}
        &=W\,s\,\mathsf{k}, \\
        v^{\nu}\mathsf{k}_{\nu}^{\hspace{4pt}\mu}\mathsf{k}_{\mu\rho}v^{\rho}
        &=W^{2}\,\frac{1}{4}\,
        \big(\,(1-\mathsf{k})^{2}\,v^{2}+(1+\mathsf{k})\,(3\mathsf{k}-1)\,s^{2}\,\big).
    \end{align}
    Inserting these into Eq.~\eqref{eq:fHatSquared} gives the following sufficient condition for $\widehat{\cE}^{2}-\widehat{\cF}^{2}\ge0$: $\forall s\in[-1,1]$ and $\forall \mathsf{h}\in[0,1]$, 
    \begin{equation}
        \big(\,1-\mathsf{h}^{2}-\frac{1}{4}(1-\mathsf{k})^{2}\,\big)
        +2(1-\mathsf{k})s\mathsf{h}
        +\big(\,\mathsf{h}^{2}-\frac{1}{4}(1+\mathsf{k})(3\mathsf{k}-1)\,\big)\,s^{2}
        \ge0,
        \label{eq:sufficientCondition}
    \end{equation}
    where we have used the assumption $v<1$.  
    Eq.~\eqref{eq:sufficientCondition} is identical to Eq.~(122) of Lemma~11 in \cite{laiu_etal_2025} and holds for the algebraic Eddington factor in Eq.~\eqref{eq:eddingtonFactorAlgebraic}.  
    Thus, relying on results in \cite{laiu_etal_2025}, we conclude that $\widehat{\bU}\in\cR$.  

    In the second step, we use the result of the first step to complete the proof.  
    From the first expression of Eq.~\eqref{eq:Hat_to_Conserved}, using a bound similar to Eq.~\eqref{eq:vDotHBound} for $v^{\mu}\widehat{\cF}_{\mu}$, we have
    \begin{equation}
        \cE = W\widehat{\cE}+v^{\mu}\widehat{\cF}_{\mu}
        \ge W\widehat{\cE}-vW\widehat{\cF}
        >W(\widehat{\cE}-\widehat{\cF})\ge0.
    \end{equation}
    Direct calculations with the expressions in Eq.~\eqref{eq:Hat_to_Conserved} gives
    \begin{equation}
        \cE^{2}-\cF^{2}
        =\widehat{\cE}^{2}-\widehat{\cF}^{2}+(v^{\mu}\widehat{\cF}_{\mu})^{2}
        \ge\widehat{\cE}^{2}-\widehat{\cF}^{2}\ge0.
    \end{equation}
    This completes the proof.  
\end{proof}

We can now prove Proposition~\ref{prop:realizable_stage_values}.
\begin{proof}[Proof of Proposition \ref{prop:realizable_stage_values}]
    Since the conditions of Proposition \ref{prop:explicit_step_realizable} hold, $\bU^{(i)}_{\mathrm{Ex}}\in\cR$.
    The realizability-enforcing limiter is applied to $\bU^{(i)}_{\mathrm{Ex}}$ to ensure pointwise realizability.
    Then, Lemma \ref{lem:realizable_collision_solver} can be repeatedly applied to Eq.~\eqref{eq:ImplicitUpdate}, starting with $\bM^{[0]}=\bU^{(i)}_{\mathrm{Ex}}$, until the iterative solver converges to the Lagrangian moments which define $\bU_{\bK}^{(i)}$.
    Thus it follows from Lemma~\ref{lemma12} that $\bU_{\bK}^{(i)}\in\cR$.
\end{proof}

\subsection{Realizability-Preserving DG-IMEX Scheme}

With the previously established lemmas and propositions in this section, we can finally prove the realizability-preserving property of our DG-IMEX scheme as stated in Theorem \ref{thm:realizable_scheme}.  
\begin{proof}[Proof of Theorem~\ref{thm:realizable_scheme}]
    Using Proposition~\ref{prop:explicit_step_realizable} with $i=1$ and Proposition~\ref{prop:realizable_stage_values} we get $\bU_{\bK}^{(1)}\in\cR$.  
    Applying the realizability-enforcing limiter gives $\bU_h^{(1)}\in\cR$ for all $\vect{x}\in\widetilde{S}_{\otimes}(\bK)$.  
    Repeated application of these steps for $i=2,\ldots,s$ gives $\bU_{\bK}^{(s)}\in\cR$.  
    Then since $\bU_{\bK}^{(s)} = \bU_{\bK}^{(n+1)}$, the proof is complete.
\end{proof}

\section{Numerical Results}
\label{sec:numericalResults}

In this section, we evaluate the performance of our proposed numerical scheme for the special relativistic two-moment model.
We use tests with and without collisions.
For tests with collisions, the polynomial degree of the DG method is quadratic ($k=2$) and the IMEX PD-ARS \cite{chu_etal_2019} method, with coefficients given by Eq.~\eqref{eq:IMEXPDARS_table}, is used.
For collisionless tests, the second order method uses linear polynomials ($k=1$) and SSPRK$2$ time stepping, with coefficients given by Eq.~\eqref{eq:IMEXPDARS_ShuOsher}, and the third order method uses quadratic polynomials ($k=2$) and SSPRK$3$ time stepping, with coefficients given by Eq.~\eqref{eq:SSPRK3_ShuOsher}.
Unless specified otherwise, spatial and energy elements are uniformly spaced, and collisionless results are presented using the third order method.
In each test, the time step is determined by the realizable time-step restriction as outlined in Eq.~\eqref{eq:realizable_timestep}.

\subsection{Moment Conversion Solver}
\label{sec:MomentConverstionTest}

The conserved to primitive moment conversion problem incurs the most cost in our realizability-preserving scheme.  
Because of this, it is desirable to have an iterative scheme that converges quickly to the desired accuracy, and where the evaluation of each iterate is inexpensive.  
Generally, Picard iteration is inexpensive per iterate but may converge slowly, while Newton iteration is more costly per iterate but with faster convergence.  
Here, we compare the proposed Picard iteration solver with Newton's method to solve Eq.~\eqref{eq:inverse_problem}.

To this end, similar to \cite{laiu_etal_2025}, we fix a pair $(v,\mathsf{h}=\cH/\cJ)$, with $v\in[0,1)$ and $\mathsf{h}\in[0,1]$, and randomly sample $100$ realizable values of $\bM$.  
To generate the samples for a fixed $(v,\mathsf{h})$, spherical polar coordinates are used to randomly sample directions for the three-velocity and comoving-frame flux.  
First, for randomly sampled $\vartheta\in[0,\pi]$ and $\varphi\in[0,2\pi]$, the fluid three-velocity is given by
\begin{subequations}
\begin{align}
    v^1 &= v\sin\vartheta\cos\varphi,\\
    v^2 &= v\sin\vartheta\sin\varphi,\\
    v^3 &= v\cos\vartheta.
\end{align}
\end{subequations}
To generate $\cJ\in[0,\infty)$ without setting a hard upper bound on $\cJ$, a parameter $\alpha$ is sampled uniformly in the interval $[-\pi/2,\pi/2].$
We then set $\cJ = -1/m_\alpha$, $m_\alpha=(1-\sin\alpha)/(0-\cos\alpha)$.
This value of $\cJ$ is determined by finding where the line connecting $(0,1)$ and $(\cos\alpha,\sin\alpha)$, given by $y=m_\alpha x+1$, intersects the $x$-axis. 
Next, using this value of $\cJ$, $\cH^{\hat{\imath}}\vcentcolon=\cJ q^{\hat{\imath}}$, where for randomly sampled $\vartheta'\in[0,\pi]$ and $\varphi'\in[0,2\pi]$
\begin{subequations}
\begin{align}
    q^{\hat{1}} &= \mathsf{h}\sin\vartheta'\cos\varphi',\\
    q^{\hat{2}} &= \mathsf{h}\sin\vartheta'\sin\varphi',\\
    q^{\hat{3}} &= \mathsf{h}\cos\vartheta'.
\end{align}
\end{subequations}
Finally, the Lorentz transformation, $\cL^\mu_{\hspace{4pt}\hat{\imath}}$, given by Eq.~\eqref{eq:lorentzTransformation}, is used to generate $\cH^\mu=\cL^\mu_{\hspace{4pt}\hat{\imath}}\cH^{\hat{\imath}}$.
Note that $\cH_\mu\cH^\mu = \cH_{\hat{\imath}}\cH^{\hat{\imath}} = \mathsf{h}^2\cJ^2$.
We then have $\cJ-\cH=(1-\mathsf{h})\,\cJ\ge0$.  
Thus, the randomly generated samples of $\bM$ are realizable.
Given realizable $\bM$, we compute $\widehat{\bU}=(\widehat{\cE},\widehat{\cF}_{i})^{\intercal}\in\cR$, using Eq.~\eqref{eq:inverse_problem}, and aim to recover $\bM$ using the two methods.  

\begin{figure*}
    \centering
    \begin{subfigure}[b]{0.45\textwidth}
        \centering
        \includegraphics[width = 8.6cm]{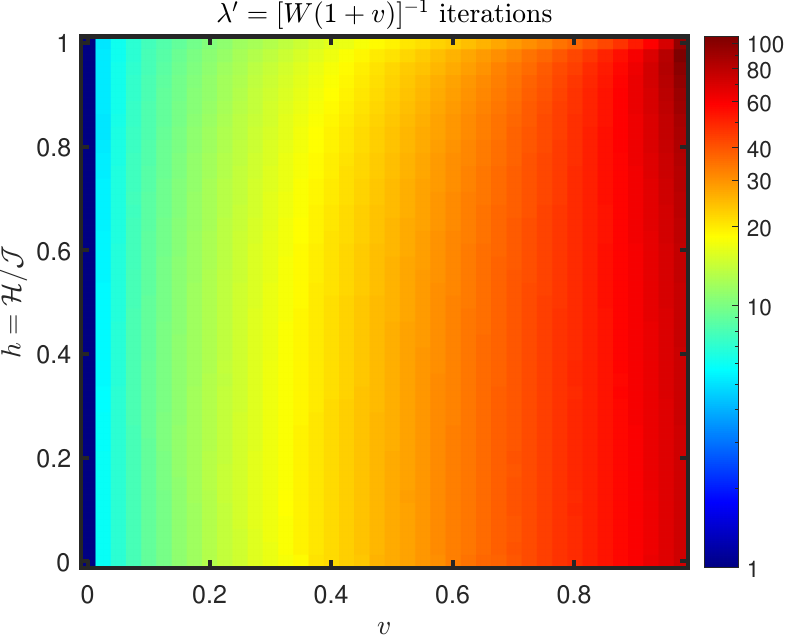}
    \end{subfigure}
    \hfill
    \begin{subfigure}[b]{0.45\textwidth}
        \centering
        \includegraphics[width = 8.6cm]{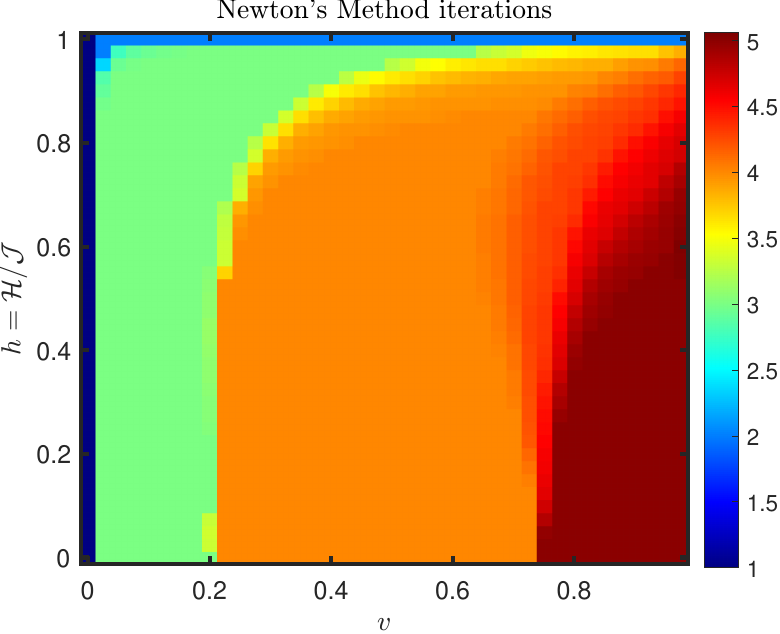}
    \end{subfigure}
    \begin{subfigure}[b]{0.45\textwidth}
        \centering
        \includegraphics[width = 8.6cm]{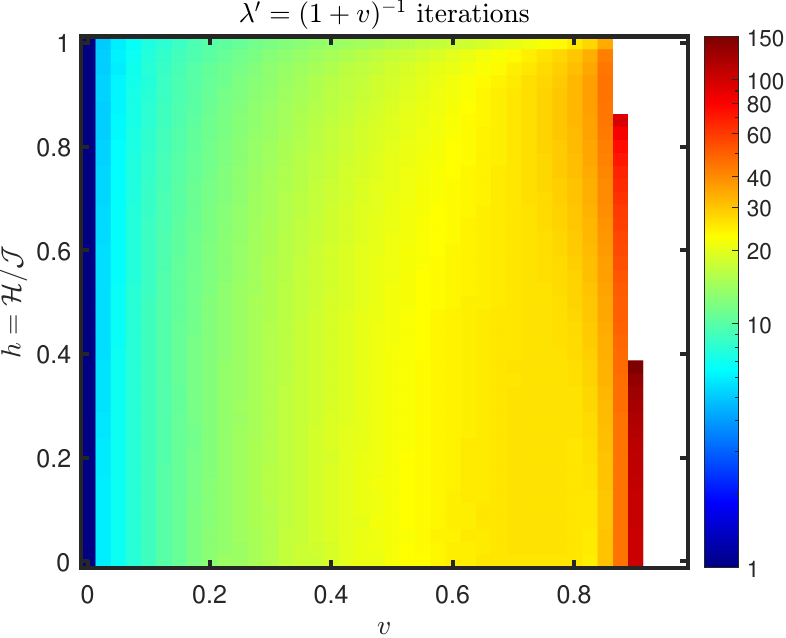}
    \end{subfigure}
    \hfill
    \begin{subfigure}[b]{0.45\textwidth}
        \centering
        \includegraphics[width = 8.6cm]{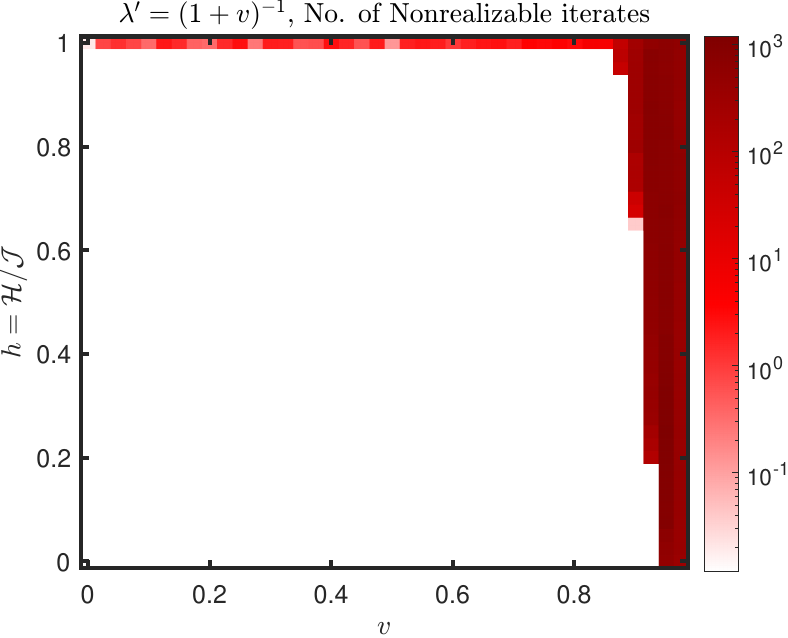}
    \end{subfigure}
    \caption{Top left: Iteration counts for the proposed Picard iteration method for the moment conversion problem using $\lambda' = \lambda/W = [W(1+v)]^{-1}$.
    Top right: Iteration counts Newton's method when solving the moment conversion problem.
    Bottom left: Picard iteration using $\lambda' = (1+v)^{-1}$.
    Each square represents the average iteration count of $100$ randomly generated samples of $\bM$ at each $(v,\mathsf{h})$.
    White squares indicate Picard iteration failed to converge within $10000$ iterations for at least one of the $100$ values of $\bM$ tested for that specific pairing of $(v,\mathsf{h})$.
    Bottom right: Number of nonrealizable iterates produced by Picard iteration when using a nonrealizable $\lambda' = (1+v)^{-1}$.
    Each square represents the average number of nonrealizable iterates before the convergence criteria are satisfied for the random samples of $\bM$.
    White squares have a value of zero, indicating realizability is maintained.}
    \label{fig:iteration_count_comparison}
\end{figure*}

In Figure \ref{fig:iteration_count_comparison}, for each $(v,\mathsf{h})$, we plot the average iteration count needed for convergence of the $100$ randomly generated samples of $\bM$ for both Picard iteration, using the ``step size'' $\lambda'=\lambda/W = [W(1+v)]^{-1}$, and Newton's method. 
From Eq.~\eqref{eq:convergence_criteria}, we set $\texttt{tol}_{\rm C2P}=10^{-8}$.
We can see that our proposed method based on Picard iteration takes significantly more iterations to reach convergence for the specified tolerance.  
At $(v,\mathsf{h}) = (0.975,1), (0.975,0), (0.3,1)$, the averaged Picard iteration count is $97, 74$, and $13$, respectively, while for Newton's method it is $2, 5$, and $2$, respectively.
The maximum observed averaged iteration count for Picard iteration was $106$, while for Newton's method it was $5$.
It should also be noted that the minimum error and maximum error (excluding the case where $v = 0$) for Picard iteration are, respectively, on the order of $10^{-11}$ and $10^{-8}$. While for Newton's method those values are $10^{-16}$ and $10^{-13}$. 
While performing our numerical experiments we have observed that, despite requiring more iterations, the Picard iteration scheme is somewhat faster than the Newton-based scheme when $v\lesssim 0.6$.  
We have not attempted to optimize the implementations to better assess which is a better choice in production simulations, but this could be considered in a future study since the conserved to primitive solver contributes significantly to the overall computational cost of our proposed method.  

In the bottom half of Figure \ref{fig:iteration_count_comparison} we break the realizability-preserving constraint on $\lambda'$ by setting it to the larger value of $\lambda' = (1+v)^{-1}$.
This causes the Picard iteration to fail to converge in $10000$ iterations for any $\mathsf{h}$ when $v>0.925$.
For $v = 0.9$ and $v = 0.875$, it fails to converge for $\mathsf{h}\geq 0.45$ and $\mathsf{h}\geq 0.925$, respectively.
Otherwise for any $\mathsf{h}$ and $v<0.875$, the Picard iteration converges.
Despite breaking the realizability preserving constraint, using the Picard iteration with $\lambda' = (1+v)^{-1}$ does not result in nonrealizable iterates for the majority of $(v,\mathsf{h})$ pairings.
Nonrealizable iterates are encountered for all $v$ whenever $\mathsf{h} = 1$, which should be expected since these $\bM$ are on the boundary of $\cR$.
Nonrealizable iterates for $\mathsf{h}<1$ occur in the region where the Picard iteration failed to converge.
For any $\mathsf{h}$ and $v>0.95$, nonrealizable iterates are encountered.
For $v = 0.875, 0.9, 0.925$, nonrealizable iterates are encountered for $\mathsf{h}>0.95$, $\mathsf{h}>0.65$, and $\mathsf{h}>0.2$, respectively.

In light of these observations, it would be desirable to prove that Newton's method $(i)$ preserves realizability of the iterates, $\bM^{[k]}$, and $(ii)$ converges for all $v\in[0,1)$.
The reason we use Picard iteration as opposed to Newton's method is because proving $(i)$ and $(ii)$ is easier since we work with the moments directly. 
Whereas in Newton's method, one would be working with derivatives of the moments.
In regards to $(i)$ and $(ii)$ we note:
\begin{enumerate}
    \item Though we have not proven $(ii)$ for Picard iteration, in all our tests using the realizability-preserving step size we have not encountered a $v$ for which the method fails to converge.
    \item We have not encountered an instance of a non-realizable iterate when using Newton's method. 
\end{enumerate}

\subsection{Streaming Sine Wave}

In this test, we model the propagation of free-streaming radiation through a background with a constant velocity field.
This test, originally from Section~8.2 of \cite{laiu_etal_2025}, is adapted here to the special relativistic case.
We set $\chi=\sigma=0$ and consider a fluid three-velocity $\vect{v}=[v,0,0]^{\intercal}$ with $v=0.1$.
The purpose of the test is to verify the scheme's order of accuracy.
The test is performed on a periodic one-dimensional unit domain, $\Omega_{x^1}=[0,1]$.
The Lagrangian energy is initialized to 
\[
    \cJ(x^1,0)=0.5+0.49\times\sin(2\pi x^1).
\]
We require the flux factor $\cH/\cJ=1$.
This implies the Lagrangian momentum for any $t$ is 
\[
\cH^1(x^1,t)=W\,\cJ(x^1,t).
\]
Under these conditions, $\cE(x^1,t)=\cF^1(x^1,t)=\cS^1_{\hspace{4pt}1}(x^1,t)=W^2(1+v)^2\cJ(x^1,t)$, so the analytical solution is given by $\cJ(x^1,t)=\cJ(x^1-t,0)$.  
We run the test until $t=1$, at which point the initial profile has returned to its initial position.

In Figure \ref{fig:SineWaveStreaming_OrderPlot} we plot the $L^2$ error of the numerical solution versus the number of spatial elements.
For the second order and third order scheme we observe the expected convergence rates to the exact solution.

\begin{figure}
    \centering
    \includegraphics[width=8.6cm]{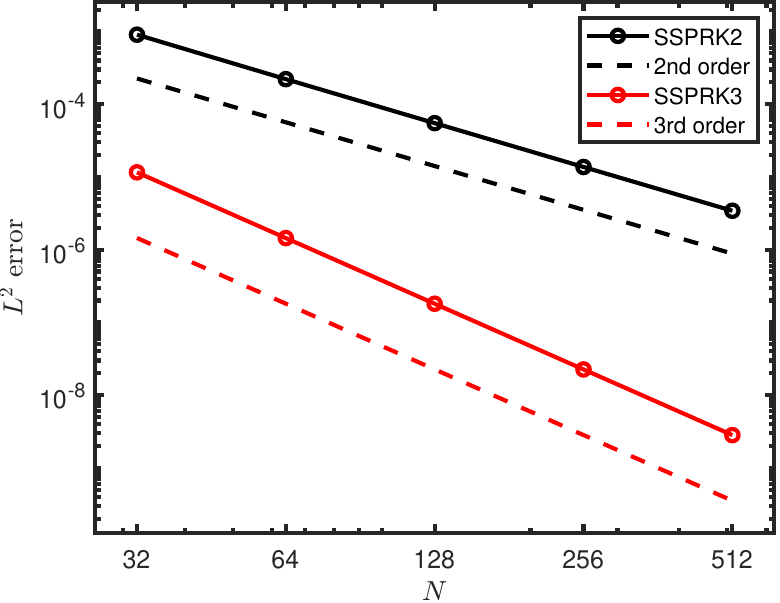}
    \caption{Convergence of the method on the streaming sine wave problem for linear (black) and quadratic (red) elements. The dashed lines are order reference lines.}
    \label{fig:SineWaveStreaming_OrderPlot}
\end{figure}

\subsection{Gaussian Diffusion 1D}

Here we present two 1D tests modeling the diffusion of radiation through a background moving with a constant speed.  
The first diffusion test follows \cite{laiu_etal_2025}, which was adopted from \cite{just_etal_2015}, while the second diffusion test follows \cite{radice_etal_2022}.  
In both tests we consider a purely scattering medium ($\chi=0$), and the two tests differ in the initial profile and velocity magnitude.  

\subsubsection{Test I}
In this test we set $\sigma=3.2\times 10^3$, and we let $\vect{v}=[v,0,0]^{\intercal}$ with $v=0.1$.  
The spatial domain is periodic, $\Omega_{x^1}=[0,3]$, and is discretized using $96$ elements.
The initial conditions are
\begin{align*}
    \cJ(x^1,0)
    &=\mathrm{exp}[-(x^1-x^1_0)^2/(4t_0\kappa_{\mathrm{D}})], \\
    \cH^1(x^1,0)
    &=-\kappa_{\mathrm{D}}\p_{x^1}\cJ(x^1,0),
\end{align*}
where $\kappa_{\mathrm{D}}=(3\sigma)^{-1}$, and we set $x^1_0=1$ and $t_0=5$.  
For the background velocity considered here, the evolution equation for the Lagrangian energy is approximately governed by the advection-diffusion equation
\begin{equation}
    \p_t\cJ+\p_{x^1}(\,\cJ v-\kappa_{\mathrm{D}}\p_{x^1}\cJ\,)=0,
\end{equation}
whose analytical solution is given by
\begin{equation}\label{eq:GD_approximate_solution}
    \cJ(x^1,t)=\sqrt{\frac{t_0}{t_0+t}}\,\mathrm{exp}\left(-\frac{((x^1-vt)-x^1_0)^2}{4(t_0+t)\kappa_{\mathrm{D}}}\right).
\end{equation}
We evolve the initial condition until $t=30$, at which point the peak in the solution has returned to its initial position.
The purpose of this test is to qualitatively verify that our scheme captures diffusion in a medium moving at a moderate speed.

In the left panel of Figure \ref{fig:GaussianDiffusion} the Lagrangian energy, $\cJ$, is plotted against $x^1-v\,t$ so that the peaks of the Gaussian pulses are all centered on $x^1=1$.  
Since $v=0.1$, we expect good agreement with the $\cO(v/c)$ results from \cite{laiu_etal_2025} and the approximate solution. 
Our numerical solution (open circles) agrees well with Eq.~\eqref{eq:GD_approximate_solution} (solid lines), demonstrating that our scheme captures the advection-diffusion behavior.

\subsubsection{Test II}
We also perform a second test with a purely scattering medium to simulate the advection of trapped radiation in a medium moving at a more relativistic speed.
We set $\sigma=1\times 10^3$, and let $\vect{v}=[v,0,0]^{\intercal}$, with $v=0.5$.  
The spatial domain is periodic, $\Omega_{x^1}=[-3,3]$, and is discretized using $100$ elements.
We follow the initialization in \cite{radice_etal_2022} by setting \begin{equation}
    \cJ(x^1,0)=\frac{3\cE(x^1,0)}{4W^2-1},
\end{equation}
where $\cE(x^1,0)=\exp[-9(x^1)^2]$.
Under the assumption of trapped radiation, we set $\cH^\mu(x^1,0)=0$.  
The solution is a slowly diffusing and translating pulse.  

In the right panel of Figure~\ref{fig:GaussianDiffusion} we plot the initial condition and the results at $t=1$ and $t=2$ for the Eulerian energy (dashed lines) along with a reference solution (solid red lines).  
The reference solution is generated by translating the solution to the diffusion equation,
\begin{equation}
    \p_t\cE = \kappa_{\mathrm{D}}\p^2_{x^1}\cE,
\end{equation}
along the fluid velocity.  
That is, for $t>0$,
\begin{equation}
    \cE(x^1,t)
    =
    \sqrt{\frac{\kappa_{\mathrm{D}}}{4\pi t}}
    \int_{-\infty}^\infty e^{-\frac{\kappa_{\mathrm{D}}((x^1-v\,t)-y)^2}{4t}}\cE(y,0)\,dy.
\end{equation}
Similar to \cite{radice_etal_2022}, we see good agreement between the reference solution and our numerical solutions.  

\begin{figure*}
     \centering
     \begin{subfigure}[b]{0.45\textwidth}
         \centering
         \includegraphics[width=8.6cm]{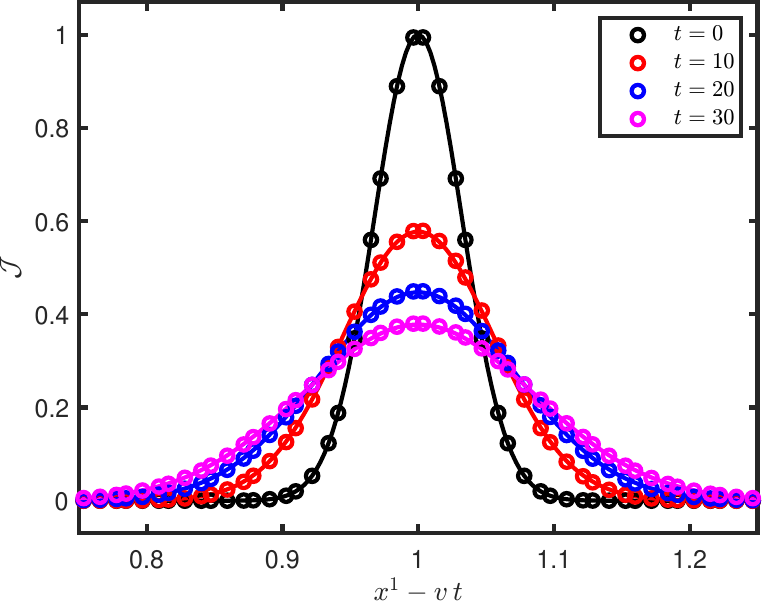}
     \end{subfigure}
     \hfill
     \begin{subfigure}[b]{0.45\textwidth}
         \centering
         \includegraphics[width=8.6cm]{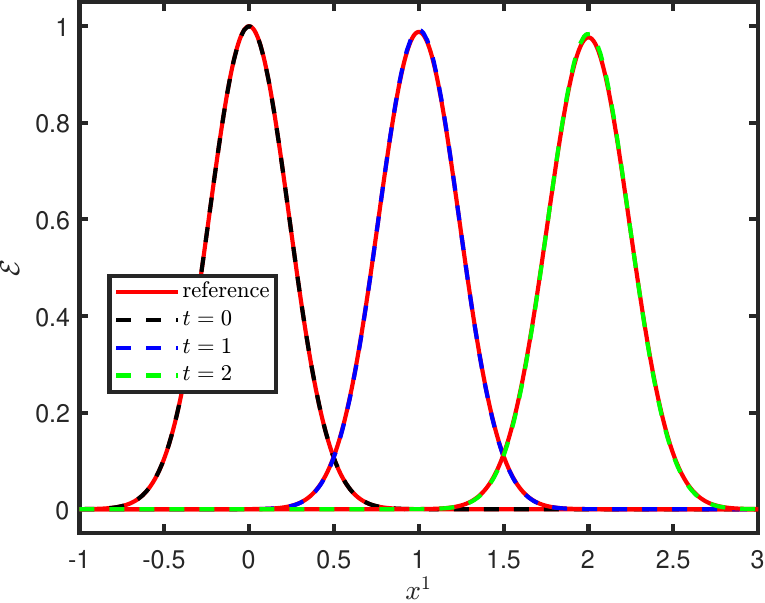}
     \end{subfigure}
    \caption{The left panel shows the results of the first diffusion test, comparing the numerical (open circles) and approximate analytic (solid lines) solutions at times $t=0,10,20,30$.  The solutions are shifted so the peaks are aligned at $x^1=1$.  The right panel shows the results of the second diffusion test, comparing numerical and reference solutions at times $t=0,1,2$.}
    \label{fig:GaussianDiffusion}
\end{figure*}

\subsection{Streaming Doppler Shift}
\label{sec:streamingDopplerShift}

In this test, we model the propagation of free-streaming radiation through a background with spatially varying velocity.
The test is adopted from \cite{vaytet_2011} (see also \cite{just_etal_2015,skinner_2019,laiu_etal_2025}), who considered methods valid to $\cO(v/c)$.  
As far as we know, this is the first time the test is performed in a special relativistic setting.  
Due to using comoving-frame momentum space coordinates for the two-moment model, the radiation energy spectra will be Doppler shifted in accordance with the background velocity.  
We consider a one-dimensional spatial domain, $\Omega_{x^1}=[0,10]$, and set the energy domain to $\Omega_\varepsilon=[0,50]$.
We set $\chi=\sigma=0$, and define the velocity as
$\vect{v}=(v,0,0)^{\intercal}$, where
\begin{equation}
    v(x^1)=
    \begin{cases}
        0, &x^1\in[0,2),\\
        v_{\mathrm{max}}\, \sin^2(\pi(x^1-2)/3), &x^1\in[2,\frac{7}{2}),\\
        v_{\mathrm{max}}, &x^1\in[\frac{7}{2},\frac{13}{2}),\\
        v_{\mathrm{max}}\,\sin^2(\pi(x^1-2)/3), &x^1\in[\frac{13}{2},8),\\
        0, &x^1\in[8,10].
    \end{cases}
\end{equation}
We let $v_{\mathrm{max}}\in\{0.0,0.1,0.3,0.6,0.9\}$ to test the model's ability to capture special relativistic effects as $v_{\mathrm{max}}$ increases.
The spatial and energy domains are discretized using $128$ and $32$ elements respectively.
The elements in the energy domain have geometrically progressing spacing with $\Delta\varepsilon_{i+1}/\Delta\varepsilon_i=1.1$.  The Lagrangian moments are initialized for all $(x^1,\varepsilon)\in \Omega_{x^1}\times \Omega_{\varepsilon}$ as
\[
\cJ=1\times10^{-40}
\quad
\text{and}
\quad
\cH^1=0.
\]
At the inner spatial boundary we impose an incoming forward-peaked radiation field with a Fermi--Dirac spectrum, setting
\begin{align*}
    \cJ(\varepsilon,x^1=0)
    &= \varepsilon/[\mathrm{exp}(\varepsilon/3-3)+1], \\
    \cH^1(\varepsilon,x^1=0)
    &= 0.999\times\cJ(\varepsilon,x^1=0).
\end{align*}
At the outer spatial boundary we impose an outflow condition.
For $t\gtrsim 10$ the solution reaches a steady state given by \cite{just_etal_2015}
\begin{equation}\label{eq:DopplerShift_steadystate}
    \cJ=
    \frac{s^2\varepsilon}{\mathrm{exp}(s\varepsilon/3-3)+1},
\end{equation}
where $s=\sqrt{(1+v)/(1-v)}$.  
The purpose of this test is to compare the numerical steady state solution with the special relativistic prediction given by Eq.~\eqref{eq:DopplerShift_steadystate} across a range of $v_{\max}$ values.  
Given the initial and boundary conditions, this is also a very challenging test with respect to maintaining moment realizability, as the solution evolves very close to the boundary of the realizable domain.  
To reach a steady state numerically, the test is run until $t=20$.  

In the left panel of Figure \ref{fig:DopplerShift_spectra} we plot spectra from the steady state solution at $x^1=5$, where the velocity is $v_{\mathrm{max}}$.  
(The spectrum for $v_{\mathrm{max}}=0$ is identical to the incoming spectrum at $x^{1}=0$.)  
As $v_{\mathrm{max}}$ increases, the energy spectra become increasingly Doppler shifted to lower energies relative to the incoming energy spectra for which $v(x^1)=0$.
The numerical spectra (dotted lines with open circles) match very well with the analytical steady state, even for large values of $v_{\mathrm{max}}$, indicating our scheme is able to capture special relativistic effects.  

In the right panel of Figure \ref{fig:DopplerShift_spectra} we plot the ${\mathrm{RMS}}$ energy, 
\begin{equation}
    \varepsilon_{\mathrm{RMS}}
    =
    \sqrt{
    \int_{\Omega_\varepsilon}\cJ\varepsilon^4\,d\varepsilon\, \bigg/ \int_{\Omega_\varepsilon}\cJ\varepsilon^2\,d\varepsilon
    },
    \label{eq:rmsEnergy}
\end{equation}
versus position for the steady state solutions.  
We expect the $\mathrm{RMS}$ energy to decrease with $x^{1}$ as the velocity increases (and vice versa).  
We observe this behavior in the numerical solutions (dotted lines with open circles), which match well with the analytical solutions (solid lines).

We have also run this test using Newton's method as the iterative solver in the moment conversion process.
We obtained similar results as in Figure \ref{fig:DopplerShift_spectra}, which uses Richardson iteration, and notably we did not observe any non-realizable iterates during the time-stepping process.  
However, we observed that the code running with Newton's method tended to be slower due to its increased computational costs.  

\begin{figure*}
     \centering
     \begin{subfigure}[b]{0.45\textwidth}
         \centering
         \includegraphics[width=8.6cm]{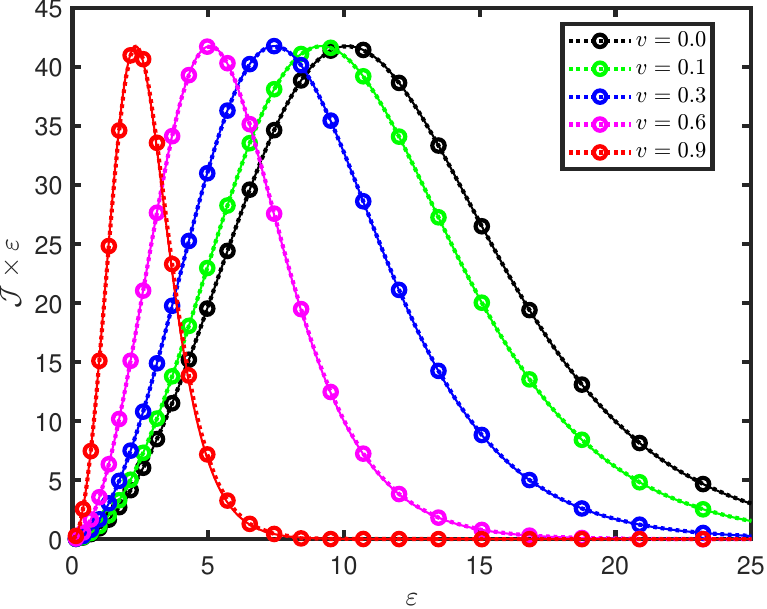}
     \end{subfigure}
     \hfill
     \begin{subfigure}[b]{0.45\textwidth}
         \centering
         \includegraphics[width=8.6cm]{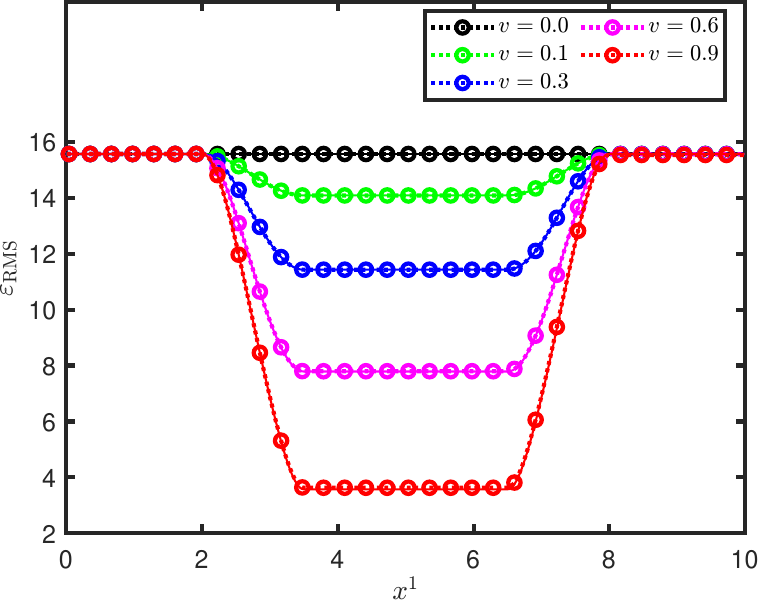}
     \end{subfigure}
    \caption{Steady state solutions ($t=20$) for the streaming Doppler shift problem with various $v_{\mathrm{max}}\in\{0.0,0.1,0.3,0.6,0.9\}$. 
    In the left panel, the spectra are plotted at $x^1=5$, with open circles representing cell centers in the $\varepsilon$-direction.
    In the right panel, the $\mathrm{RMS}$ energy defined in Eq.~\eqref{eq:rmsEnergy} is plotted versus $x^{1}$, with open circles representing every fourth cell center in the $x$-direction.
    In both panels, solid lines represent the analytical solution, while dotted lines with open circles represent the numerical solution.}
    \label{fig:DopplerShift_spectra}
\end{figure*}

\subsection{Transparent Shock}
\label{sec:transparentShock}

\response{
In this test, adapted from \cite{laiu_etal_2025}, we consider the propagation of radiation through a velocity jump for which the gradient varies.  
We consider both smooth and discontinuous transitions by letting $\vect{v}=[v,0,0]^{\intercal}$, where
\begin{equation}
\label{eq:TS_velocity}
    v(x^1)
    =\frac{1}{2}\,v_{\mathrm{max}}\times
    \big[\,1+\tanh\big(\,(x^1-1)/H\,\big)\,\big].
\end{equation}
We will vary the velocity magnitude with the parameter $v_{\mathrm{max}}$ and the gradient with the length scale $H$.
We set the opacities as $\chi=\sigma=0$, the spatial domain as $\Omega_{x^1}=[0,2]$, and the energy domain as $\Omega_{\varepsilon}=[0,300]$.
Unless stated otherwise, these domains are discretized using $80$ and $32$ elements, respectively.  
With this resolution in the energy dimension, elements have geometrically progressing spacing with $\Delta \varepsilon_{i+1} / \Delta \varepsilon_i = 1.119237083677839$.  
The moments are initially set for all $(\varepsilon,x^1)\in \Omega_{\varepsilon}\times \Omega_{x^1}$ as
\[
    \cJ=1\times 10^{-8}
    \quad
    \text{and}
    \quad
    \cH^1=(1-\delta)W\cJ,
\]
where $\delta = 10^{-8}$.  
That is, the initial moments are very close to the boundary of the realizable domain.  
We use the same boundary conditions as in the streaming Doppler shift test, 
except at the inner spatial boundary
\[
    \cH^1(\varepsilon,x^1 = 0) = (1-\delta)W(x^1 = 0)\cJ(\varepsilon,x^1 = 0).
\]
}
\response{
The moments are evolved until $t=3$ when an approximate steady state is reached, again given by Eq.~\eqref{eq:DopplerShift_steadystate}.
This test is challenging for two main reasons: (1) the solution evolves very close to the boundary of the realizable domain and (2) the velocity profile is essentially discontinuous for small values of $H$.
We will first present results under the setting outlined above with $v_{\mathrm{max}}\in\{-0.1,-0.5\}$ and $H\in\{3\times10^{-2}, 10^{-2}, 10^{-3}\}$. 
We will follow up these results with a discussion on the effect of spatial mesh refinement, varying the polynomial degree, and the use of a slope limiter to control nonphysical oscillations.  
}

\subsubsection{$v_{\mathrm{max}} = -0.1$}

\begin{figure*}
     \centering
     \begin{subfigure}[b]{0.45\textwidth}
         \centering
         \includegraphics[width=8.6cm]{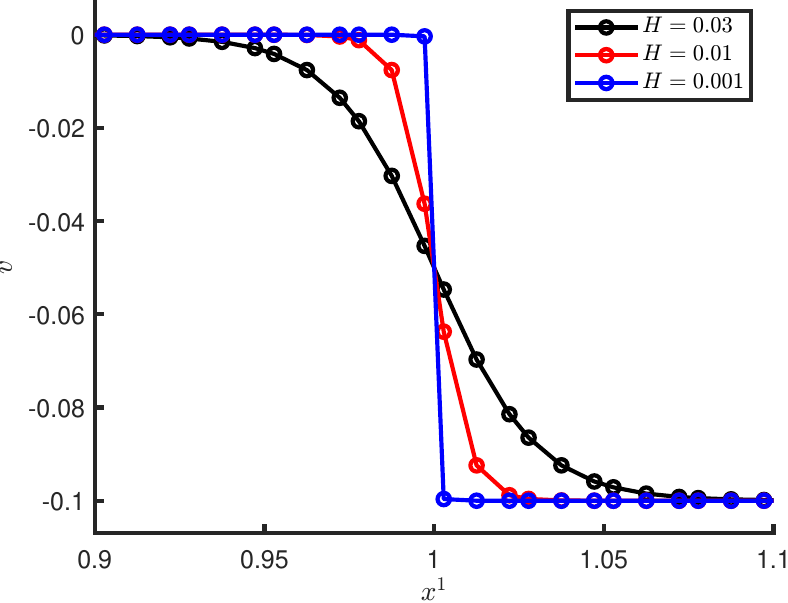}
     \end{subfigure}
     \hfill
     \begin{subfigure}[b]{0.45\textwidth}
         \centering
         \includegraphics[width=8.6cm]{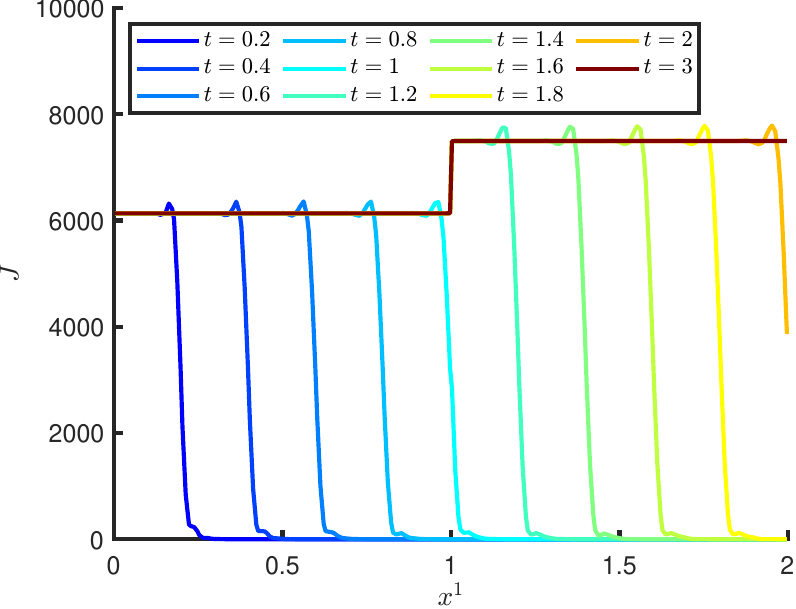}
     \end{subfigure}
        \caption{
        \response{
        Transparent Shock test for $v_{\mathrm{max}}=-0.1$.
        The left panel plots velocity profiles given by Eq.~\eqref{eq:TS_velocity} for various values of $H$.
        The right panel plots the time evolution of $J$, the Lagrangian energy density, for the case with $H = 10^{-3}$.
        }
        }
        \label{fig:TS_v01_velocity+TE}
\end{figure*}

\begin{figure*}
    \centering
    \begin{subfigure}[b]{0.45\textwidth}
        \centering
        \includegraphics[width=8.6cm]{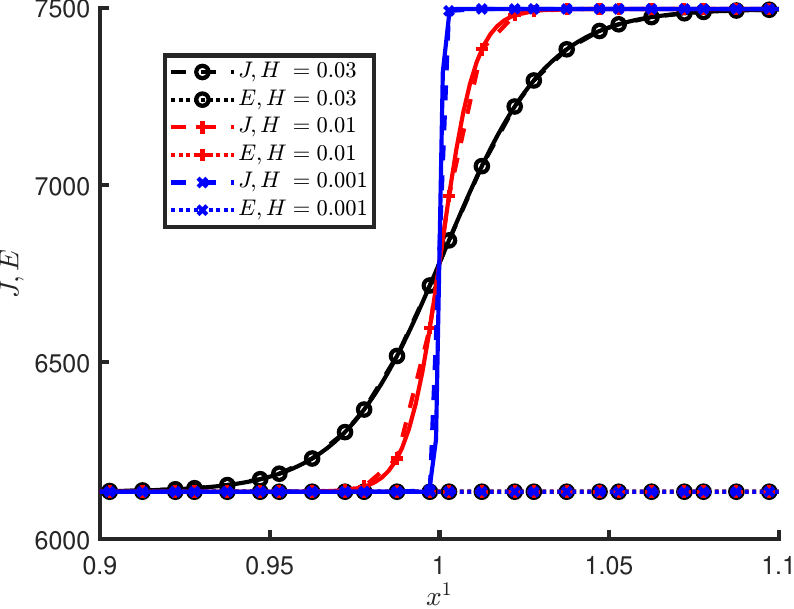}
    \end{subfigure}
    \hfill
    \begin{subfigure}[b]{0.45\textwidth}
        \centering
        \includegraphics[width=8.6cm]{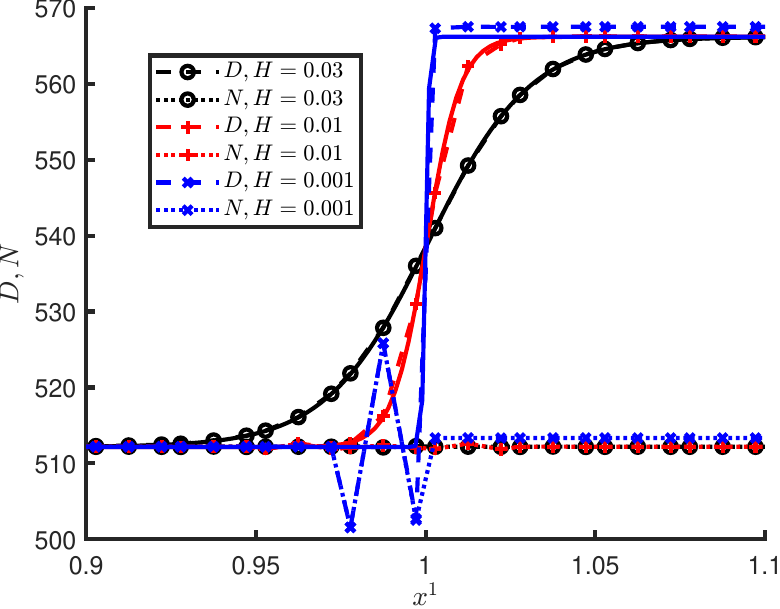}
    \end{subfigure}
    \caption{
    \response{
        Transparent Shock test for $v_{\mathrm{max}}=-0.1$.
        The left panel plots the Lagrangian and Eulerian energy densities, $J$ (dashed) and $E$ (dotted), respectively, for all $H$.
        The right panel plots the Lagrangian and Eulerian number densities, $D$ (dashed) and $N$ (dotted), respectively, for all $H$.  
        The analytic solutions for $J$ and $D$ are plotted with solid lines, with line colors matching the corresponding numerical solution.
        }
        }
        \label{fig:TS_v01}
\end{figure*}

\response{
In the left panel of Figure \ref{fig:TS_v01_velocity+TE}, we plot the velocity profiles for each length scale $H$.  
With the spatial grid $\Delta x^1 = 0.025$, we have $\Delta x / H \in \{5/6, 5/2, 25\}$ for the length scales we consider, which implies that the ``shock'' is resolved for $H=3\times 10^{-2}$, under-resolved for $H=10^{-2}$, and discontinuous for $H=10^{-3}$.
The right panel illustrates the time evolution of $J$ when $H = 10^{-3}$.
As time evolves, a front, represented by a discontinuity in $J$, propagates from left to right through the domain, and the solution settles to a steady state in its wake.  
The steady state solution at $t=3$ is constant before and after the shock, and the size of the jump across the shock increases with the magnitude of $v_{\max}$.  
As stated before, one difficulty in this test is capturing the change in $J$ when the velocity jump around $x^1 = 1$ is sharp.
}

\response{
Figure \ref{fig:TS_v01} shows the results of the test for $v_{\max}=-0.1$ with the length scales $H=\{3\times 10^{-2}, 10^{-2}, 10^{-3}\}$.
We plot Lagrangian and Eulerian energy and number densities versus position, 
\begin{equation}
    \{J,E,D,N\}=\int_{\Omega_\varepsilon}\{\cJ,\cE,\cJ/\varepsilon,\cN\}\varepsilon^2\,d\varepsilon.
    \label{eq:grayMoments}
\end{equation}
The Eulerian spectral number density, $\cN$, is given by Eq.~\eqref{eq:eulerianNumberDensity}.  
The left panel plots $J$ (dashed lines), the Lagrangian energy density, versus position for each length scale.  
For $v_{\max}=-0.1$ we observe good agreement between the numerical (dashed) and analytic (solid) solutions for $J$ for all length scales $H$.  
To quantify errors, we define the error in a quantity $X\in\{J,D\}$ relative to the analytic solution $X_{\rm A}$ as
\begin{equation}
    ||\delta X||_{2}^{\pm}
    =\sqrt{\int_{\Omega_{x}^{\pm}}(X-X_{\rm A})^{2}dx/\int_{\Omega_{x}^{\pm}} X_{\rm A}^{2}dx},
\end{equation}
where the spatial domains to the left and right of the shock are $\Omega_{x}^{-}=[0,1]$ and $\Omega_{x}^{+}=[1,2]$, respectively.  
For the steady state Lagrangian energy density, we find $||\delta J||_{2}^{-}=\{\,2\times10^{-7},\,3\times10^{-8},\,7\times10^{-8}\,\}$ and $||\delta J||_{2}^{+}=\{\,6\times10^{-9},\,3\times10^{-9},\,4\times10^{-8}\,\}$, for $H=\{3\times10^{-2},\,10^{-2},\,10^{-3}\,\}$, respectively.  }

\response{
In the same panel we also plot the Eulerian energy density $E$ (dotted lines).
In a steady state, the Eulerian energy density should remain constant across the spatial domain --- including across the velocity jump --- in this test.  
For $v_{\max}=-0.1$, $E$ is constant to machine precision before the shock, where $\vect{v}=0$.  
We observe a relative jump in $E$ on the order of $10^{-8}$ across the shock, and $E$ remains relatively constant from the shock to the outer boundary.  
For $H = 10^{-3}$, $E$ is noticeably oscillatory after the shock, whereas for the other two values of $H$, $E$ is relatively smooth.  
We note that no slope limiter is applied in the numerical methods to generate these results.
These oscillations can be removed or reduced by increasing the number of spatial elements to resolve the shock or using a slope limiter (see Sections~\ref{sec:TS.SpatialMeshRefinement} and \ref{sec:TS.SlopeLimiter}, respectively).  
This demonstrates that our method captures Doppler shifts correctly --- even when velocity gradients are large.  
However, our results illustrate challenges with using comoving-frame momentum coordinates, which requires approximating four-velocity gradients---see Eq.~\eqref{eq:fourvelocitygradient}---that can become inaccurate with finite spatial resolution, as we will investigate further in this section.  
We mention that we tried replacing approximate derivatives with exact derivatives, but this did not result in any noticeable improvements.}  

\response{
We solve conservation equations for the Eulerian-frame energy and momentum densities, and our proposed scheme conserves these quantities to machine precision, as illustrated in Figure~\ref{fig:TS_v01_EnergyChangeVsTime}, which shows the relative change in the Eulerian-frame energy versus time (normalized to the energy in the computational domain at $t=3$).  
Similar to conservation plots in \cite{laiu_etal_2025}, energy fluxes through the boundaries of the computational domain are included in the energy balance accounting.  
This is a nontrivial result for this challenging test, where the solution evolves very close to the boundary of the realizable domain.  
(Flux factors, both $\cF/\cE$ and $\cH/\cJ$, are found to be $1-\cO(10^{-8})$ across the spatial domain.)  
By construction, cell-averaged moments remain realizable at all times, but the realizability-enforcing limiter is frequently triggered to recover pointwise realizability of moments in a way that ensures our scheme is simultaneously conservative \emph{and} realizability-preserving.}  

\response{
The right panel in Figure~\ref{fig:TS_v01} is similar to the left panel, but plots the Lagrangian and Eulerian number densities, $D$ (dashed) and $N$ (dotted), respectively. 
As discussed at the end of Section~\ref{sec:model.EvolutionEquations}, our method does not evolve the number density directly, and is not designed to conserve the Eulerian-frame particle number.  
Yet, it is interesting to see that, for $H=3\times10^{-2}$ and $H=10^{-2}$, $D$ matches the analytic steady state solution very well, and that $N$ is relatively constant across the domain.  
In these two cases, $||\delta D||_{2}^{\pm}=7\times10^{-5}$ and $10^{-4}$, respectively, and the relative change in $N$ across the domain is on the order of $10^{-4}$ away from the shock, but spikes to $\sim10^{-3}$ in the shock.  
The method does not handle the case with $H=10^{-3}$ as well as the other two, as $D$ is slightly larger than the analytic solution after the shock ($||\delta D||_{2}^{-}=3\times10^{-3}$ and $||\delta D||_{2}^{+}=2\times10^{-3}$), there are visible oscillations in $D$ and $N$ just before the shock with relative amplitude of $\sim2\times10^{-2}$, and there is a slight relative increase in $N$ of order $10^{-3}$ after the shock.  
To improve these results, we suspect that additional care in the discretization of the two-moment model in Eq.~\eqref{eq:momentEquations} will be necessary to ensure better consistency with the number equation in Eq~\eqref{eq:eulerianNumberEquation} --- e.g., as discussed in Appendix~D in \cite{cardall_etal_2013a}.  
}

\begin{figure}
    \centering
    \includegraphics[width=8.6cm]{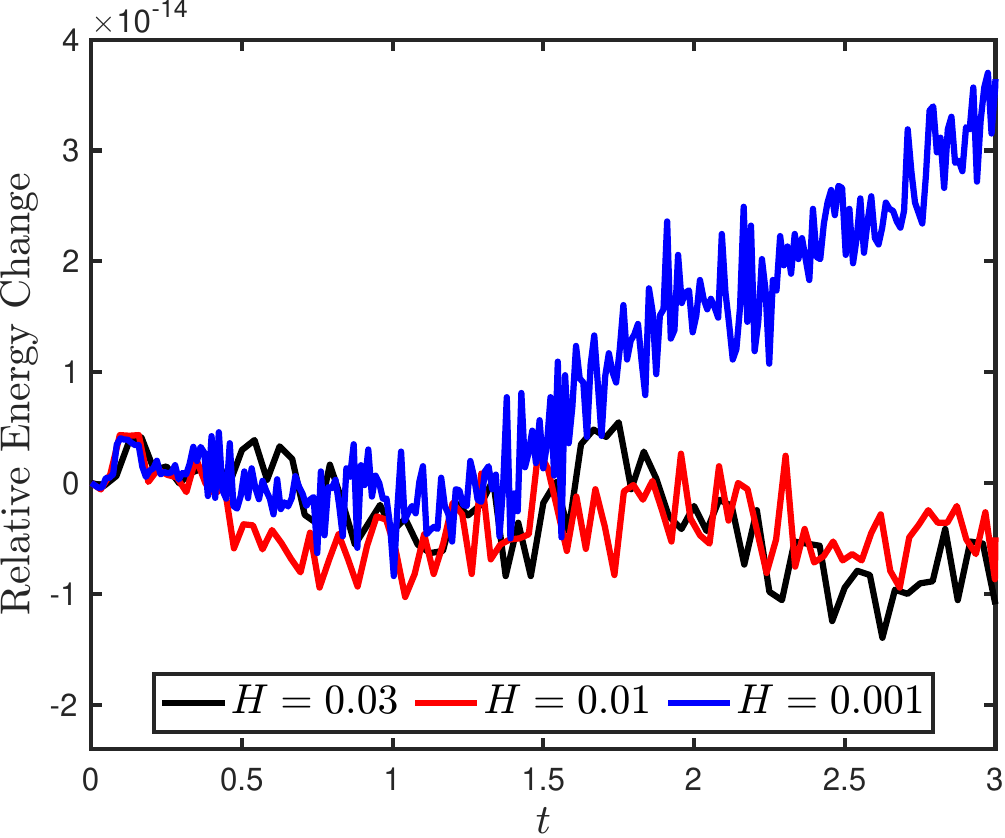}
    \caption{
    \response{
    Relative change in the total Eulerian frame energy in the computational domain versus time for the case with $v_{\max}=-0.1$ and various values of $H$.  
    }
    }
    \label{fig:TS_v01_EnergyChangeVsTime}
\end{figure}

\subsubsection{$v_{\mathrm{max}} = -0.5$}

\begin{figure*}
    \centering
    \begin{subfigure}[b]{0.45\textwidth}
        \centering
        \includegraphics[width=8.6cm]{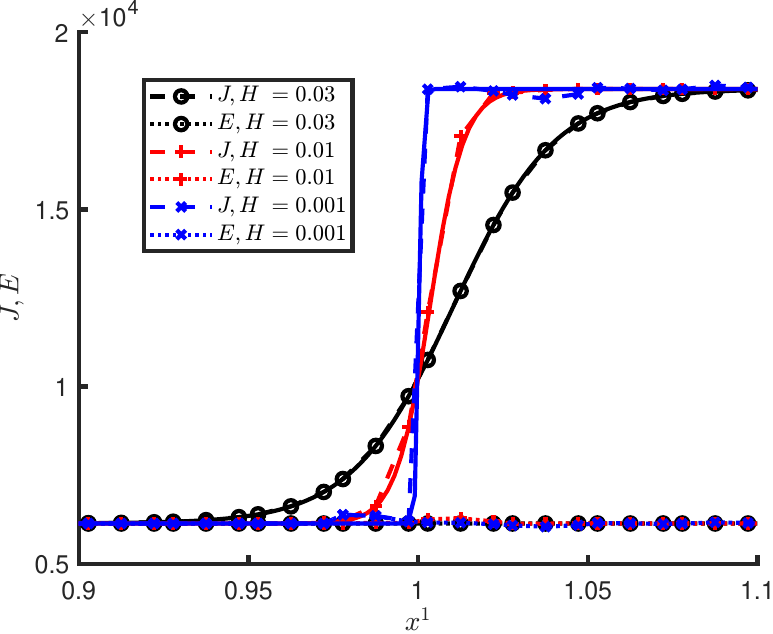}
    \end{subfigure}
    \hfill
    \begin{subfigure}[b]{0.45\textwidth}
        \centering
        \includegraphics[width=8.6cm]{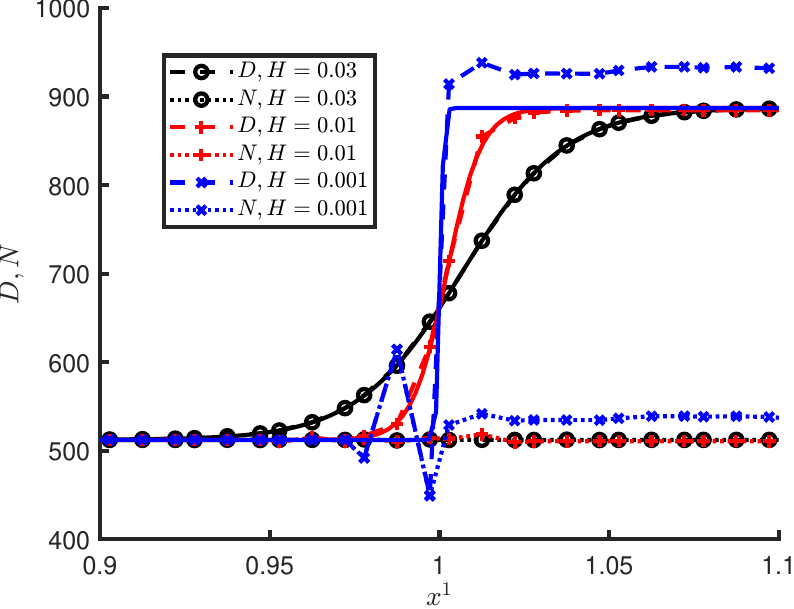}
    \end{subfigure}
    \caption{
    \response{
    Transparent Shock test for $v_{\mathrm{max}}=-0.5$.  The plotted quantities correspond to the ones in Figure~\ref{fig:TS_v01}.
    }
    }
    \label{fig:TS_v05}
\end{figure*}

\response{
Figure~\ref{fig:TS_v05} shows the same quantities as Figure~\ref{fig:TS_v01}, but for the more relativistic case with $v_{\mathrm{max}}=-0.5$.
The larger velocity magnitude makes the test even more challenging.  
Despite this challenge, our proposed method is able to capture the Lagrangian energy density across all length scales, however $J$ becomes visibly oscillatory after the shock for $H=10^{-3}$.  
(We also observe smaller oscillations for $H=10^{-2}$.)
Moreover, $||\delta J||_{2}^{-}=\{\,5\times 10^{-7},\,3\times10^{-4},\,5\times10^{-3}\,\}$ and $||\delta J||_{2}^{+}=\{\,2\times10^{-5},\,2\times10^{-3},\,5\times10^{-3}\,\}$, for $H=\{\,3\times10^{-2},\,10^{-2},\,10^{-3}\,\}$.  
The larger velocity jump also impacts the evolution of $E$ after the shock.  
The case with $H=3\times10^{-2}$ is the only one to maintain relative changes in $E$ below $10^{-6}$ away from the shock, spiking to $\sim10^{-4}$ in the shock.  
The relative jump in $E$ across the shock increases as $H$ decreases, and $E$ exhibits oscillatory behavior after the shock.  
A trend of increasing error with decreasing $H$ is also reflected in the results for $D$ and $N$.  
For $H=10^{-3}$, $D$ is visibly different from the analytic solution after the shock, there are oscillations in $D$ and $N$ just before the shock, and there is a noticeable relative jump in $N$ of about $5\%$ across the shock.  
Specifically, we find $||\delta D||_{2}^{-}=\{\,2\times10^{-4},\,8\times10^{-4},\,2\times10^{-2}\,\}$ and $||\delta D||_{2}^{+}=\{\,8\times10^{-5},\,3\times10^{-3},\,5\times10^{-2}\,\}$, for $H=\{\,3\times10^{-2},10^{-2},\,10^{-3}\,\}$.  
}

\response{
The sustained oscillations in the numerical solution, seen after the initial transient (illustrated in Figure~\ref{fig:TS_v01_velocity+TE}) has propagated through the domain and the solution has settled into a quasi-steady state, are in part caused by intermittent triggering of the realizability-enforcing limiter \cite{laiu_etal_2025}.  
The inner boundary condition is set to ensure that the steady state solution is in the free-streaming limit (flux factor is $1-10^{-8}$), which facilitates comparison with the analytic solution in Eq.~\eqref{eq:DopplerShift_steadystate}.  
For $t=3$, we find that the energy-averaged flux factor, computed with either Eulerian or Lagrangian moments, remains very close to unity across the computational domain; specifically, $\cF/\cE,\cH/\cJ=1-\cO(10^{-8})$.  
Despite evolution close to the boundary of the realizable domain and frequent triggering of the realizability-enforcing limiter, our proposed method maintains its conservation properties in this more relativistic case, as can be seen in the plot of the relative change in the Eulerian-frame energy versus time in Figure~\ref{fig:TS_v05_EnergyChangeVsTime}.  
At the end of the simulation, the relative change in total energy for the case with $v_{\max}=-0.5$ and $H=10^{-3}$ (blue line in Figure~\ref{fig:TS_v05_EnergyChangeVsTime}) is almost a factor of 40 larger than the model with $v_{\max}=-0.1$ and $H=10^{-3}$ (blue line in Figure~\ref{fig:TS_v01_EnergyChangeVsTime}).  
The reason for this is the accumulation of errors as a result of the smaller time step, due to the larger velocity magnitude and the time step's dependence on the velocity gradient through $a^{\varepsilon}$ (recall Eq.~\eqref{eq:realizable_timestep}), and the corresponding increased number of time steps needed for the more relativistic models to complete.  
For the case with $v_{\max}= - 0.1$, we find $\dt=\{\,10^{-3},\,8\times10^{-4},\,3\times10^{-4}\,\}$ for $H=\{\,3\times10^{-2},\,10^{-2},\,10^{-3}\,\}$.  
For the more relativistic model with $v_{\max}=-0.5$, we find $\dt=\{\,1.3\times10^{-4},\,4.5\times10^{-5},\,1.1\times10^{-5}\,\}$ for $H=\{\,3\times10^{-2},\,10^{-2},\,10^{-3}\,\}$.  
We note that the time step restriction is a sufficient---not necessary---condition to maintain realizability.  
In practical applications, a larger time step may be taken, but this will be problem dependent.  
However, due to the possibly severe time step restriction for maintaining realizability with explicit integration of energy derivative terms in the moment equations, implicit integration of these terms, e.g., as explored in \cite{Kuroda_2016}, may be a fruitful avenue for future investigation.}

\begin{figure}
    \centering
    \includegraphics[width=8.6cm]{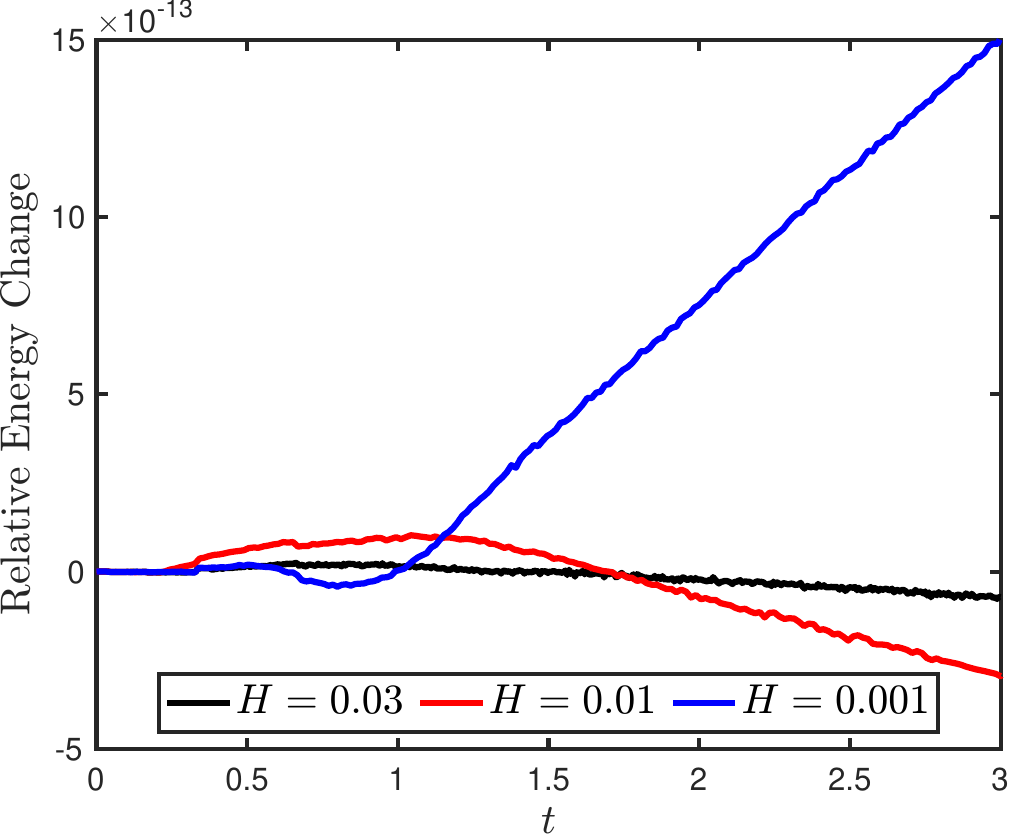}
    \caption{
    \response{
    Relative change in the total Eulerian frame energy versus time for the case with $v_{\max}=-0.5$ and various values of $H$.  
    }
    }
    \label{fig:TS_v05_EnergyChangeVsTime}
\end{figure}

\subsubsection{Spatial Mesh Refinement}
\label{sec:TS.SpatialMeshRefinement}
\response{
The results obtained for the under-resolved ($H = 10^{-2}$) and discontinuous ($H = 10^{-3}$) shock cases with $v_{\max} = -0.5$ exhibit features associated with the sharp velocity gradient around $x^{1}=1$.
Notably, the steady state solutions for the Lagrangian energy density $J$ are visibly oscillatory after the shock.  
The same can be said for the Lagrangian number density, which is also visibly offset from the analytic solution when $H=10^{-3}$.  
}

\begin{figure*}
    \centering
    \begin{subfigure}[b]{0.45\textwidth}
        \centering
        \includegraphics[width=8.6cm]{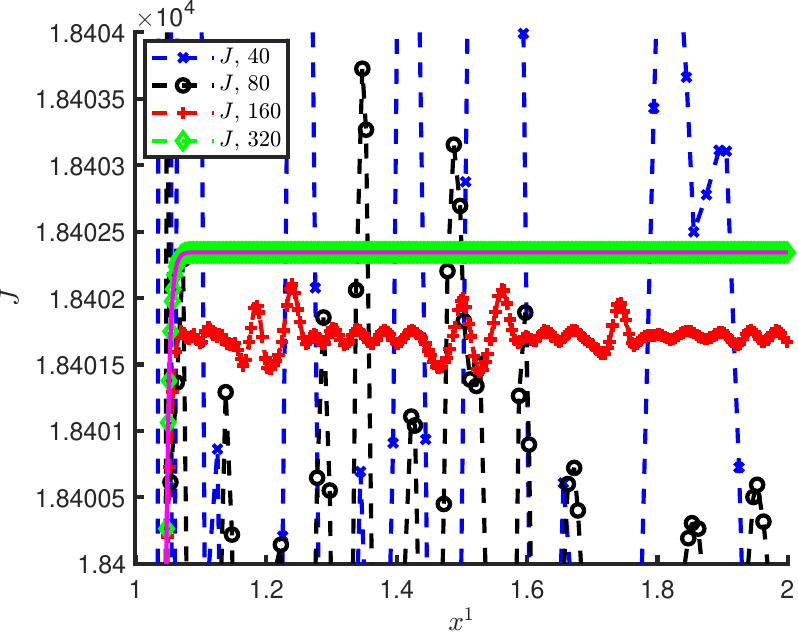}
    \end{subfigure}
    \hfill
    \begin{subfigure}[b]{0.45\textwidth}
        \centering
        \includegraphics[width=8.6cm]{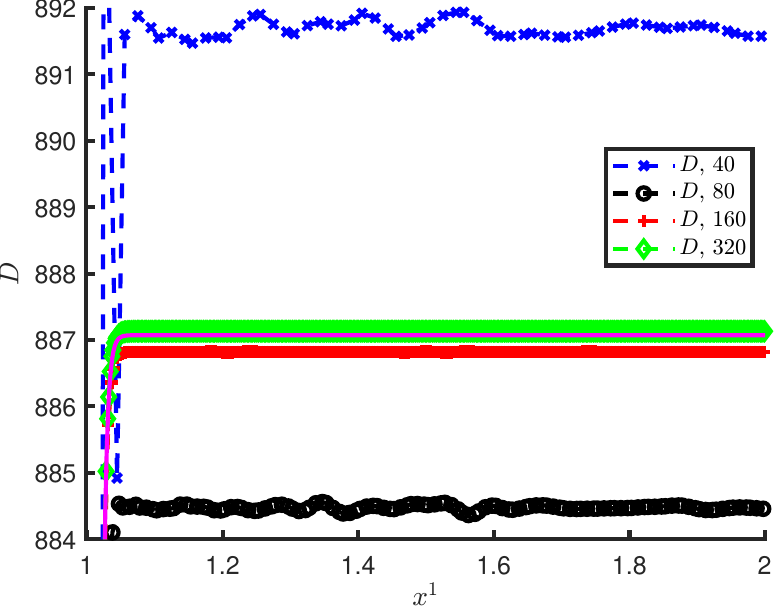}
    \end{subfigure}
    \caption{
    \response{
        Spatial mesh refinement results for $v_{\mathrm{max}} = -0.5$ and $H = 10^{-2}$ using quadratic polynomials.
        The left and right panels illustrate the effect of spatial mesh refinement on $J$ and $D$, respectively.
        The number of spatial elements are $40$ (blue), $80$ (black), $160$ (red), and $320$ (green).
        The number of energy elements is $32$ for each spatial mesh.  
        The analytic solution is plotted with solid magenta line in both panels.  
        }
        }
        \label{fig:TS_v05_meshrefinement}
\end{figure*}

\response{
Here we study the effects of spatial mesh refinement for the case with $v_{\mathrm{max}} = -0.5$ and $H = 10^{-2}$.   
Beginning with 40 spatial elements, we refine the mesh until $\Delta x^1/H < 1$ to resolve the shock.  
We double the number of spatial elements, going up to $320$, so that $\Delta x^1/H \in \{5, 5/2, 5/4, 5/8\}$.  
Figure \ref{fig:TS_v05_meshrefinement}, which plots $J$ (left panel) and $D$ (right panel) versus position, clearly illustrates the effect spatial mesh refinement has on improving the quality of both $J$ and $D$ after the shock.
As the shock becomes better resolved, the numerical solution for both $J$ and $D$ (dashed lines with markers) approaches the analytic solution (solid magenta lines) and oscillations are significantly reduced.  
With $\Delta x^1/H = 5/8$ (green curves), the shock is resolved and there are no visible oscillations in $J$ or $D$.
Specifically, we find $||\delta J||_{2}^{+}=\{\,10^{-3},\,2\times10^{-3},\,3\times10^{-4},\,10^{-8}\,\}$ and $||\delta D||_{2}^{+}=\{\,6\times10^{-3},\,3\times10^{-3},\,3\times10^{-4},\,7\times10^{-5}\,\}$ for $\{40,\,80,\,160,\,320\,\}$ spatial elements, respectively.
}

\response{
Based on these spatial mesh refinement results, we believe that sharp gradients in the fluid four-velocity are a dominant source of errors in numerical methods for the conservative spectral two-moment model based on comoving-frame momentum coordinates.  
Refining the spatial mesh to resolve the shock when $H=10^{-3}$ is computationally expensive as it would require $2560$ spatial elements.  
In practice, it may be necessary to apply smoothing to the four-velocity supplied to the phase-space advection solver to avoid numerical artifacts.  
}

\subsubsection{Polynomial Degree}
\label{sec:TS.PolynomialDegree}
\response{
In this section, we investigate the effect of the polynomial degree, $k$, on the solution in the most challenging case of $v_{\max} = -0.5$ and $H = 10^{-3}$.  
We keep the total number of degrees of freedom the same for each polynomial degree.
Thus, for quadratic polynomials ($k=2$) we use $32$ energy elements and $60$ spatial elements; for linear polynomials ($k=1$), $48$ energy elements and $90$ spatial elements; and for constant polynomials ($k=0$), $96$ energy elements and $180$ spatial elements.  
(The case with $k=0$ is identical to a first-order finite-volume method.)  
When changing the number of energy elements, we adjust the geometric progression of the element widths to maintain a similar distribution of points.  
For $k=0$, we use $\Delta\varepsilon_{i+1}/\Delta\varepsilon_{i}=1.038647428867211$, while for $k=1$, we use $\Delta\varepsilon_{i+1}/\Delta\varepsilon_{i}=1.019368113873667$.  
We use SSPRK3 time stepping for all values of $k$.  
}

\response{
Figure \ref{fig:TS_v05_PolynomialComparison} presents the results of this investigation.  
The left panel plots $J$ versus position around the shock.
Numerical solutions obtained with constant, linear, and quadratic polynomials capture the features of the analytic solution well (modulo oscillations for $k>0$).  
The constant polynomial solution is the only case that remains nonoscillatory.  
The right panel shows $D$ versus position around the shock.  
In addition to displaying oscillatory behavior for $k>0$, results for all values of $k$ indicate larger errors in the jump in $D$ across the shock.  
Specifically, $||\delta J||_{2}^{+}=\{\,6\times10^{-3},\,2\times 10^{-3},\,5\times 10^{-3}\,\}$ and $||\delta D||_{2}^{+}=\{\,4\times10^{-2},\,6\times10^{-2},\,5\times10^{-2}\}$, for polynomial degree $k=\{\,0,\,1,\,2\,\}$.  
In future work, we believe it would be worthwhile to investigate structure-preserving methods for the two-moment model that maintain consistency with the number conservation equation --- with coarse spatial meshes --- to see if results for this test improve.
}

\begin{figure*}
    \centering
    \begin{subfigure}[b]{0.45\textwidth}
        \centering
        \includegraphics[width=8.6cm]{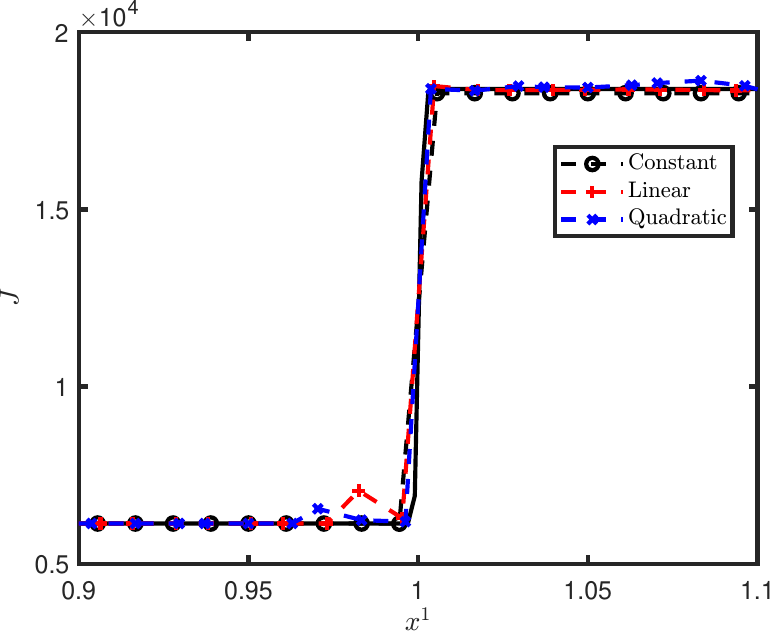}
    \end{subfigure}
    \hfill
    \begin{subfigure}[b]{0.45\textwidth}
        \centering
        \includegraphics[width=8.6cm]{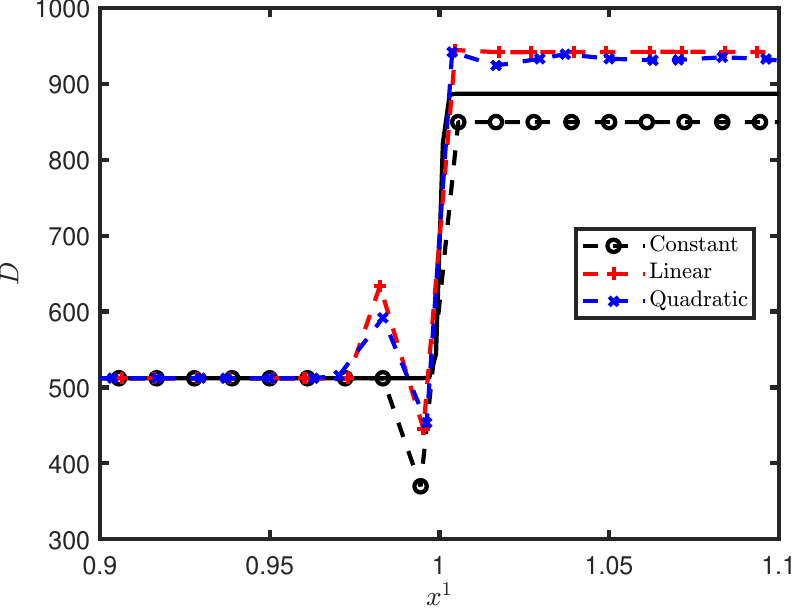}
    \end{subfigure}
    \caption{
        \response{The effect of varying the polynomial degree is illustrated for $J$ (left panel) and $D$ (right panel) in the most challenging case considered for the Transparent Shock test: $v_{\max} = -0.5$ and $H = 10^{-3}$.
        The solid black line is the analytic solution, while the dashed lines are the numerical solutions for constant (black), linear (red), and quadratic (blue) polynomials.
        }
        }
        \label{fig:TS_v05_PolynomialComparison}
\end{figure*}

\subsubsection{Slope Limiter}
\label{sec:TS.SlopeLimiter}
\response{
The oscillations present in the numerical solutions after the shock for $k>0$ can be damped using a slope limiter.  
In Figure \ref{fig:TS_v05_TVD} we show the effects of using a TVD slope limiter in the challenging case with $v_{\max} = -0.5$ and $H = 10^{-3}$, using quadratic polynomials on a mesh with $32$ energy elements and $60$ spatial elements.  
The minmod-type TVD limiter we use is applied to the polynomial representation in position space, and is described in Section~3.3 of \cite{Pochik_2021}.  
Specifically, for $\beta_{\mathrm{TVD}}\in[1,2]$, we set $\beta_{\mathrm{TVD}} = 1.25$, and the limiter is applied component-wise to the conserved variables $\bU$, as opposed to characteristic variables.
As can be seen in the plot of $J$ versus position, the limiter damps the nonphysical oscillations present in the numerical solution.
The damping of oscillations carries over to $E$, $D$, and $N$.  
For these two models, we find $||\delta J||_{2}^{+}=6\times10^{-3}$ (without limiter) and $||\delta J||_{2}^{+}=8\times10^{-4}$ (with limiter), and $||\delta D||_{2}^{+}=5\times10^{-2}$ (with and without limiter).  
}

\begin{figure}
    \centering
    \includegraphics[width=8.6cm]{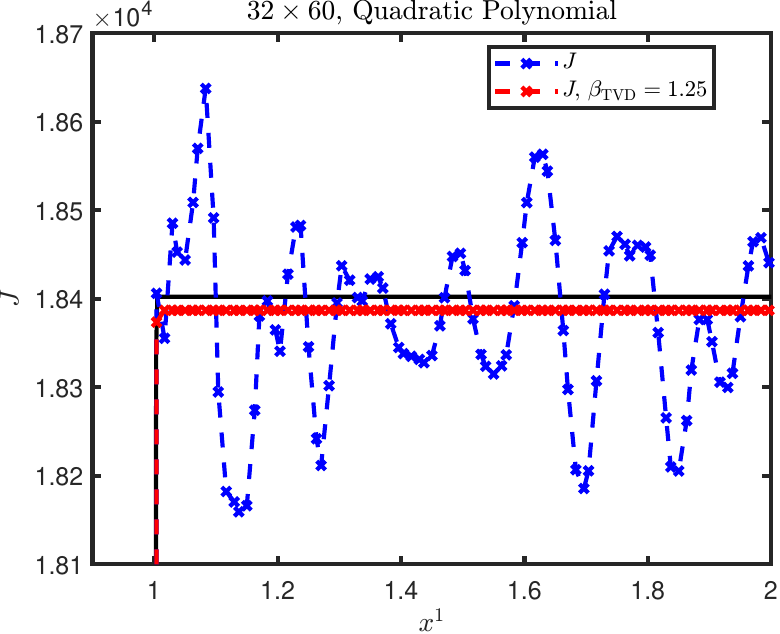}
    \caption{
    \response{
    Plot of $J$ versus position, illustrating the effect of a slope limiter for the case with $v_{\max} = -0.5$ and $H = 10^{-3}$, using quadratic polynomials.
    The blue line is without slope limiter, while the red line is obtained with the minmod-type TVD slope limiter described in \cite{Pochik_2021}, using $\beta_{\mathrm{TVD}} = 1.25$.  
    The solid black line is the analytical solution.  
    }
    }
    \label{fig:TS_v05_TVD}
\end{figure}

\subsection{Shadow Casting}

In this 2D test, we consider a material at rest and set up a source region of radiation and an absorbing region.  
The goal of the test is to demonstrate that our code can capture the radiation shadow cast by the absorbing region as the radiation propagates from the source region.  
Similar setups are frequently used to demonstrate the advantage of two-moment models over flux-limited diffusion models in capturing shadows (e.g., \cite{just_etal_2015,Kuroda_2016,Anninos_2020,radice_etal_2022}).  
Following \cite{just_etal_2015}, we set up a Cartesian domain of $\Omega_{x^1}\times \Omega_{x^2}=[0,15]\times[-5,5]$, discretized into $300\times200$ elements.
We use outflow boundary conditions in each direction.
The fluid velocity is set to $\vect{v}=0$.
The source region, $\cS$, is defined as the circle of radius $r_{\cS}=1.5$ centered at the point $(x^1,x^2)=(3,0)$.
The absorbing region, $\mathcal{A}$, is defined as the circle of radius $r_{\mathcal{A}}=2$ centered at the point $(x^1,x^2)=(11,0)$.
We set $\sigma=0$, while $\chi$ and $\mathcal{J}_0$ are determined as follows
\begin{align}
    \chi(\vect{x})
    =&
    \begin{cases}
        10\exp[-(4|\vect{x}-\vect{x}_{\cS}|/r_{\cS})^2], &\vect{x}\in\cS \\
        10, &\vect{x}\in\cA \\
        0, &\text{elsewhere},
    \end{cases}\\
    \nonumber\\
    \cJ_{\Equilibrium}(\vect{x})
    =&
    \begin{cases}
        10^{-1}, &\vect{x}\in\cS \\
        0, &\text{elsewhere}.
    \end{cases}
\end{align}
Throughout this section, $\vect{x}=(x^1,x^2)$.
We initialize the Lagrangian moments as
\[
    \cJ(\vect{x},0)=10^{-10},
    \quad
    \cH^i(\vect{x},0)=0,
    \quad
    \text{for all}
    \,\,\,
    \vect{x}\in \Omega_{x^1}\times \Omega_{x^2}.
\]
We run the test until $t=15$, when an approximate steady state has been established.

In Figure \ref{fig:Shadow_Casting} we plot a contour map of the isotropic luminosity emitted by the source region,
\begin{equation}
    L=2\pi|\vect{x}-\vect{x}_{\cS}|\sqrt{\cF_\mu\cF^\mu}.
\end{equation}
Similar to \cite{just_etal_2015}, the plots depict the solution at three different time values to illustrate the emission of radiation from the source and the formation of the shadow, demonstrating our code is able to capture radiation shadows.  

\begin{figure}
    \centering
    \begin{subfigure}[b]{0.45\textwidth}
         \centering
         \includegraphics[width=8.6cm]{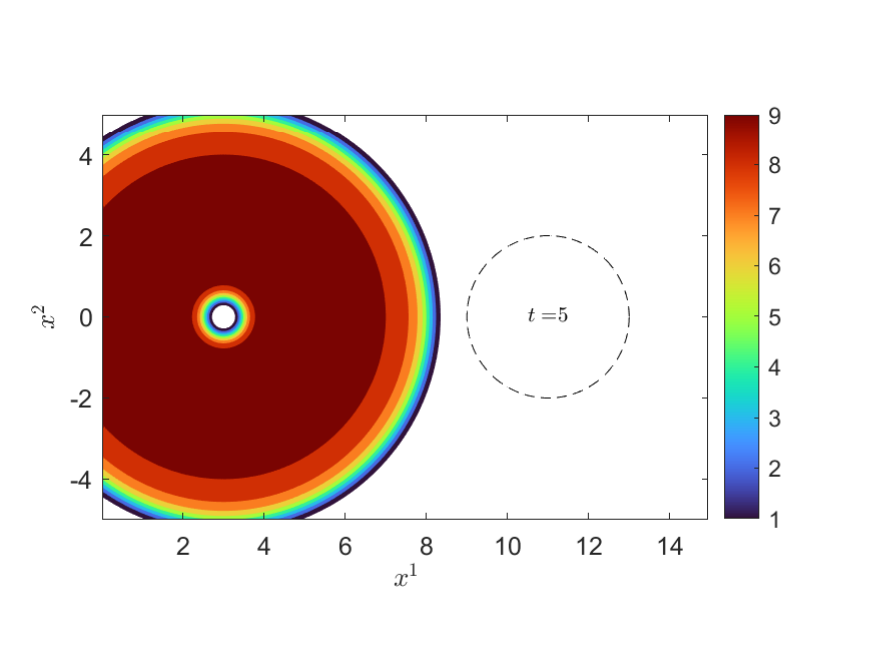}
     \end{subfigure}
     \hfill
     \begin{subfigure}[b]{0.45\textwidth}
         \centering
         \includegraphics[width=8.6cm]{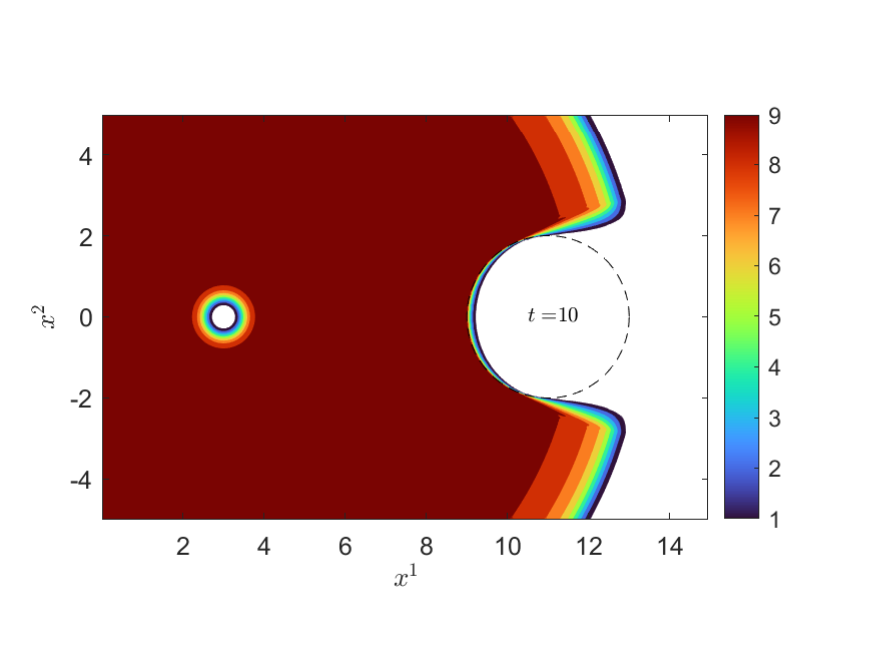}
     \end{subfigure}
     \hfill
     \begin{subfigure}[b]{0.45\textwidth}
         \centering
         \includegraphics[width=8.6cm]{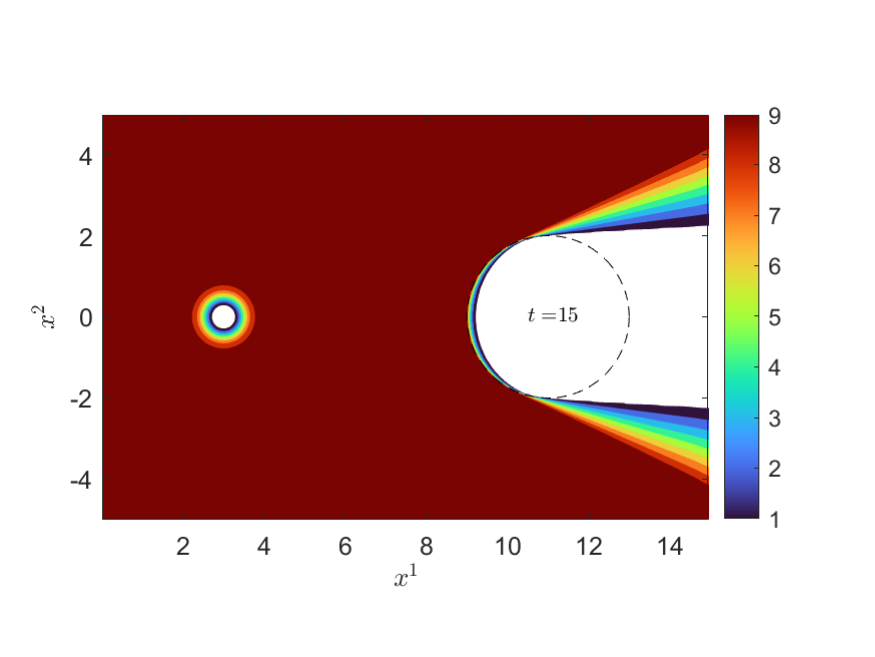}
     \end{subfigure}
    \caption{Contour plots of $10^2\times L$, where $L=2\pi|\vect{x}-\vect{x}_S||\cF|$, for the Shadow Casting problem at various times.  
    The dashed circle centered at $(x^{1},x^{2})=(11,0)$ indicates the boundary of the absorbing region.}
    \label{fig:Shadow_Casting}
\end{figure}

\subsection{Transparent Vortex}

In this test, we consider a two-dimensional spatial domain, $\Omega_{x^1}\times \Omega_{x^2}=[-5,5]\times[-5,5]$, and model the propagation of free-streaming radiation through a spatially varying velocity field mimicking a vortex.  
For the setup we follow \cite{laiu_etal_2025}, which is an adaptation of the spherical-polar test in \cite{just_etal_2015} to Cartesian coordinates.  
We set the opacities as $\chi=\sigma=0$, and let the velocity field $\vect{v}=[v^1,v^2,0]^{\intercal}$ be given by
\begin{align}
    v^1(x^1,x^2)
    &=-v_{\mathrm{max}}\,x^2\,\exp\big[\,(1-|\vect{x}|^2)/2\,\big], \\
    v^2(x^1,x^2)
    &=\hspace{8pt}v_{\mathrm{max}}\,x^1\,\exp\big[\,(1-|\vect{x}|^2)/2\,\big],
\end{align}
where $|\vect{x}|=\sqrt{(x^1)^2+(x^2)^2}$.  
The energy domain is given by $\Omega_{\varepsilon}=[0,300]$.
In the energy dimension, elements have geometrically progressing spacing with $\Delta \varepsilon_{i+1} / \Delta \varepsilon_i = 1.119237083677839$.
We discretize the spatial and energy domains with $48\times 48$ and $32$ elements, respectively.  
The moments are initialized as 
\response{
$\cJ=1\times 10^{-8}$, $\cH^1=(1-\delta)W\cJ$, and $\cH^2=0$ for all $(x^1,x^2)\in \Omega_{x^1}\times \Omega_{x^2}$, with $\delta = 10^{-8}$.
}
For the inner spatial boundary in the $x^1$-direction we set 
\response{
$\cJ(\varepsilon,x^1=-5,x^2)=\varepsilon/\exp[(\varepsilon/3-3)+1]$, $\cH^1(\varepsilon,x^1=-5,x^2)=(1-\delta)W(x^1=-5,x^2) \cJ(\varepsilon,x^1=-5,x^2)$, and $\cH^2(\varepsilon,x^1=-5,x^2)=0$.
}
For the outer spatial boundary in the $x^1$-direction we use outflow conditions.  
In the $x^2$-direction we impose reflecting boundary conditions.
In this test, the moments are evolved until a steady state is reached at $t=20$, and we run simulations with $v_{\mathrm{max}}\in\{0.03,0.1,0.3\}$.

In the steady state, as radiation propagates through the vortex, we expect the spectra to be Doppler shifted towards lower (red shift) or higher (blue shift) energies, depending on the local direction and magnitude of the velocity.  
When considering energy integrated quantities, we expect the Lagrangian number and energy densities to be affected by the velocity field due to Doppler shifts, while the Eulerian number and energy densities should be unaffected.  
The purpose of this test is to investigate to what extent these aspects are captured by our proposed scheme.  
Similar to the tests in \ref{sec:streamingDopplerShift} and \ref{sec:transparentShock}, this test is also challenging because the solution evolves close to the boundary of the realizable domain.  

\begin{figure*}
    \centering
    \begin{subfigure}[b]{0.45\textwidth}
        \centering
        \includegraphics[width=1.0\textwidth]{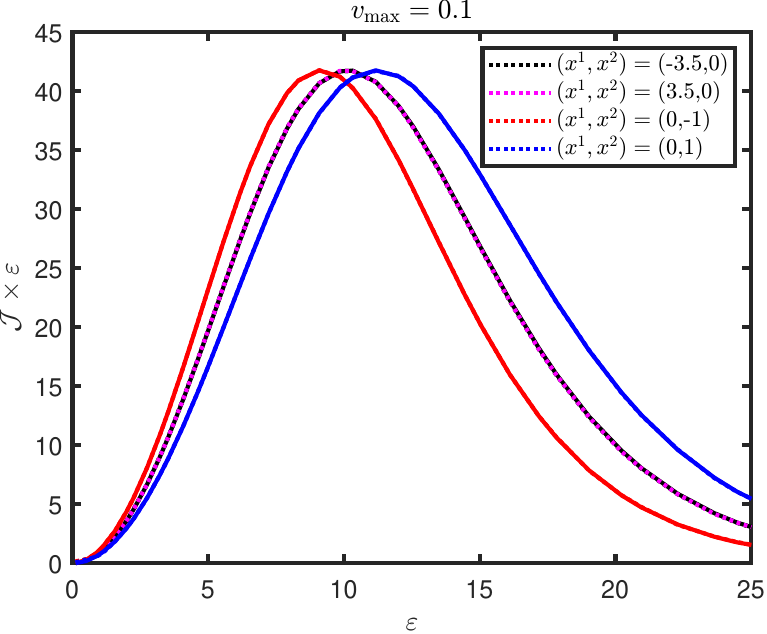}
    \end{subfigure}
    \hfill
    \begin{subfigure}[b]{0.45\textwidth}
        \centering
        \includegraphics[width=1.0\textwidth]{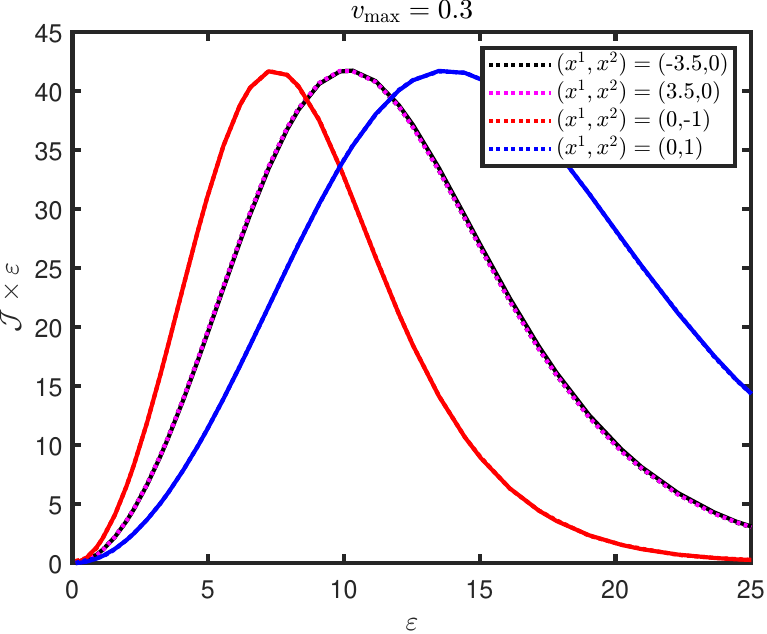}
    \end{subfigure}
    \caption{
    \response{
    Numerical energy spectra at select locations (indicated by the legends) for the transparent vortex problem.  
    The left panel shows results for $v_{\mathrm{max}}=0.1$, while the right panel shows results for $v_{\mathrm{max}}=0.3$.  
    In each panel, the solid lines are the analytic spectra, while dotted lines are the numerical results.
    }
    }
    \label{fig:TV_spectra}
\end{figure*}

Figure \ref{fig:TV_spectra} plots the energy spectra at different locations in the spatial domain for $v_{\mathrm{max}}\in\{0.1,0.3\}$.
For all plots, the black and magenta curves are located at points where the velocity is approximately zero, while the red and blue curves are sampled at locations where $\vect{v}\approx[v_{\mathrm{max}},0,0]^{\intercal}$ and $\vect{v}\approx[-v_{\mathrm{max}},0,0]^{\intercal}$, respectively.  
We expect the black and magenta curves to match the spectra imposed at the inner $x^{1}$ boundary because the velocity is approximately zero at these points.  
\response{
We observe very good agreement between the numerical and analytic spectra for all values of $v_{\mathrm{max}}$.
Even the magenta curve, which is located after the vortex, matches up almost identically with the black curve.
This is a major improvement over the $\cO(v/c)$ results in \cite{laiu_etal_2025}, where they reported trouble recovering the expected spectra $(x^1, x^2) = (3.5,0)$.
As for the red and blue curves, they are red- and blue-shifted relative to the incoming spectra, respectively.  
We again see good agreement for all $v_{\mathrm{max}}$.  
}
When comparing with the results obtained with the $\cO(v/c)$ implementation reported in \cite{laiu_etal_2025}, our results are in better agreement with the analytic solution, even for $v_{\mathrm{max}}=0.3$, and especially in the wake of the vortex.  
We elaborate further on this point below.  

\begin{figure}
    \centering
    \includegraphics[width=0.45\textwidth]{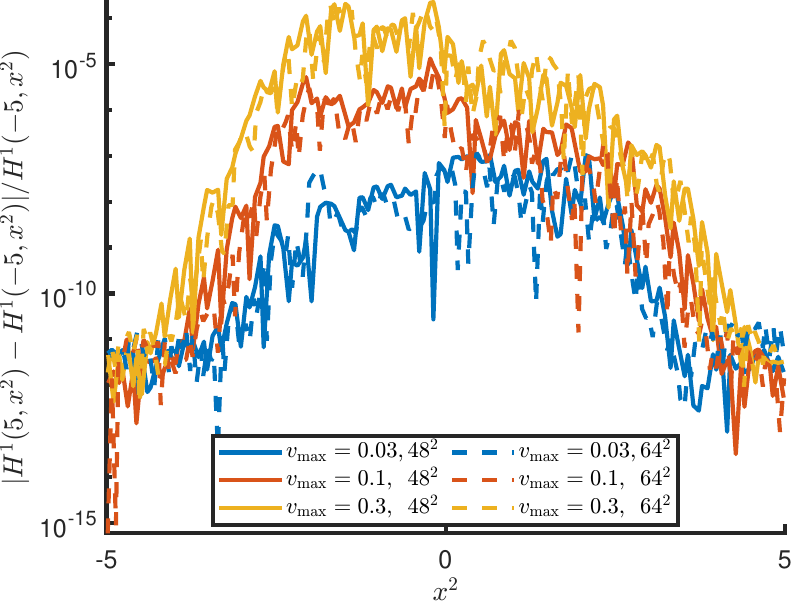}
    \caption{
    \response{
    Relative difference between the incoming and outgoing flux in the $x^1$-direction for $v_{\mathrm{max}}\in\{0.03,0.1,0.3\}$.
    The solid lines are from solutions on a $48^2$ spatial mesh.
    The dashed lines are on a $64^2$ spatial mesh.
    Blue lines are for $v_{\mathrm{max}} = 0.03$, red lines are for $v_{\mathrm{max}} = 0.1$, and yellow lines are for $v_{\mathrm{max}} = 0.3$.
    }
    }
    \label{fig:TV_inout_flux}
\end{figure}

Figure \ref{fig:TV_inout_flux} plots the relative difference between the energy integrated $x^1$-component of the Lagrangian flux densities evaluated at the inner and outer boundary, defined as $|H^1(5,x^2)-H^1(-5,x^2)|/H(-5,x^2)$.  
This quantity should vanish for exact calculations.  
We find that the error increases with increasing $v_{\mathrm{max}}$. 
\response{
However, with the relativistic model, we are able to improve over the $\cO(v/c)$ results reported in \cite{just_etal_2015} and \cite{laiu_etal_2025}, since for $v_{\mathrm{max}}=0.1$ our error is no larger than about $10^{-5}$ --- compared to about $3\times10^{-2}$ in \cite{just_etal_2015} and about $5\times10^{-2}$ in \cite{laiu_etal_2025}.
Even with a more relativistic velocity, $v_{\mathrm{max}} = 0.3$, this error is no larger than $10^{-4}$.
}
This indicates that for this test the inclusion of higher-order velocity observer correction terms is important even for non-relativistic ($v\lesssim0.1$) velocities. 
\response{
Furthermore, Figure \ref{fig:TV_inout_flux} compares this quantity when using $48$ (solid lines) and $64$ (dashed lines) spatial elements in each spatial dimension.  
For the different spatial meshes this quantity is not significantly different.
}

\begin{figure*}
    \centering
    \begin{subfigure}[b]{0.45\textwidth}
        \centering
        \includegraphics[width=1.1\textwidth]{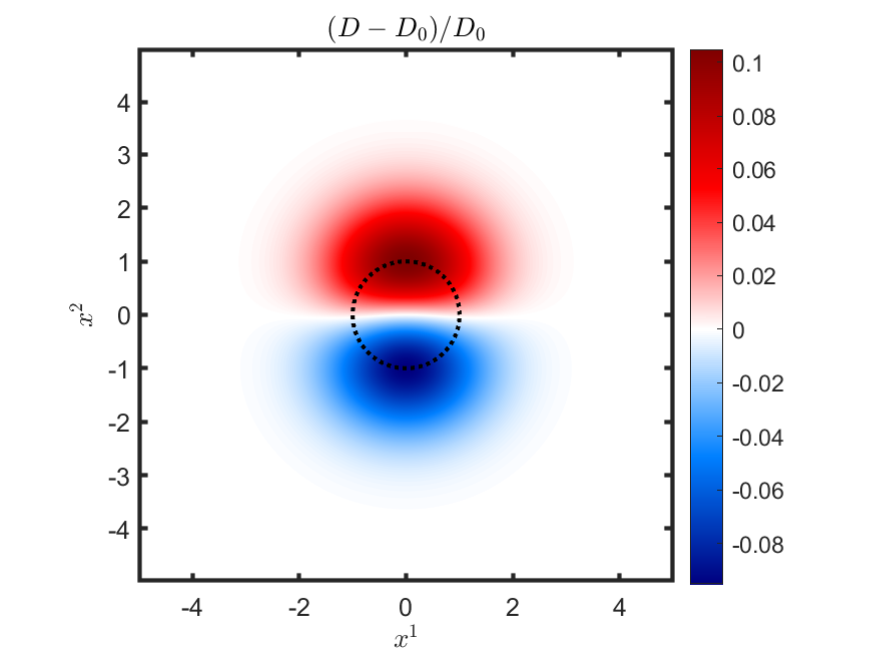}
    \end{subfigure}
    \hfill
    \begin{subfigure}[b]{0.45\textwidth}
        \centering
        \includegraphics[width=1.1\textwidth]{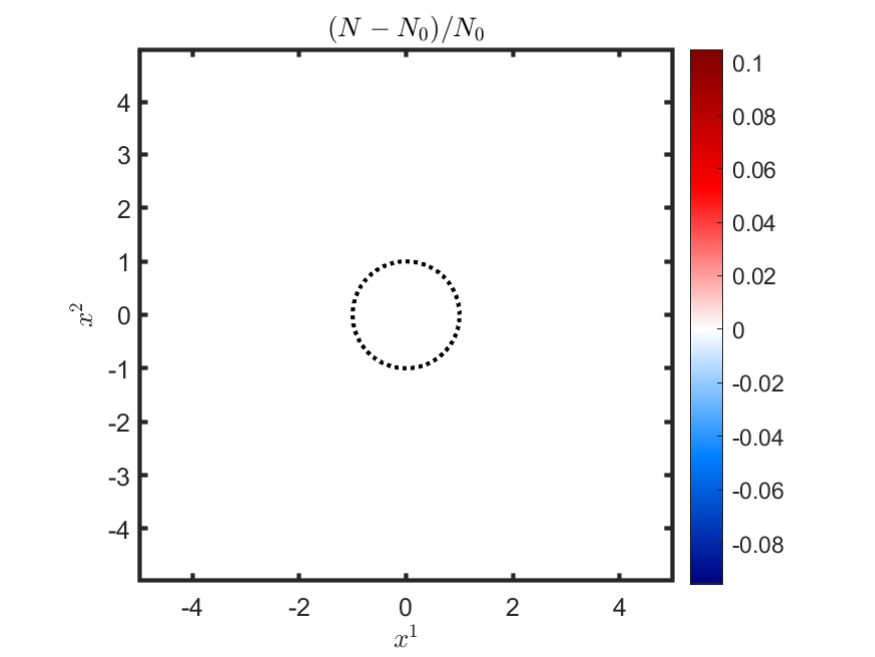}
    \end{subfigure}
    \hfill
    \begin{subfigure}[b]{0.45\textwidth}
        \centering
        \includegraphics[width=1.1\textwidth]{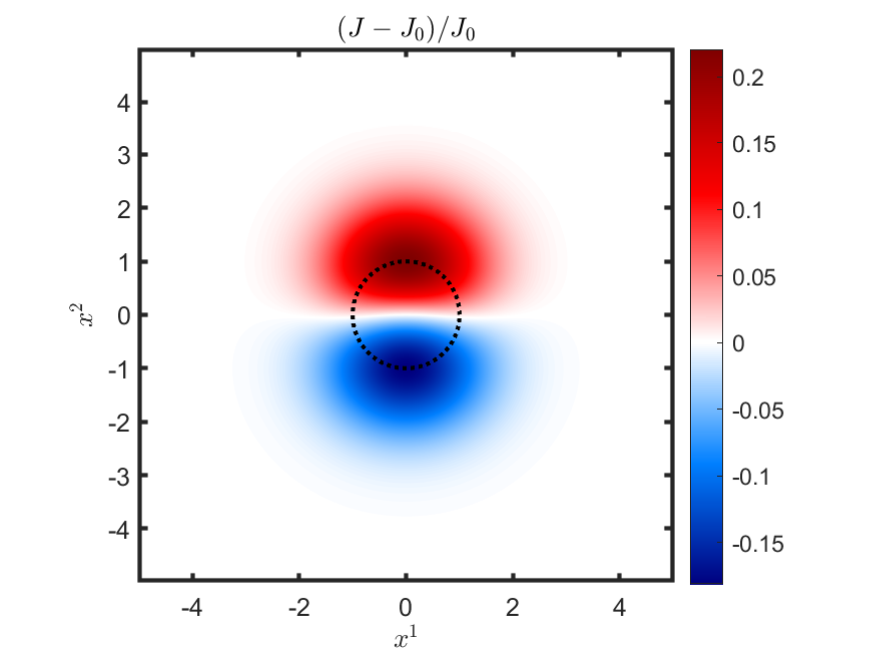}
    \end{subfigure}
    \hfill
    \begin{subfigure}[b]{0.45\textwidth}
        \centering
        \includegraphics[width=1.1\textwidth]{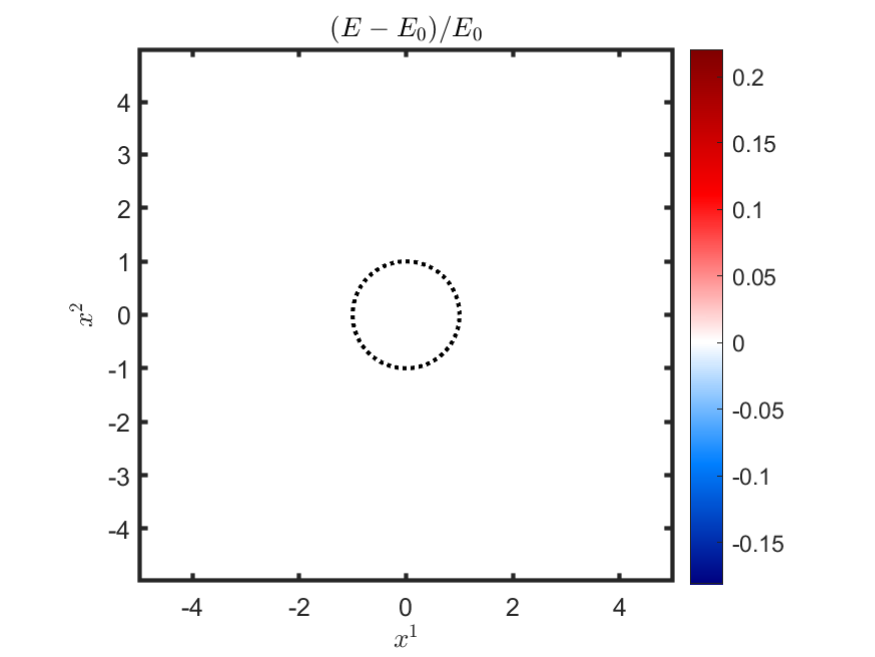}
    \end{subfigure}
    \caption{
    \response{
    Result for the transparent vortex problem with $v_{\mathrm{max}}=0.1$.
    In each panel, the dotted circle represents the contour on which the magnitude of the velocity is $v_{\mathrm{max}}$.
    The upper left panel displays the relative deviation between $D$ and $D_0=D(x^1=-5,x^2)$; i.e., Lagrangian number density.  
    The upper right panel displays the relative deviation between $N$ and $N_0=N(x^1=-5,x^2)$; i.e., Eulerian number density.  
    The lower left panel displays the relative deviation between $J$ and $J_0=J(x^1=-5,x^2)$; i.e., Lagrangian energy density.  
    The lower right panel displays the relative deviation between $E$ and $E_0=E(x^1=-5,x^2)$; i.e., Eulerian energy density.  
    See Eq.~\eqref{eq:grayMoments} for the definition of $D$, $N$, $J$, and $E$.
    }
    }
    \label{fig:TV_contour_v01}
\end{figure*}

\begin{figure*}
    \centering
    \begin{subfigure}[b]{0.45\textwidth}
        \centering
        \includegraphics[width=1.1\textwidth]{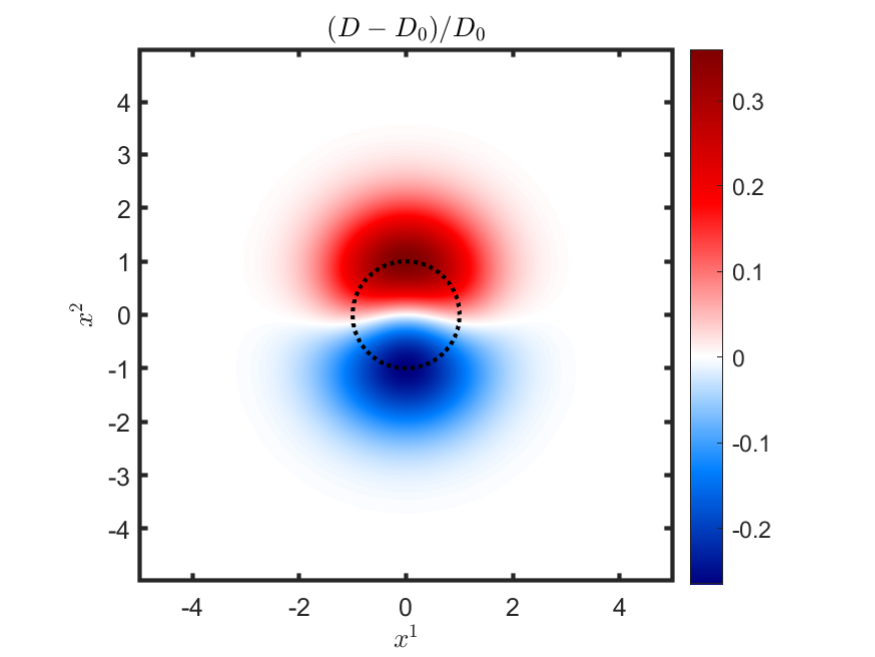}
    \end{subfigure}
    \hfill
    \begin{subfigure}[b]{0.45\textwidth}
        \centering
        \includegraphics[width=1.1\textwidth]{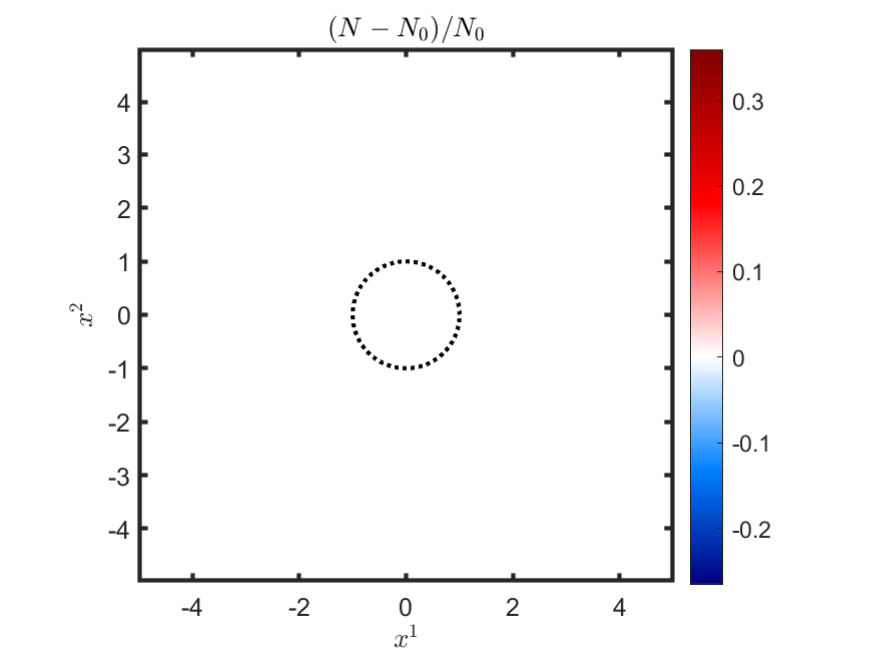}
    \end{subfigure}
    \hfill
    \begin{subfigure}[b]{0.45\textwidth}
        \centering
        \includegraphics[width=1.1\textwidth]{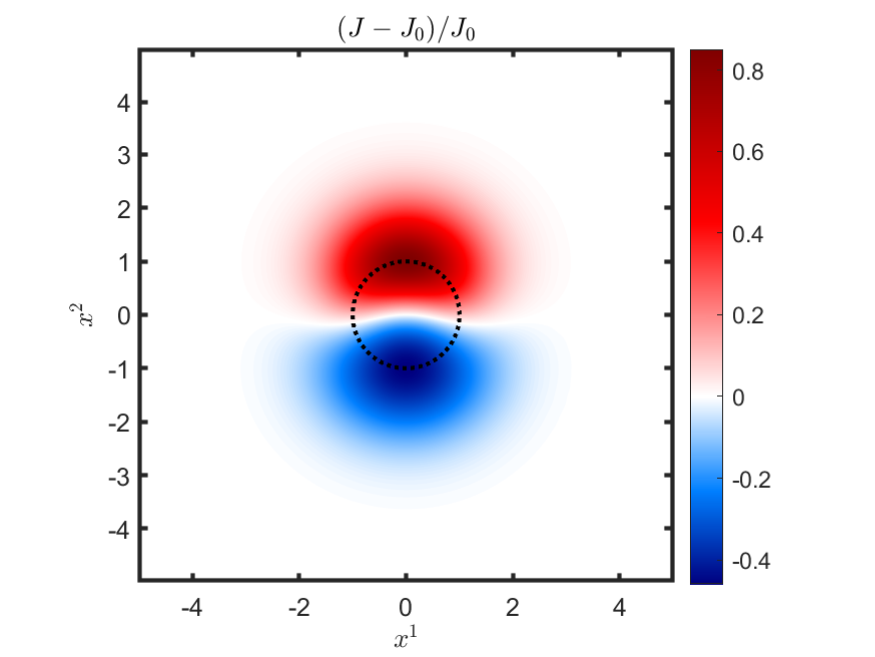}
    \end{subfigure}
    \hfill
    \begin{subfigure}[b]{0.45\textwidth}
        \centering
        \includegraphics[width=1.1\textwidth]{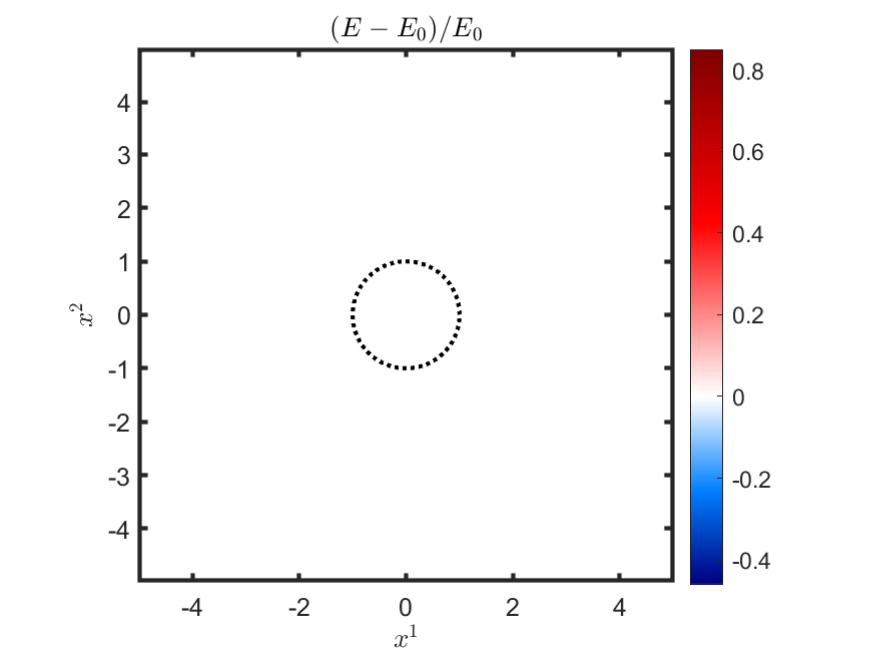}
    \end{subfigure}
    \caption{
    \response{
    Result for the transparent vortex problem.  
    The same quantities as in Figure~\ref{fig:TV_contour_v01} are displayed for $v_{\mathrm{max}}=0.3$, and again the dotted circle is the contour where the magnitude of the velocity is $v_{\mathrm{max}}$.
    }
    }
    \label{fig:TV_contour_v03}
\end{figure*}

In Figures \ref{fig:TV_contour_v01} and \ref{fig:TV_contour_v03} we plot the energy integrated Lagrangian and Eulerian number and energy densities for $v_{\mathrm{max}}=0.1$ and $v_{\mathrm{max}}=0.3$, respectively.  
To emphasize the impact of the Doppler shift, in each panel we plot the relative difference with respect to the value at the inner $x^{1}$ boundary.  
We can see that for all values of $v_{\mathrm{max}}$ the relative difference of the Lagrangian number density $D$ matches with the behavior we see in Figure~\ref{fig:TV_spectra}.  
For $x^2>0$, inside the vortex, we observe an increase corresponding to blue-shifted spectra, and for $x^2<0$, inside the vortex, we observe a decrease corresponding to red-shifted spectra.
In the wake of the vortex we observe that the relative differences do not return exactly to zero, and that the magnitudes of these errors scale with $v_{\mathrm{max}}$.
The plots for $J$ are similar to the plots for $D$, the only difference being the scale on the plots for $J$ is roughly doubled.
\response{
For comparison, the relative changes in the Eulerian number and energy densities, $N$ and $E$, using the same color scales, are presented in the right panels in Figures~\ref{fig:TV_contour_v01} and \ref{fig:TV_contour_v03}.
Both the Eulerian number and energy densities are constant to within $10^{-5}$ and $10^{-4}$ for $v_{\mathrm{max}} = 0.1$ and $0.3$, respectively.
For both $N$ and $E$, these maximal values occur in the wake of the vortex.}
While we observe small variations in the Eulerian number and energy densities, which ideally should be unaffected by the vortex, we emphasize that the variations are considerably smaller than was reported for the $\cO(v/c)$ model considered in \cite{laiu_etal_2025}.
\response{
Furthermore, as in the Transparent Shock test, we observe at $t=20$ that the energy-averaged flux factor, computed either as $\cF/\cE$ or $\cH/\cJ$ are found to be  $1 - \cO(10^{-8})$ across the computational domain.
}

\section{Summary and Conclusions}
\label{sec:summaryConclusions}

We have proposed and analyzed a realizability-preserving numerical method for solving a spectral two-moment model to simulate the transport of massless, neutral particles interacting with a steady background material moving with relativistic velocities.
The model is obtained as the special relativistic limit of the four-momentum-conservative general relativistic model from \cite{cardall_etal_2013a}, which is formulated in terms of comoving-frame momentum-space coordinates to facilitate computational modeling of the material coupling.  
The two-moment model solves for the Eulerian-frame energy and momentum, and is closed by expressing higher-order moments in terms of the evolved moments using the maximum-entropy closure proposed by Minerbo \cite{minerbo_1978}.
The realizability-preserving method maintains algebraic bounds on the Eulerian-frame energy and momentum.

The proposed numerical method combines DG phase-space discretization with IMEX time-stepping.
The phase-space advection terms are treated explicitly, while the collision term is treated implicitly.  
To preserve realizability of the evolved moments, the DG phase-space discretization is equipped with numerical fluxes, which allow us to derive a sufficient time-step restriction to preserve realizability of the cell-averaged moments.
For the IMEX scheme, the time-step restriction is no worse than the restriction imposed by the explicit Forward Euler scheme.  
A realizability-enforcing limiter is used to recover pointwise realizability from the cell-averaged moments.  

Because the closure procedure and the implicit collision solve are formulated in terms of primitive moments, the proposed method requires solving two respective nonlinear systems.  
For both of these nonlinear systems, solution methods have been proposed in the form of fixed-point problems.  
The fixed-point problems are formulated such that they preserve moment realizability of each iterate, subject to a mild step-size constraint.
For the primitive moment recovery problem we have proven convergence to a unique realizable moment pair for the algebraic approximation to the Minerbo closure for $v< 0.221075$.  
Although this (sufficient) convergence condition appears restrictive, it is likely not sharp since, numerically, we have not observed any convergence failures during primitive moment recovery for any $0\leq v < 1$.  
Additionally, in practice we have not observed any issues regarding convergence or maintaining realizability when using Newton's method for primitive moment recovery.  

The proposed algorithm has been implemented and tested against six benchmark problems.
We run two tests --- Streaming Sine Wave and Gaussian Diffusion $1$D --- with constant background velocities.
The results of these tests demonstrate the proposed method's order of accuracy for smooth solutions and ability to perform well in the dynamic diffusion regime (e.g., \cite{mihalasMihalas_1999}), respectively.  
Three tests --- Streaming Doppler Shift, Transparent Shock, and Transparent Vortex --- are run with spatially varying (smooth and discontinuous) background velocities.
The Streaming Doppler Shift and Transparent Vortex tests demonstrate the proposed method's ability to capture special relativistic effects, such as Doppler shifted energy spectra, in one and two spatial dimensions, respectively.  
The results of these tests also demonstrate improvements over the same tests for the $\cO(v/c)$ model in \cite{laiu_etal_2025}.  
The Transparent Shock test demonstrates the proposed method's ability to capture solutions which develop discontinuities due to steep velocity gradients, as we see good agreement with reference solutions, even for large, $v\leq 0.5$, velocities.
The Shadow Casting test, which has zero velocity, demonstrates the proposed method's ability to accurately capture radiation shadows.

Further work is needed to address challenges associated with extensions of the model considered here.  
First, the collision operator considered in this paper is simplified and the background is fixed.
Extending the method to collision operators that include more complex particle interactions, such as inelastic scattering, and considering the dynamical response of the material should be done.  
\response{To this end, we would consider extending the nested fixed-point algorithm proposed in \cite{laiu_etal_2025} to the relativistic case.  
In this algorithm, the matter states are updated in an outer iteration layer, while the radiation moments are updated in an inner iteration layer with the matter states held fixed.  
In this context, the iterative collision solver proposed in Section~\ref{sec:collision_solver} can be used straightforwardly in the inner layer.}
Second, this paper addresses neutral particles obeying Maxwell--Boltzmann statistics, but we aim to apply the proposed method to model neutrino transport in nuclear astrophysics applications.
Neutrinos are fermions, which due to Pauli's exclusion principle have an upper bound on the phase-space density and associated bounds on the moments.  
Therefore, our analysis should be extended to moment closures based on Fermi--Dirac statistics.
Third, the model is special relativistic and formulated in Cartesian coordinates, so it is desirable to extend to general relativity and more general curvilinear spatial coordinates.  
Finally, our method is not designed to conserve particles even though our model possesses this conservation law at the continuum level.  
To model neutrinos, it is desirable to simultaneously conserve both number and energy (e.g., \cite{mezzacappa_etal_2020}).  
We believe the analysis and method proposed here can be helpful as a starting point to address these challenges in future work.

\begin{acknowledgments}
    This research was supported in part by the Exascale Computing Project (17-SC-20-SC), a collaborative effort of the U.S. Department of Energy (DOE) Office of Science and the National Nuclear Security Administration.
    E.E. acknowledges support from the National Science Foundation (NSF) Computational Mathematics program under grant 2309591 and the Gravitational Physics Program under grant 2110177.
    J.H. and Y.X. acknowledge support from the NSF grants DMS-1753581 and DMS-2309590.
    This research was supported in part by an appointment with the NSF Mathematical Sciences Graduate Internship (MSGI) Program.
    This program is administered by the Oak Ridge Institute for Science and Education (ORISE) through an interagency agreement between the U.S. DOE and NSF.
    ORISE is managed for DOE by ORAU.
    All opinions expressed in this paper are the author's and do not necessarily reflect the policies and views of NSF,  ORAU/ORISE, or DOE.
\end{acknowledgments}

\appendix

\section{Sharpness of $a^\varepsilon$ with respect to $|q|$}
\label{appendix:q_bound}

\begin{figure}
    \centering
    \includegraphics[width=8.6cm]{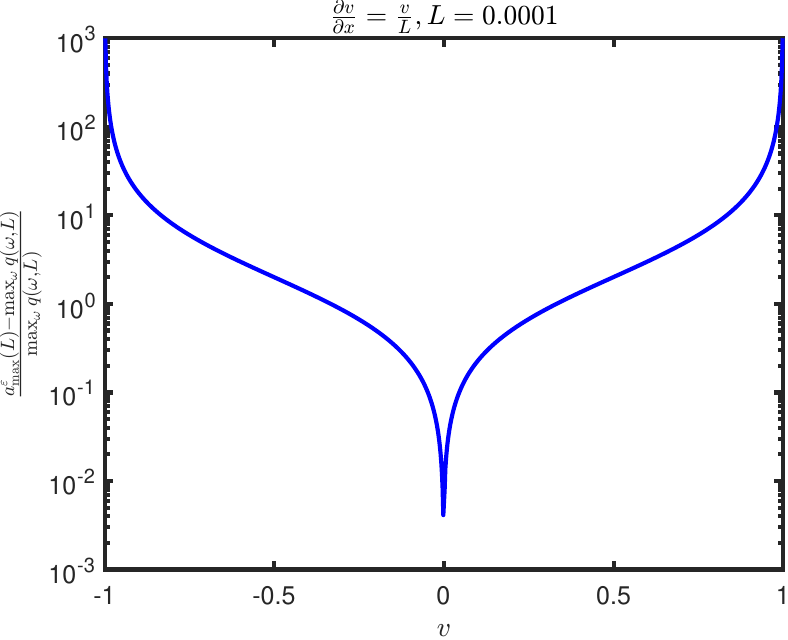}
    \caption{The relative difference in Eq.~\eqref{eq:boundsharpness} is plotted against $-1< v <1$, where we set $\p_x v^i = v^i / L$, with $L = 0.0001$.}
    \label{fig:q_bound}
\end{figure}
The upper bound for $|q|$ given by Eq.~\eqref{eq:q_upperbound} is not a sharp bound overall.
For small values of $v$, the bound is sharp, but it becomes looser as $v$ increases (see Figure \ref{fig:q_bound}).
To test these bounds we consider a one-dimensional three-velocity which has derivatives only in the $x$-direction.  We let $v^i = (v, 0, 0)^{\intercal}$, where $-1 < v < 1,$ and set ${\p_x v^i = v^i/L}$.
Here $L$ is a length scale parameter, and in the rest of this appendix we will denote $q$ by $q(\omega;L)$ to emphasize its dependence on the angle $\omega$ and parameter $L$.
We then generate $100$ different values of $q(\omega;L)$, where the parameter $\omega = (\vartheta,\varphi)$ is randomly sampled from $\vartheta\in(0,2\pi)$ and $\varphi\in(0,\pi)$ to produce $\ell^\mu$ and $L^\mu$ as given by Eqs.~\eqref{eq:tetradUandL} and \eqref{eq:EulerianUnitDirection}, respectively.
In Figure \ref{fig:q_bound} we plot the relative difference
\begin{equation}
\label{eq:boundsharpness}
    \frac{a^\varepsilon_{\mathrm{max}}(L)-\max_{\omega}q(\omega;L)}{\max_{\omega}q(\omega;L)}
\end{equation}
where $a^\varepsilon_{\mathrm{max}}$ is defined in Eq.~\eqref{eq:q_upperbound}.
If this relative difference is close to zero, then the bound is sharp.
We ran the test with $L\in \{10^{-4}, 10^{-3},...,10^3, 10^4\}$.
The results of the test in Figure \ref{fig:q_bound} are the same for any $L$.  Thus, the bound is sharp only for small $v$, and becomes looser as $v$ increases.
For $-0.004 < v < 0.004$ the relative difference is less than $10^{-2}$;
for $-0.048 < v < 0.048$, less than $10^{-1}$;
and for $-0.334 < v < 0.334$, less than $10^{0}$.

Since our bound is not sharp, it could be worthwhile to improve this upper bound in the hopes of improving test results, especially for larger velocities.
Improving the bound would require either a more careful bounding of the quantities $L^i B_i$, $L^i L^j C_{ij}$, and $E^2/\varepsilon^2$, or attempting an upper bound through a \emph{Lagrangian} decomposition of the tensor $A_{\mu\nu}$. A Lagrangian decomposition proves more difficult to bound since ${\ell^0 \neq 0}$, whereas with the Eulerian decomposition, $L^0 = 0$.

\section{Moment Conversion Convergence Proof}
\label{appendix:convergence}

We derive a restriction on $v=\sqrt{v_\mu v^\mu}$ to guarantee the moment conversion scheme, Eq.~\eqref{eq:Richardson_update}, converges to a unique solution.
We do this by showing for $\bM\in\cR$ the mapping in Eq.~\eqref{eq:Richardson_update},
\begin{equation}
    \bH_{\widehat{\bU}}(\bM)
    =
    \begin{bmatrix}
        (1-\lambda W)\cJ-\lambda v^i\cH_i+\lambda\widehat{\cE}\\
        (1-\lambda W)\cH_i -\lambda v^j\cK_{ij} +\lambda\widehat{\cF}_i
    \end{bmatrix},
\end{equation}
is a contraction mapping on $\cR$ equipped with the $2$-norm.
That is, for all $\bM^{(a)},\bM^{(b)}\in\cR$, we have
\begin{equation}
\label{eq:contraction_mapping}
    \|\bH_{\widehat{\bU}}(\bM^{(a)})-\bH_{\widehat{\bU}}(\bM^{(b)})\|_2\leq L\|\bM^{(a)}-\bM^{(b)}\|_2
\end{equation}
for some $0\leq L<1$.
We will follow the framework used in the appendix of \cite{laiu_etal_2025} where the authors proved a convergence condition for a similar moment conversion problem.

The following inequalities will be necessary for proving our restriction on $v$
\begin{equation}
    \cH_i\cH^i\leq W^2\cH^2,
    \quad
    (\cH^0)^2\leq W^2\cH^2v^2.
    \label{eq:cH_bounds}
\end{equation}
They can be derived from definition of $\cH^2$ by using the identities ${\cH^0 = v^i\cH_i}$, ${(1-v^2) = W^{-2}}$, and ${W^2 - 1 = (Wv)^2}$.
We then define
\begin{align}
    \Delta\cJ &= \cJ^{(a)}-\cJ^{(b)}, \\
    \Delta\cH_i &= \cH^{(a)}_i - \cH^{(b)}_i, \\
    \Delta\cK_{ij} &= \cK^{(a)}_{ij} - \cK^{(b)}_{ij}.
\end{align}
Then the norm on the right-hand side of Eq.~\eqref{eq:contraction_mapping} is given by
 \begin{equation}
     \|\bM^{(a)}-\bM^{(b)}\|_2
     =
     \begin{Vmatrix}
         \Delta\cJ\\
         \Delta\cH_i
     \end{Vmatrix}_2
     =
     \sqrt{\|\Delta\cJ\|^2_2+\|\Delta\cH_i\|^2_2},
 \end{equation}
 where $\|\Delta\cJ\|^2_2=|\Delta\cJ|^2$, and $\|\Delta\cH_i\|^2_2=|\Delta\cH_1|^2+|\Delta\cH_2|^2+|\Delta\cH_3|^2$.
Applying the triangle inequality to the left-hand side of Eq.~\eqref{eq:contraction_mapping}, we have the following upper bound
\begin{align}
     &\|\bH_{\widehat{\bU}}(\bM^{(a)})-\bH_{\widehat{\bU}}(\bM^{(b)})\|_2
     =
     \begin{Vmatrix}
         (1-\lambda W)\Delta \cJ-\lambda v^i\Delta\cH_i\\
         (1-\lambda W)\Delta\cH_i-\lambda v^j\Delta\cK_{ij}
     \end{Vmatrix}_2 \notag \\
     &\qquad 
     \leq
     (1-\lambda W)
     \begin{Vmatrix}
         \Delta \cJ\\
         \Delta\cH_i
     \end{Vmatrix}_2
     +\lambda
     \begin{Vmatrix}
         v^i\Delta\cH_i\\
         v^j\Delta\cK_{ij}
     \end{Vmatrix}_2.
\end{align}
Thus it suffices to show for some $0\leq \widetilde{L} <W$
\begin{equation}
    \begin{Vmatrix}
         v^i\Delta\cH_i\\
         v^j\Delta\cK_{ij}
    \end{Vmatrix}_2
    \leq
    \widetilde{L}
    \begin{Vmatrix}
         \Delta\cJ\\
         \Delta\cH_i
    \end{Vmatrix}_2.
    \label{eq:appendix_bound_target}
\end{equation}
We will prove the above inequality holds for ${\widetilde{L} = v\sqrt{14W^6+\sqrt{14}W^5+1}}$ with the restriction that $v<0.221075$.

We first bound the left-hand side of Eq.~\eqref{eq:appendix_bound_target} using Cauchy-Schwarz,
    \begin{equation}
        \begin{Vmatrix}
         v^i\Delta\cH_i\\
         v^j\Delta\cK_{ij}
         \end{Vmatrix}_2^2
         \leq
         v^2\|\Delta \cH_j\|_2^2+\|v^j \Delta \cK_{ij}\|_2^2,
    \end{equation}
    and by the mean-value theorem for vector-valued functions
    \begin{equation}
        \|v^j \Delta \cK_{ij}\|_2
        \leq
        \|v^j\frac{\p\cK_{ij}}{\p \cJ}\|_2\,\|\Delta \cJ\|_2
        +
        \|v^j\frac{\p\cK_{ij}}{\p \cH_k}\|_2\,\|\Delta \cH_j\|_2.
        \label{eq:norm_vjDeltaKij}
    \end{equation}
Thus to prove a convergence criterion, we seek bounds on $\|v^j\frac{\p\cK_{ij}}{\p \cJ}\|_2$ and $\|v^j\frac{\p\cK_{ij}}{\p \cH_k}\|_2$.

\begin{lemma}
    For any $\bM\in\cR$, $\|v^j\frac{\p\cK_{ij}}{\p \cJ}\|_2\leq W^2v$.
    \label{lem:appendix_lem1}
\end{lemma}
\begin{proof}
    \begin{align}
        v^j\frac{\p\cK_{ij}}{\p \cJ}
        &=
        v^j\bigg(\frac{1}{2}\big[(1-\mathsf{k})(\eta_{ij}+u_iu_j)+(3\mathsf{k}-1)\hat{\mathsf{h}}_i\hat{\mathsf{h}}_j\big] \nonumber \\
        &\hspace{12pt}
        +\frac{\cJ}{2}\big[(3\hat{\mathsf{h}}_i\hat{\mathsf{h}}_j-\eta_{ij}-u_iu_j)\frac{\p\mathsf{k}}{\p\mathsf{h}}\frac{\p\mathsf{h}}{\p\cJ}\big]\bigg) \nonumber \\
        &=
        \frac{1}{2}[W^2v_i(1-\mathsf{k}+\mathsf{k}'\mathsf{h})+\hat{\mathsf{h}}^0\hat{\mathsf{h}}_i(3\mathsf{k}-1-3\mathsf{k}'\mathsf{h})] \nonumber \\
        &=
        \frac{1}{2}(W^2v_i\psi_0+\hat{\mathsf{h}}^0\hat{\mathsf{h}}_i\psi_1),
        \label{eq:appendix_vjKij_pJ}
    \end{align}
    where $\mathsf{k}'=\p\mathsf{k}/\p\mathsf{h}$, and we have defined
    \begin{align}
    \psi_0 &= 1-\mathsf{k}+\mathsf{k}'\mathsf{h}, \\
    \psi_1 &= 3\mathsf{k}-1-3\mathsf{k}'\mathsf{h}.
    \end{align}
    Recall $\mathsf{k}$ is the Eddington factor given by Eq.~\eqref{eq:eddingtonFactorAlgebraic}, and $\mathsf{h} = \cH / \cJ \in [0,1]$ is the flux factor.
    The functions $\psi_0$ and $\psi_1$ satisfy the following inequalities for $\mathsf{h}\in[0,1]$
    \begin{equation*}
        \frac{2}{3}\leq \psi_0 \leq 2, \quad -4 \leq \psi_1 \leq 0.
    \end{equation*}
    These bounds can be verified by plotting $\psi_0$ and $\psi_1$ as functions of $\mathsf{h}$.
    Taking the norm squared of Eq.~\eqref{eq:appendix_vjKij_pJ} and bounding $\hat{\mathsf{h}}_i\hat{\mathsf{h}}^i$ from above yields
    \begin{align}
        \|v^j\frac{\p\cK_{ij}}{\p \cJ}\|_2^2
        &=
        \frac{1}{4}(W^4v^2\psi_0^2+(\hat{\mathsf{h}}^0)^2\hat{\mathsf{h}}_i\hat{\mathsf{h}}^i\psi_1^2+2W^2(\hat{\mathsf{h}}^0)^2\psi_0\psi_1) \nonumber \\
        &\leq
        \frac{1}{4}(W^4v^2\psi_0^2+W^2(\hat{\mathsf{h}}^0)^2(\psi_1^2+2\psi_0\psi_1)).
    \end{align}
    It can be shown that $\psi_1^2+2\psi_0\psi_1\leq 0$ for all $\mathsf{h}\in[0,1]$.
    Then
    \begin{equation*}
        \|v^j\frac{\p\cK_{ij}}{\p \cJ}\|_2^2
        \leq
        \frac{1}{4}W^4v^2\psi_0^2
        \leq W^4v^2.
    \end{equation*}
\end{proof}
Note Lemma \ref{lem:appendix_lem1} is comparable to Lemma 8 in \cite{laiu_etal_2025}.

\begin{lemma}
    For any $\bM\in\cR$, $\|v^j\frac{\p\cK_{ij}}{\p \cH_k}\|_2 \leq \sqrt{14}W^3v$.
    \label{lem:appendix_lem2}
\end{lemma}
\begin{proof}
    \begin{align*}
        v^j\frac{\p\cK_{ij}}{\p\cH_k}
        =&
        v^j[(3\hat{\mathsf{h}}_i\hat{\mathsf{h}}_j-h_{ij})\frac{\p\mathsf{k}}{\p \mathsf{h}}\frac{\p \mathsf{h}}{\p\cH_k} \\
        &\hspace{12pt}
        +(3\mathsf{k}-1)(\hat{\mathsf{h}}_i\frac{\p\hat{\mathsf{h}}_j}{\cH_k}+\hat{\mathsf{h}}_j\frac{\p\hat{\mathsf{h}}_i}{\cH_k})]\frac{\cJ}{2} \\
        =&
        \frac{1}{2}\mathsf{k}'(3\hat{\mathsf{h}}_i\hat{\mathsf{h}}^0-W^2v_i)(\hat{\mathsf{h}}_k+\hat{\mathsf{h}}_0v_k) \\
        &\hspace{12pt}
        +
        \frac{3\mathsf{k}-1}{2\mathsf{h}}(\hat{\mathsf{h}}_iv_k + \hat{\mathsf{h}}^0\delta_{ik} -2\hat{\mathsf{h}}_i\hat{\mathsf{h}}^0(\hat{\mathsf{h}}_k+\hat{\mathsf{h}}_0v_k))
    \end{align*}
    To prove $\|v^j\frac{\p\cK_{ij}}{\p \cH_k}\|_2 \leq \sqrt{14}W^3v$, we will show that
    \begin{equation*}
        \|v^jy^k\frac{\p\cK_{ij}}{\p\cH_k}\|_2 \leq \sqrt{14}W^3vy,
        \quad
        \forall\,y^i=(y^1,y^2,y^3)^{\intercal},
    \end{equation*}
    where $y=\sqrt{y_iy^i}$.
    We define the following quantities
    \begin{align}
        \psi_2 &= \frac{3\mathsf{k}-1}{h},
        \quad
        \beta = v_ky^k,
        \quad
        \gamma = \hat{\mathsf{h}}_ky^k, \\
        \alpha_k &= \hat{\mathsf{h}}_k+\hat{\mathsf{h}}_0v_k=\hat{\mathsf{h}}_k-\hat{\mathsf{h}}^0v_k.
    \end{align}
    Then,
    \begin{align}
        v^jy^k\frac{\p\cK_{ij}}{\p\cH_k}
        &=
        \frac{1}{2}\mathsf{k}'(3\hat{\mathsf{h}}_i\hat{\mathsf{h}}^0-W^2v_i)\alpha_ky^k \nonumber \\
        &\hspace{12pt}
        +
        \frac{1}{2}\psi_2(\hat{\mathsf{h}}_i\beta + \hat{\mathsf{h}}^0y_i -2\hat{\mathsf{h}}_i\hat{\mathsf{h}}^0\alpha_ky^k).
        \label{eq:appendix_vjykKij_pH}
    \end{align}
    Taking the squared norm of Eq.~\eqref{eq:appendix_vjykKij_pH}, it then follows that
    \begin{widetext}
    \begin{align*}
        \|v^jy^k\frac{\p\cK_{ij}}{\p\cH_k}\|^2_2
        =&
        \frac{1}{4}(\mathsf{k}')^2(\alpha_ky^k)^2(9\hat{\mathsf{h}}_i\hat{\mathsf{h}}^i(\hat{\mathsf{h}}^0)^2-6W^2(\hat{\mathsf{h}}^0)^2+W^4v^2)\\
        &+
        \frac{1}{2}\mathsf{k}'\psi_2(\alpha_ky^k)
        (3\hat{\mathsf{h}}_i\hat{\mathsf{h}}^i\hat{\mathsf{h}}^0\beta
        +3(\hat{\mathsf{h}}^0)^2\gamma
        -6\hat{\mathsf{h}}_i\hat{\mathsf{h}}^i(\hat{\mathsf{h}}^0)^2\alpha_ky^k -2W^2\hat{\mathsf{h}}^0\beta
        +2W^2(\hat{\mathsf{h}}^0)^2(\alpha_ky^k))\\
        &+
        \frac{1}{4}\psi_2^2
        (\hat{\mathsf{h}}_i\hat{\mathsf{h}}^i\beta^2
        +2\hat{\mathsf{h}}^0\beta\gamma
        -4\hat{\mathsf{h}}_i\hat{\mathsf{h}}^i\hat{\mathsf{h}}^0(\alpha_ky^k)\beta +(\hat{\mathsf{h}}^0)^2y^2-4(\hat{\mathsf{h}}^0)^2(\alpha_ky^k)\gamma
        +4\hat{\mathsf{h}}_i\hat{\mathsf{h}}^i(\hat{\mathsf{h}}^0)^2(\alpha_ky^k)^2) \\
        =&
        \gamma^2(\hat{\mathsf{h}}^0)^2[\hat{\mathsf{h}}_i\hat{\mathsf{h}}^i(\frac{9}{4}(\mathsf{k}')^2-3\mathsf{k}'\psi_2+\psi_2^2)+W^2(-\frac{3}{2}(\mathsf{k}')^2+\mathsf{k}'\psi_2)+\frac{3}{2}\mathsf{k}'\psi_2-\psi_2^2]
        +
        \frac{1}{4}\beta^2\hat{\mathsf{h}}_i\hat{\mathsf{h}}^i\psi_2^2\\
        &-2
        \hat{\mathsf{h}}^0\beta\gamma[-\frac{1}{2}\hat{\mathsf{h}}_i\hat{\mathsf{h}}^i(\frac{3}{2}\mathsf{k}'\psi_2-\psi_2^2)+\frac{1}{2}W^2\mathsf{k}'\psi_2-\frac{1}{4}\psi_2^2]
        +
        \frac{1}{4}W^4v^2\gamma^2(\mathsf{k}')^2
        +
        \frac{1}{4}(\hat{\mathsf{h}}^0)^2y^2\psi_2^2\\
        &+
        (\hat{\mathsf{h}}^0)^2\beta^2[\hat{\mathsf{h}}_i\hat{\mathsf{h}}^i(\hat{\mathsf{h}}^0)^2(\frac{9}{4}(\mathsf{k}')^2-3\mathsf{k}'\psi_2+\psi_2^2)+(\hat{\mathsf{h}}^0)^2W^2(-\frac{3}{2}(\mathsf{k}')^2+\mathsf{k}'\psi_2)\\
        &\hspace{5em}+\frac{1}{4}W^4v^2(\mathsf{k}')^2-\hat{\mathsf{h}}_i\hat{\mathsf{h}}^i(\frac{3}{2}\mathsf{k}'\psi_2-\psi_2^2)+W^2\mathsf{k}'\psi_2]\\
        &-2
        \hat{\mathsf{h}}^0\beta\gamma[(\hat{\mathsf{h}}^0)^2\hat{\mathsf{h}}_i\hat{\mathsf{h}}^i(\frac{9}{4}(\mathsf{k}')^2-3\mathsf{k}'\psi_2+\psi_2^2)+(\hat{\mathsf{h}}^0)^2W^2(-\frac{3}{2}(\mathsf{k}')^2+\mathsf{k}'\psi_2)+\frac{1}{2}(\hat{\mathsf{h}}^0)^2(\frac{3}{2}\mathsf{k}'\psi_2-\psi_2^2)]\\
        &-2\hat{\mathsf{h}}^0\beta\gamma[\frac{1}{4}W^4v^2(\mathsf{k}')^2].
    \end{align*}
    \end{widetext}
    To save space, we introduce the following quantities
    \begin{align*}
        A&=\frac{9}{4}(\mathsf{k}')^2-3\mathsf{k}'\psi_2+\psi_2^2 \geq 0,
        \\
        B&=-\frac{3}{2}(\mathsf{k}')^2+\mathsf{k}'\psi_2 \leq 0,
        \\
        C&=\frac{3}{2}\mathsf{k}'\psi_2-\psi_2^2 \geq 0.
    \end{align*}
    The inequality for $A$ holds for all $\mathsf{h}\in[0,1]$, while the inequalities for $B$ and $C$ hold for $\mathsf{h} \geq 1/6$.
    The following bounds can be shown by expressing the quantities as polynomials of the flux factor $\mathsf{h}$ and using the inequalities in Eq.~\eqref{eq:cH_bounds},
    \begin{alignat}{2}
        -\frac{1}{2}\hat{\mathsf{h}}_i\hat{\mathsf{h}}^i C+\frac{1}{2}W^2\mathsf{k}'\psi_2-\frac{1}{4}\psi_2^2 &\geq 0, \,\, \text{for} \,\, &&\mathsf{h}\in[0,1],
        \label{eq:app_bound1}\\
        \hat{\mathsf{h}}_i\hat{\mathsf{h}}^i A+W^2 B+\frac{1}{2}C &\geq 0, \,\, \text{if} \,\, &&\mathsf{h} \leq 1/6, 
        \label{eq:app_bound2}\\
        \hat{\mathsf{h}}_i\hat{\mathsf{h}}^i A+W^2 B+\frac{1}{2}C &\leq 0, \,\, \text{if} \,\, &&\mathsf{h} \geq 1/6, 
        \label{eq:app_bound3}
    \end{alignat}
    We will ignore the case of $\mathsf{h}\leq 1/6$ as it results in the smaller upper bound. 
    Thus, for the remainder of the proof we assume $\mathsf{h} \geq 1/6$.
    We apply the inequality $-2ab\leq a^2 + b^2$ to the three $-2\hat{\mathsf{h}}^0\beta\gamma$ terms, where for the first term we take $a=\gamma\hat{\mathsf{h}}^0$ and $b=\beta$, while for the remaining terms we take $a=\hat{\mathsf{h}}^0\beta$ and $b=\gamma$.
    Note that we can apply this inequality since Eqs.~\eqref{eq:app_bound1},~\eqref{eq:app_bound3} show the coefficients for each of the $-2\hat{\mathsf{h}}^0\beta\gamma$ terms do not change sign for $\mathsf{h}\geq 1/6$.
    Reexpressing the norm in terms of $A,B,$ and $C$, and applying the $-2ab$ inequality, we have
    \begin{widetext}
    \begin{align*}
        \|v^jy^k\frac{\p\cK_{ij}}{\p\cH_k}\|^2_2
        \leq&
        \gamma^2(\hat{\mathsf{h}}^0)^2[
        \hat{\mathsf{h}}_i\hat{\mathsf{h}}^i A -\frac{1}{2}\hat{\mathsf{h}}_i\hat{\mathsf{h}}^i C + W^2 B +\frac{1}{2}W^2\mathsf{k}'\psi_2 + C -\frac{1}{4}\psi_2^2
        ]
        +
        \beta^2[\hat{\mathsf{h}}_i\hat{\mathsf{h}}^i\frac{1}{4}\psi_2^2 -\frac{1}{2}\hat{\mathsf{h}}_i\hat{\mathsf{h}}^i C +\frac{1}{2}W^2\mathsf{k}'\psi_2 -\frac{1}{4}\psi_2^2]\\
        &+
        (\hat{\mathsf{h}}^0\beta)^2[
        -\frac{1}{2}(\hat{\mathsf{h}}^0)^2 C + \frac{1}{2}W^4v^2(\mathsf{k}')^2-\hat{\mathsf{h}}_i\hat{\mathsf{h}}^i C + W^2\mathsf{k}'\psi_2 
        ]
        +\frac{1}{4}W^4v^2\gamma^2(\mathsf{k}')^2
        +\frac{1}{4}(\hat{\mathsf{h}}^0)^2y^2\psi_2^2\\
        &+
        \gamma^2[
        -\hat{\mathsf{h}}_i\hat{\mathsf{h}}^i(\hat{\mathsf{h}}^0)^2 A - W^2(\hat{\mathsf{h}}^0)^2 B -\frac{1}{2}(\hat{\mathsf{h}}^0)^2 C + \frac{1}{4}W^4v^2(\mathsf{k}')^2
        ].
    \end{align*}
    \end{widetext}
    We now focus on upper bounding the terms in brackets.  The remaining terms are easily upper bounded by Eq.~\eqref{eq:cH_bounds} and the inequality ${0 \leq \mathsf{k}' \leq 2}$.
    
    \begin{align*}
        &\hat{\mathsf{h}}_i\hat{\mathsf{h}}^i A -\frac{1}{2}\hat{\mathsf{h}}_i\hat{\mathsf{h}}^i C + W^2 B +\frac{1}{2}W^2\mathsf{k}'\psi_2 + C -\frac{1}{4}\psi_2^2  \\
        &\qquad \leq
        W^2(A+B+\frac{1}{2}\mathsf{k}'\psi_2)
        + (\frac{1}{2}C - \frac{1}{4}\psi_2^2)
        \leq W^2, \\
        &\hat{\mathsf{h}}_i\hat{\mathsf{h}}^i\frac{1}{4}\psi_2^2 -\frac{1}{2}\hat{\mathsf{h}}_i\hat{\mathsf{h}}^i C +\frac{1}{2}W^2\mathsf{k}'\psi_2 -\frac{1}{4}\psi_2^2 \\
        &\qquad \leq
        \hat{\mathsf{h}}_i\hat{\mathsf{h}}^i(\frac{1}{4}\psi_2^2-\frac{1}{2}C)+\frac{1}{2}W^2\mathsf{k}'\psi_2 -\frac{1}{4}\psi_2^2 \\
        &\qquad \leq
        W^2(\frac{1}{4}\psi_2^2-\frac{1}{2}C+\frac{1}{2}\mathsf{k}'\psi_2)-\frac{1}{4}\psi_2^2 \leq 2W^2, \\
        &-\frac{1}{2}(\hat{\mathsf{h}}^0)^2 C + \frac{1}{2}W^4v^2(\mathsf{k}')^2-\hat{\mathsf{h}}_i\hat{\mathsf{h}}^i C + W^2\mathsf{k}'\psi_2 \\
        &\qquad \leq
        \frac{1}{2}W^4v^2(\mathsf{k}')^2 + W^2\mathsf{k}'\psi_2
        \leq
        6W^4v^2, \\
        &-\hat{\mathsf{h}}_i\hat{\mathsf{h}}^i(\hat{\mathsf{h}}^0)^2 A - W^2(\hat{\mathsf{h}}^0)^2 B -\frac{1}{2}(\hat{\mathsf{h}}^0)^2 C + \frac{1}{4}W^4v^2(\mathsf{k}')^2 \\
        &\qquad \leq
        W^4v^2(\frac{1}{4}(\mathsf{k}')^2-B)
        \leq
        3W^4v^2.
    \end{align*}
    These bounds can be shown by using our inequalities on $A$,$B$, and $C$, the inequalities for $\cH_i\cH^i$ and $(\cH^0)^2$, and by plotting the quantities as functions of the flux factor $\mathsf{h}$.
    Thus we have
    \begin{align*}
        \|v^jy^k\frac{\p\cK_{ij}}{\p\cH_k}\|^2_2
        \leq&
        W^2\gamma^2(\hat{\mathsf{h}}^0)^2
        +2W^2\beta^2
        +6W^4v^2(\hat{\mathsf{h}}^0\beta)^2 \\
        &
        +W^4v^2\gamma^2
        +(\hat{\mathsf{h}}^0)^2y^2
        +3W^4v^2\gamma^2\\
        \leq&
        W^6(v^2y^2)
        +2W^2(v^2y^2)
        +6W^6v^4(v^2y^2) \\
        &
        +W^6(v^2y^2)
        +W^2(v^2y^2)
        +3W^6(v^2y^2)\\
        \leq&
        14W^6(v^2y^2).
    \end{align*}
    Hence, $\|v^j\frac{\p\cK_{ij}}{\p \cH_k}\|_2 \leq \sqrt{14}W^3v$
\end{proof}

Applying Lemmas \ref{lem:appendix_lem1} and \ref{lem:appendix_lem2} to Eq.~\eqref{eq:norm_vjDeltaKij}, we have
\begin{align}
        \|v^j \Delta \cK_{ij}\|_2^2
        \leq&
        (W^2v\|\Delta \cJ\|_2
        +
        \sqrt{14}W^3v\|\Delta \cH_j\|_2)^2 \nonumber \\
        =&
        W^4v^2\|\Delta\cJ\|_2^2
        +14W^6v^2\|\Delta\cH_j\|_2^2 \nonumber \\
        &+2\sqrt{14}W^5v^2\|\Delta\cJ\|_2\|\Delta\cH_j\|_2.
\end{align}
Then,
\begin{align}
    \begin{Vmatrix}
     v^i\Delta\cH_i\\
     v^j\Delta\cK_{ij}
     \end{Vmatrix}_2^2
     \leq&
     W^4v^2\|\Delta\cJ\|_2^2
        +(14W^6v^2+v^2)\|\Delta\cH_j\|_2^2 \nonumber \\
        &+2\sqrt{14}W^5v^2\|\Delta\cJ\|_2\|\Delta\cH_j\|_2 \nonumber \\
    \leq&
    (W^4+\sqrt{14}W^5)v^2\|\Delta\cJ\|_2^2 \nonumber \\
    &+
    (14W^6+\sqrt{14}W^5+1)v^2\|\Delta\cH_j\|_2^2 \nonumber \\
    \leq&
    (14W^6+\sqrt{14}W^5+1)v^2
    \begin{Vmatrix}
        \Delta\cJ\\
        \Delta\cH_j
    \end{Vmatrix}_2^2.
\end{align}
We let $\widetilde{L}\vcentcolon=v\sqrt{14W^6+\sqrt{14}W^5+1}$. Then, using WolframAlpha, we find that $\widetilde{L}<W$ when ${v<0.221075}$.
Hence the scheme is guaranteed to converge for ${v<0.221075}$.


\bibliography{references}

\end{document}